\documentclass[10pt,leqno, english]{amsart}
\usepackage[colorlinks=true]{hyperref}
\usepackage{srcltx}
\usepackage[english]{babel}
\usepackage{amsmath, amsthm}
\numberwithin{equation}{section}
\usepackage[top=3.5cm, bottom=3.5cm, left=2cm, right=2cm,marginparwidth=1.8cm, marginparsep=0.1cm]{geometry}

\usepackage{amsfonts}
\usepackage{amssymb}
\usepackage{bbm}
\usepackage{mathrsfs}
\usepackage{tikz-cd}
\usepackage{tensor}
\usepackage{mathtools}
\usepackage{MnSymbol}

\newtheorem{proposition}{Proposition}[section]
\newtheorem{theorem}[proposition]{Theorem}
\newtheorem{assumption}[proposition]{Assumption}
\newtheorem{lemma}[proposition]{Lemma}
\newtheorem{corollary}[proposition]{Corollary}
\newtheorem{conjecture}[proposition]{Conjecture}

\newtheorem{question}[proposition]{Question}

\newtheorem{claim}[proposition]{Claim}
\newtheorem*{problem*}{Inverse Sieve Problem}

\theoremstyle{definition}
\newtheorem{definition}[proposition]{Definition}

\newtheorem{notation}[proposition]{Notation}

\theoremstyle{remark}
\newtheorem{remark}[proposition]{Remark}

\newtheorem*{remark*}{Remark}
\newtheorem*{remarks*}{Remarks}

\def\house#1{{%
    \setbox0=\hbox{$#1$}
    \vrule height \dimexpr\ht0+1.4pt width .5pt depth \dp0\relax
    \vrule height \dimexpr\ht0+1.4pt width \dimexpr\wd0+2pt depth \dimexpr-\ht0-1pt\relax
    \llap{$#1$\kern1pt}
    \vrule height \dimexpr\ht0+1.4pt width .5pt depth \dp0\relax}}

\begin{document}

\title[]{Uniform bounds for the number of rational points on varieties over global fields}

\author[M. Paredes, R. Sasyk]{Marcelo Paredes $^{2}$ \MakeLowercase{and} Rom\'an Sasyk $^{1,2}$}

\address{$^{1}$Instituto Argentino de Matem\'aticas Alberto P. Calder\'on-CONICET,
Saavedra 15, Piso 3 (1083), Buenos Aires, Argentina;}

\address{$^{2}$Departamento de Matem\'atica, Facultad de Ciencias Exactas y Naturales, Universidad de Buenos Aires, Argentina.}

\email{\textcolor[rgb]{0.00,0.00,0.84}{mparedes@dm.uba.ar}}
\email{\textcolor[rgb]{0.00,0.00,0.84}{rsasyk@dm.uba.ar}}



\begin{abstract}	
We extend the work of Salberger; Walsh; Castryck, Cluckers, Dittmann and Nguyen; and Vermeulen to prove the uniform dimension growth conjecture of Heath-Brown and Serre for varieties of degree at least $4$ over global fields. As an intermediate step, we generalize the bounds of Bombieri and Pila to curves over global fields and in doing so we improve the $B^{\varepsilon}$ factor by a $\log(B)$ factor. 

\end{abstract}

\maketitle

\section{Introduction}

Let $X$ be a projective variety defined over a global field $K$. A central problem in diophantine geometry is to find bounds for the number of $K$-rational points in $X$ of bounded height, for some adequate height function. When $K=\mathbb{Q}$ and $X$ is a hypersurface, perhaps the first account of such results with great generality is due to Cohen. Specifically, as a consequence of the results in \cite{Cohen} concerning Hilbert's irreducibility theorem, in the appendix of Heath-Brown's article \cite{Heath-Brown83} it is proved that for an absolutely irreducible form $G\in \mathbb{Z}[X_{1},\ldots ,X_{n}]$ of degree $d\geq 2$, it holds
\begin{equation*}
|\{\boldsymbol x=(x_{1},\ldots ,x_{n})\in \mathbb{Z}^{n}:G(\boldsymbol x)=0,\;\text{h.c.f.}(x_{1},\ldots ,x_{n})=1,\;\max_{i}|x_{i}|\leq B\}|\lesssim_{\varepsilon,G}B^{n-\frac{3}{2}+\varepsilon},
\end{equation*}
for any $\varepsilon>0$. Furthermore, in \cite[Page 227]{Heath-Brown83} Heath-Brown posed the question:    
\begin{question}
Let $G$ be an absolutely irreducible form with coefficients in $\mathbb{Z}$ of degree $d\geq 2$ in $n$ variables. Is it true that for every $\varepsilon>0$ it holds
$$|\{\boldsymbol x=(x_{1},\ldots ,x_{n})\in \mathbb{Z}^{n}:G(\boldsymbol x)=0,\max_{i}|x_{i}|\leq B\}|\lesssim_{\varepsilon,G} B^{n-2+\varepsilon}?$$ 
\label{question1}
\end{question}

The results in \cite{Cohen} where later generalized by Serre in \cite{Serre1983} to the context of projective varieties over number fields. Morever, in \cite{Serre1983}, Serre proposed the following variation of Question \ref{question1} (see \cite[Page 178]{Serre}), 
\begin{question}
Let $K$ be a number field of degree $d_{K}$, and let  $X\subseteq \mathbb{P}_{K}^{n}$ be  an integral projective variety which is not a linear variety. Let $H$ be the absolute projective multiplicative height. Is there a constant $c$ such that
$$|\{\boldsymbol x\in X(K):H(\boldsymbol x)\leq B\}|\lesssim_{X} B^{d_{K}\dim(X)}(\log(B))^{c}?$$ 
\label{question2}
\end{question}

By considering the quadric $xy=zw$ in $\mathbb{P}_{K}^{3}$, Serre remarked in \cite[Page 178]{Serre} that the logarithmic factor in Question \ref{question2} can not be dispensed. Moreover, in \cite[Page 27]{Serre2}, he formulated a variant to Question \ref{question2} where the logarithmic factor is replaced by a factor $B^{\varepsilon}$ for all $\varepsilon$. 

The first breakthrough concerning bounds in terms of the degree and dimension of $X$ is due to Bombieri and Pila in \cite{Bombieri0}. In that article they developed the now called determinant method of Bombieri-Pila and proved that the number of integral zeroes of size up to $B$ of an absolutely irreducible polynomial $F(X,Y)\in \mathbb{Z}[X,Y]$ of degree $d\geq 2$ is $\lesssim_{\varepsilon,d}B^{\frac{1}{d}+\varepsilon}$ for all $\varepsilon>0$. This result was used in \cite{Pilaafin} by Pila to prove uniform bounds for projective and affine varieties, for instance, he proved that the number of rational zeroes of height up to $B$ of an absolutely irreducible form $G\in \mathbb{Z}[X_{1},\ldots ,X_{n}]$ of degree $d$ is $\lesssim_{\varepsilon}B^{n-2+\frac{1}{d}+\varepsilon}$.

The determinant method of Bombieri-Pila is of an affine nature. Subsequently, in \cite{Heath-Brown} Heath-Brown developed a $p$-adic determinant method which allowed him to prove general uniform bounds for projective hypersurfaces. In particular, he gave a positive answer to Question \ref{question1} in the case $d=2$ and also he proved the following projective version of the main result in \cite{Bombieri0}. 

\begin{theorem}[{\cite{Heath-Brown}}]
Let $F\in \mathbb{Z}[X_{1},X_{2},X_{3}]$ be an absolutely irreducible form of degree $d$. Then for all $\varepsilon>0$ the number of rational zeroes of height up to $B$ of $F$ is at most $\lesssim_{\varepsilon,d}B^{\frac{2}{d}+\varepsilon}$.
\label{curves Heath-Brown}
\end{theorem}

Moreover, in this article Heath-Brown stated a uniform version of Question \ref{question1} (\cite[Conjecture 1.2]{Heath-Brown}). A more general version of this conjeture for projective varieties over $\mathbb{Q}$ appeared first in the literature in \cite[Conjecture 3.3]{Browningbook} where it is called the ``dimension growth conjeture''.

\begin{conjecture}[Dimension growth conjecture]
Let $X\subseteq \mathbb{P}_{\mathbb{Q}}^{n}$ be an integral projective variety of degree $d\geq 2$. Let $H$ be the absolute projective multiplicative height. Then for any $\varepsilon>0$ it holds:
$$|\{\boldsymbol x\in X(\mathbb{Q}):H(\boldsymbol x)\leq B\}|\lesssim_{\dim(X),d,\varepsilon}B^{\dim(X)+\varepsilon}.$$ 
\label{dim growth}
\end{conjecture}

Conjecture \ref{dim growth} was established in the case $d=2$, and the cases $n=3$ and $n=4$ for any degree, by Heath-Brown in \cite[Theorem 2]{Heath-Brown} and \cite[Theorems 3 and 9]{Heath-Brown}, respectively. For $n=5$ and $d\geq 4$ the conjecture was proved by Broberg and Salberger in \cite[Theorem 1]{BrobSal}, and for $n=5$ and  $d=3$ it was proved by Browning and Heath-Brown in \cite[Theorem 3]{BrowningHeathBrown}. Moreover, Conjecture \ref{dim growth} was proved for varieties of degree $d\geq 6$ for all $n$ by Browning, Heath-Brown and Salberger in \cite[Corollary 2]{Salberger2}.

In order to tackle Conjecture \ref{dim growth} in the cases of lower degrees, Salberger further developed the determinant method. In \cite{Salberger1} he extended the $p$-adic determinant method devised in \cite{Heath-Brown} to prove Conjecture \ref{dim growth} for $d\geq 4$ whenever $X$ contains finitely many linear varieties defined over $\overline{\mathbb{Q}}$, of dimension $r-1$. This result was superseded by those of the article \cite{Salberger}, where Salberger introduced a global version of the method of Heath-Brown and proved the following theorem.

\begin{theorem}[{\cite{Salberger2, Salberger}}]
The dimension growth conjecture holds for every integral projective variety $X\subseteq \mathbb{P}_{\mathbb{Q}}^{n}$ of degree $d\geq 4$. In the case that $X$ has degree $d=3$, it holds
$$|\{\boldsymbol x\in X(\mathbb{Q}):H(\boldsymbol x)\leq B\}|\lesssim_{\dim(X),\varepsilon}B^{\dim(X)-1+\frac{2}{\sqrt{3}}+\varepsilon}.$$ 
\label{dim growth th}
\end{theorem}

Theorem \ref{curves Heath-Brown} and Theorem \ref{dim growth th} (the latter when $d\geq 4$) have exponents that are essentially optimal, up to the $\varepsilon$-factor. Then, it remained an open question if in general one could remove the factor $B^{\varepsilon}$. Building on the global determinant method of Salberger, and using the polynomial method in a clever way, in \cite{Walsh3} Walsh proved that this is indeed the case for curves.

\begin{theorem}[{\cite{Walsh3}}]
Let $F\in \mathbb{Z}[X_{1},X_{2},X_{3}]$ be an absolutely irreducible form of degree $d$. Then the number of rational zeroes of height up to $B$ of $F$ is at most $\lesssim_{d}B^{\frac{2}{d}}$.
\label{Walsh curves}
\end{theorem}

This was further explored in \cite{Cluckers} by Castryck, Cluckers, Dittmann, and Nguyen where they provided effective versions of several results in \cite{Ellenberg, BrowningHeathBrown, Salberger, Walsh3}, which allowed them to prove that the dimension growth conjecture holds without the $\varepsilon$ factor when the degree is at least $5$.

\begin{theorem}[{\cite{Cluckers}}]
Given $n>1$, there exist constants $c=c(n)$ and $e=e(n)$, such that for all integral projective varieties $X\subseteq \mathbb{P}_{\mathbb{Q}}^{n}$ of degree $d\geq 5$ and all $B\geq 1$ one has
$$|\{\boldsymbol x\in X(\mathbb{Q}):H(\boldsymbol x)\leq B\}|\leq cd^{e}B^{\dim(X)}.$$
\label{uniform growth}
\end{theorem}

While there has been a lot of progress in giving uniform bounds for varieties defined over $\mathbb{Q}$, there are not as many results for global fields. In \cite{Broberg}, Broberg generalized several of the estimates in \cite{Heath-Brown} to number fields. Also, the determinant method was reinterpreted in the framework of Arakelov's theory in \cite{Chen}, \cite{Chen2} and \cite{Liu}. For function fields, the Bombieri-Pila bound was proved for $\mathbb{F}_{q}(T)$ in \cite{Sedunova} by Sedunova, adapting a reinterpretation given in \cite{Helfgott} by Helfgott and Venkatesh of the determinant method of Bombieri-Pila. Also, by methods of model theory, in \cite{Cluckers0} Cluckers, Forey and Loeser proved a stronger bound than the one in \cite{Sedunova} for $\mathbb{F}_{q}(T)$ with $\text{char}(\mathbb{F}_{q})$ large enough. Quite recently, in \cite{Vermeulen} Vermeulen proved analogues of Theorem \ref{Walsh curves} and Theorem \ref{uniform growth} for hypersurfaces over $\mathbb{F}_{q}(T)$ of degree $d\geq 64$. 

In this article we extend the work of Salberger, Walsh, and Castryck, Cluckers, Dittmann and Nguyen on the  determinant method of Heath-Brown and Salberger, to give uniform estimates for the number of rational points of bounded height on projective varieties defined over global fields. More  precisely, we first prove the following extension of Theorem \ref{Walsh curves}, \cite[Theorem 2]{Cluckers} and \cite[Theorem 1.1]{Vermeulen} to global fields.

\begin{theorem}
Let $K$ be a global field of degree $d_{K}$. Let $H$ be the absolute projective multiplicative height. For any integral projective curve $C\subseteq \mathbb{P}_{K}^{n}$ of degree $d$ it holds
$$\left|\left\{ \boldsymbol x\in C(K):H(\boldsymbol x)\leq B \right\}\right|\lesssim_{K,n}\begin{cases} d^{4}B^{\frac{2d_{K}}{d}} & \text{ if }K\text{ is a number field},\\ d^{8}B^{\frac{2d_{K}}{d}} & \text{ if }K\text{ is a function field}.\end{cases} $$
\label{teorema1}
\end{theorem}
Theorem \ref{teorema1} is new when $K$ is a global field different from $\mathbb{Q}$ and $\mathbb{F}_{q}(T)$. Previous to this result, in the number field case the only known bound was $O_{K,n,d}(B^{\frac{2d_{K}}{d}+\varepsilon})$ given in \cite[Corollary 1]{Broberg}. For number fields different from $\mathbb{Q}$, our bound was simultaneously obtained by Liu in \cite{Liu2}.

Adapting the strategy devised in \cite[Remark 2.3]{Ellenberg} and developed in \cite[Proposition 4.2.1]{Cluckers}, from Theorem \ref{teorema1} we deduce the following extension of \cite[Theorem 3]{Cluckers} and \cite[Theorem 1.2]{Vermeulen}  to global fields.

\begin{theorem}[Bombieri-Pila type of bound]
Let $K$ be a global field of degree $d_{K}$. For any integral curve $C\subseteq \mathbb{A}_{K}^{n}$ of degree $d$, it holds
$$\left|\left\{\boldsymbol x\in C(K)\cap [B]_{\mathcal{O}_{K}}^{n}\right\}\right|\lesssim_{K,n} \begin{cases} d^{3}B^{\frac{1}{d}}(\log(B)+d) & \text{ if }K\text{ is a number field}, \\ d^{7}B^{\frac{1}{d}}(\log(B)+d) & \text{ if }K \text{ is a function field},\end{cases}$$
where $[B]_{\mathcal{O}_{K}}^{n}$ is defined in Section \ref{subsection 5.2}.
\label{teorema1.5}
\end{theorem}

We remark that in the function field case, we obtain more precise bounds in the exponent of $d$ than the ones presented in Theorem \ref{teorema1} and Theorem \ref{teorema1.5} which depend on the characteristic of the field as in \cite{Vermeulen} (see Definition \ref{multicharacteristic}, Theorem \ref{Walsh general curves} and Theorem \ref{Walsh general affine curves} for the precise statements).

Next, we obtain the following estimate for the number of points of bounded height on affine hypersurfaces, extending \cite[Theorem 0.4]{Salberger}, \cite[Theorem 4]{Cluckers} and \cite[Theorem 4.2]{Vermeulen} to global fields.

\begin{theorem}
Let $K$ be a global field of degree $d_{K}$. Given $n>2$ there exist a constant $e=e(n)$ such that for all polynomials $f\in \mathcal{O}_{K}[Y_{1},\ldots ,Y_{n}]$ of degree $d$,  whose homogeneous part of degree $d$ is absolutely irreducible, it holds
$$\left|\left\{ \boldsymbol x\in \mathcal{Z}(f)\cap [B]_{\mathcal{O}_{K}}^{n} \right\}\right|\lesssim_{K,n}d^{e}B^{n-2},\text{ whenever }d\geq 5.$$
In the case when $d=3,4$, for any $\varepsilon>0$ we have
$$\left|\left\{ \boldsymbol x\in \mathcal{Z}(f)\cap [B]_{\mathcal{O}_{K}}^{n} \right\}\right|\lesssim_{K,n,\varepsilon}\begin{cases} B^{n-2+\varepsilon} & \text{ if }d=4,\\ B^{n-3+\frac{2}{\sqrt{3}}+\varepsilon} & \text{ if }d=3.\end{cases}$$
\label{teorema3}
\end{theorem}
Theorem \ref{teorema3} is new when $K$ is a number field different from $\mathbb{Q}$. When $K$ is a function field, the result is new when $K$ is a function field different from $\mathbb{F}_{q}(T)$ or  $K=\mathbb{F}_{q}(T)$ and $3\leq d< 64$. 

By means of an effective projection argument as in \cite{P2} which relies on the Combinatorial Nullstellensatz \cite{Alon} and on an argument of Mumford found in \cite{MR0282975}, from Theorem \ref{teorema3} we deduce the dimension growth conjecture for global fields. More precisely, we have   

\begin{theorem}[Dimension growth conjecture for global fields]
Let $K$ be a global field of degree $d_{K}$. Given $n>1$, there exists a constant $e=e(n)$ such that for all integral projective varieties $X\subseteq \mathbb{P}_{K}^{n}$ of degree $d$ it holds
$$\left| \left\{ \boldsymbol x\in X(K):H(\boldsymbol x)\leq B\right\} \right|\lesssim_{K,n}d^{e}B^{d_{K}\dim(X)}\text{ whenever }d\geq 5.$$
In the case $d=3,4$, for any $\varepsilon>0$ we have 
$$\left| \left\{ \boldsymbol x\in X(K):H(\boldsymbol x)\leq B\right\} \right|\lesssim_{K,n,\varepsilon}\begin{cases} B^{d_{K}(\dim(X)+\varepsilon)} & \text{ if }d=4,\\ B^{d_{K}(\dim(X)-1+\frac{2}{\sqrt{3}}+\varepsilon)} & \text{ if }d=3.\end{cases}$$
\label{teorema4}
\end{theorem}
Theorem \ref{teorema4} is new for number fields different from $\mathbb{Q}$. When $K$ is a function field, the result is new when $K$ is different from $\mathbb{F}_{q}(T)$ or when $K=\mathbb{F}_{q}(T)$ and $X$ is a hypersurface of degree $3\leq d<64$ or when $K=\mathbb{F}_{q}(T)$ and $X$ is a projective variety of codimension at least $2$ and degree $d\geq 3$. Arguably,   these are the first bounds appearing in the literature estimating the number of points of bounded height for varieties over global fields different from $\mathbb{Q}$ or $\mathbb{F}_{q}(T)$. 

The proof of all these theorems  are obtained by adapting and extending the strategies developed in \cite{Heath-Brown,Salberger1, Salberger2, Salberger, Walsh3, Cluckers} to global fields. Roughly speaking, the proof are obtained by the polynomial method (as it was presented in \cite{Walsh3}), namely by constructing a polynomial $g$ of small degree vanishing on $\{\boldsymbol x\in X(K):H(\boldsymbol x)\leq B\}$, which does not vanishes identically on $X$. Then we study the rational points lying in the irreducible components of $\mathcal{Z}(g)\cap X$. For those irreducible components of high degree we argue as in \cite{Cluckers}, while for those of small degree we rework and simplify the proof of \cite[Main Lemma 3.2]{Salberger}. Carrying out this last step adds difficulties all along the article, which were not present in \cite{Cluckers, Walsh3}. Also new challenges appear from dealing with general global fields that will be explained as they arise. It is relevant to emphasize that this article presents a unified treatment that deals with number fields and function fields simultaneously.

When we were at the final stages of writing the first arXiv version of this manuscript, Liu uploaded to the arXiv the article \cite{Liu2} where he proved Theorem \ref{teorema1} for number fields by reinterpreting \cite{Salberger} in the framework of Arakelov theory and by using ideas of \cite{Walsh3,Cluckers}. 

\subsection*{Acknowledgments} We thank Juan Manuel Menconi for useful discussions. We thank Floris Vermeulen for pointing us a mistake in the first version of the proof of Proposition \ref{induction step n=2} in the case of function fields. We thank Chunhui Liu for useful comments regarding the effectiveness of the estimates in our results. We thank the referee for his/her exhaustive revision of the manuscript, and for pointing us some errors and making several comments that improved the exposition of the article.

\section{Heights and primes in global fields}
\label{section 2}

The purpose of this section is manifold. First we establish a normalization of the absolute values of a global field. We use this to define the height function that will be used in this article and recall some basic properties of it.  Secondly, we prove a proposition that allows us to find affine coordinates of  projective points with controlled affine height. Then we recall the theorems of Bombieri and Vaaler, and Thunder, that give solutions of controlled height of a system of linear equations. Next, we define and analize two notions of heights of polynomials that will be used in this article. Finally we present some estimates regarding the distribution of primes in global fields and we comment on how to make all the bounds in all the statements in this manuscript effective on the dependence of the global field $K$. 

\begin{notation}
We use the asymptotic notation $X=O(Y)$ or $X\lesssim Y$ to mean $|X|\leq C|Y|$ for some constant $C$. We also use $O_{K,n,d}(Y)$ or $\lesssim_{K,n,d}Y$ to mean that the implicit constants depend on $K,n$ and $d$.
\end{notation}

\subsection{Absolute values and relative height}
 Throughout this paper, $K$ denotes a global field, i.e. a finite separable extension of $\mathbb Q$ or $\mathbb{F}_{q}(T)$, in  which case we further assume that the field of constants is $\mathbb{F}_{q}$. We will denote by $d_{K}$ the degree of the extension  $K/ \mathbbm{k}$,  where $\mathbbm{k}$ indistinctively denotes the base fields  $\mathbb{Q}$  or  $\mathbb{F}_{q}(T)$.

Let $K$ be a number field and let $\mathcal{O}_{K}$ be its ring of integers. Then each embedding $\sigma:K\hookrightarrow \mathbb{C}$ induces a place $v$, by means of the equation
\begin{equation}
|x|_{v}:=|\sigma(x)|^{\frac{n_{v}}{d_{K}}}_{\infty},
\label{definition of arq places}
\end{equation}
where $|\cdot |_{\infty}$ denotes the absolute value of $\mathbb{R}$ or $\mathbb{C}$ and $n_{v}=1$ or $2$, respectively. 
Such places will be called the places at infinite, and the set of these places is denoted by  $M_{K,\infty}$. Note that $\sum_{v\in M_{K,\infty}}n_{v}=d_{K}$. They are all the archimedean places of $K$. Since the complex embeddings come in pairs that differ by complex conjugation, we have $|M_{K,\infty}|\leq d_{K}$. 

Now let $\mathfrak{p}$ be a non-zero prime ideal of the number field $K$, and denote by $\text{ord}_{\mathfrak{p}}$ the usual $\mathfrak{p}$-adic valuation. Associated to $\mathfrak{p}$, we have the place $v$ in $K$ given by the equation
$$|x|_{v}:=|x|_{\mathfrak{p}}:=\mathcal{N}_{K}(\mathfrak{p})^{-\frac{\text{ord}_{\mathfrak{p}}(x)}{d_{K}}},$$
where $\mathcal{N}_{K}(\mathfrak{p})$ denotes the cardinal of the finite quotient $\mathcal{O}_{K}/\mathfrak{p}$. Similarly, the norm of a non-zero ideal $I\subseteq \mathcal{O}_{K}$, denoted by $\mathcal{N}_{K}(I)$, is just the cardinal of the finite quotient $\mathcal{O}_{K}/I$. Such places are called the finite places, and the set of these places is denoted by $M_{K,\text{fin}}$. They are all the non-archimedean places of $K$. The set of places of $K$ is then the union $M_{K,\infty}\cup M_{K,\text{fin}}$, and it will be denoted by $M_{K}$.

Now, let us suppose that $K$ is a function field over $\mathbb{F}_{q}$, such that $\mathbb{F}_{q}$ is algebraically closed in $K$ (in other words, the constant field of $K$ is $\mathbb{F}_{q}$). A prime in $K$ is, by definition, a discrete valuation ring $\mathcal{O}_{(\mathfrak{p})}$ with maximal ideal $\mathfrak{p}$ such that $\mathbb{F}_{q}\subseteq \mathcal{O}_{(\mathfrak{p})}$ and the quotient field of $\mathcal{O}_{(\mathfrak{p})}$ equals to $K$. By abuse of notation, when we refer to a prime in $K$, we will refer to the maximal ideal $\mathfrak{p}$. Associated to $\mathfrak{p}$, we have the usual $\mathfrak{p}$-adic valuation, that we will denote by $\text{ord}_{\mathfrak{p}}$. The degree of $\mathfrak{p}$, denoted by $\deg(\mathfrak{p})$ will be the dimension of $\mathcal{O}_{(\mathfrak{p})}/\mathfrak{p}$ as an $\mathbb{F}_{q}$-vector space, which is finite. Then the norm of $\mathfrak{p}$ is defined as $\mathcal{N}_{K}(\mathfrak{p}):=q^{\deg(\mathfrak{p})}$. 
Any prime $\mathfrak{p}$ of $K$ induces a place $v$ in $K$ by the equation
$$|x|_{v}:=|x|_{\mathfrak{p}}:=\mathcal{N}_{K}(\mathfrak{p})^{-\frac{\text{ord}_{\mathfrak{p}}(x)}{d_{K}}}.$$
They are all the places in $K$. The set of all places in $K$ is denoted by $M_{K}$. Now we fix an arbitrary place $v_{\infty}$ in $M_{K}$ above the place in $\mathbb{F}_{q}(T)$ defined by $\left |\frac{f}{g}\right |_{\infty}:=q^{\deg(f)-\deg(g)}$. Its corresponding prime will be denoted $\mathfrak{p}_{\infty}$; it has degree at most $d_{K}$. The ring of integers of $K$ is the subset
$$\mathcal{O}_{K}:=\left\{ x\in K: |x|_{v}\leq 1\text{ for all }v\in M_{K},v\neq  v_{\infty}\right\}.$$
Given $x\in \mathcal{O}_{K}\;\backslash \; \{0\}$ we define $\mathcal{N}_{K}(x):=\prod_{\mathfrak{p}\neq \mathfrak{p}_{\infty}}\mathcal{N}_{K}(\mathfrak{p})^{\text{ord}_{\mathfrak{p}}(x)}$. By definition, $\text{ord}_{\mathfrak{p}}(x)\geq 0$ for all $\mathfrak{p}\neq \mathfrak{p}_{\infty}$, so that $\mathcal{N}_{K}(x)$ is a positive integer. 
The primes in $\mathcal{O}_{K}$ will be the primes $\mathfrak{p}\neq \mathfrak{p}_{\infty}$ of $K$ . We will denote $M_{K,\infty}:=\{v_{\infty}\}$ and $M_{K,\text{fin}}:=M_{K}\, \backslash \, M_{K,\infty}$. 

By our choice of normalization of the absolute values of $K$, it holds the product formula
\begin{equation}
\prod_{v\in M_{K}}|x|_{v}=1\text{ for all }x\in K.
\label{product formula}
\end{equation}
Now, given a global field $K$, we define the absolute multiplicative projective height of $K$ of a point ${\bf x}=(x_{0}:\ldots :x_{n})\in \mathbb{P}^{n}(K)$, to be the function
$$H({\bf x}):=\displaystyle \prod_{v\in M_{K}}\max_{i}\{|x_{i}|_{v}\},$$
and the relative multiplicative projective height by
$$H_{K}({\bf x}):=H({\bf x})^{d_{K}}.$$
If $x\in K$, $H_{K}(x)$ will always denote the projective height $H_{K}(1:x)$. 

For our purposes, it will be necessary to understand how the affine height of a point behaves under the action of a polynomial. It is easy to show (see \cite[Proposition B.2.5]{Hindry}) that if $f(T_{1},\ldots ,T_{n})=\sum_{(i_{1},\ldots ,i_{n})}c_{i_{1},\ldots ,i_{n}}T_{1}^{i_{1}}\cdots T_{n}^{i_{n}}$, ${\bf c}=(c_{i_{1},\ldots ,i_{n}})_{i_{1},\ldots ,i_{n}}$ and $R$ is the number of $(i_{1},\ldots, i_{n})$ with $c_{i_{1},\ldots ,i_{n}}\neq 0$, then
\begin{equation}
H_{K}(f({\bf x}))\leq R^{d_{K}}H_{K}(1:{\bf c})H_{K}(1:{\bf x})^{\deg(f)}.
\label{polynomialheight2}
\end{equation}
Now, let $\phi_{0},\ldots ,\phi_{r}$ be homogeneous polynomials of degree $m$ in $X_{0},\ldots ,X_{n}$ with coefficients in $K$. Let $\boldsymbol c$ be the projective point consisting of the coefficients of all the polynomials $\phi_{i}$. Let us suppose that the maximum number of monomials appearing in any one of the $\phi_{i}$ is $R$.  Then for all $\boldsymbol x\in \mathbb{P}^{n}(K)$ such that $\phi_{i}(\boldsymbol x)$ is not zero for all $i$, it holds
\begin{equation}
H_{K}(\phi_{0}(\boldsymbol x):\ldots :\phi_{r}(\boldsymbol x))\leq R^{d_{K}}H_{K}(\boldsymbol c)H_{K}(\boldsymbol x)^{m}.
\label{polynomialheight}
\end{equation}  
Also, for any $x\in \mathcal{O}_{K}\backslash \{0\}$, it holds
\begin{equation}
\mathcal{N}_{K}(x)\leq H_{K}(x).
\label{norm}
\end{equation}

We will require to lift a bounded subset in projective space to a subset in affine space. The next proposition, which is a generalization of \cite[Section 13.4]{Serre}, states that this can be done in a controlled manner.

\begin{proposition}
Let $K$ be a global field, and let $d\geq 1$ be an integer. There exist effective computable    constants $c_{1}:=c_{1}(K)$, $c_{2}:=c_{2}(K)$ and $c_{3}:=c_{3}(K)$ such that for every $\boldsymbol x\in \mathbb{P}^{d}(K)$ there exists $(y_{0},\ldots ,y_{d})\in \mathcal{O}_{K}^{d+1}$ a lift of $\boldsymbol x$ verifying:
\begin{enumerate}
\item if $K$ is a number field, for all embedding $\sigma:K\hookrightarrow \mathbb{C}$, $\max_{i}|\sigma(y_{i})|\leq c_{1}H(\boldsymbol x)$, thus for all $v\in M_{K,\infty}$, $\max_{i}|y_{i}|_{v}\leq c_{1} H_{K}(\boldsymbol x)^{\frac{n_{v}}{d_{K}^{2}}}$ after relabeling $c_{1}$. If $K$ is a function field, for the place $v_{\infty}$ it holds $\max_{i}|y_{i}|_{v_{\infty}}\leq c_{1}H_{K}(\boldsymbol x)^{\frac{1}{d_{K}}}$;  
\item for any prime $\mathfrak{p}\notin M_{K,\infty}$ with $\mathcal{N}_{K}(\mathfrak{p})> c_{2}$, it holds that $\mathfrak{p}\not |\sum_{i=0}^{d}y_{i}\mathcal{O}_{K}$. Equivalently, for any place $v$ corresponding to such primes it holds $\max_{i}|y_{i}|_{v}=1$; 
\item it holds
$$\prod_{v\in M_{K,\emph{\text{fin}}}}\max_{i}|y_{i}|_{v}\geq c_{3}.$$
\end{enumerate} 
\label{Serre}
\end{proposition}

\begin{proof}
Given $\boldsymbol x$ as in the statement, we choose cordinates $(x_{0},\ldots ,x_{d})$ with $x_{i}\in \mathcal{O}_{K}$ for all $i$ and consider the ideal $\mathfrak{a}_{x_{0},\ldots ,x_{d}}=\sum_{i=0}^{d}x_{i}\mathcal{O}_{K}$. Note that
$$H(\boldsymbol x)=\frac{1}{\mathcal{N}_{K}(\mathfrak{a}_{x_{0},\ldots ,x_{d}})^{\frac{1}{d_{K}}}}\prod_{v\in M_{K,\infty}}\max_{i}|x_{i}|_{v}.$$
The ideal $\mathfrak{a}_{x_{0},\ldots ,x_{d}}$ depends on the coordinates, but its ideal class depends only on $\boldsymbol x$. Hence, if we take integral ideals $\mathfrak{a}_{1},\ldots \mathfrak{a}_{r}$, representing all the ideal classes of $\mathcal{O}_{K}$, satisfying Minkowski's bound $\mathcal{N}_{K}(\mathfrak{a}_{l})\leq c_{2}$ for all $l$, it holds that $\mathfrak{a}_{x_{0},\ldots ,x_{d}}\mathfrak{a}_{l}^{-1}=\alpha \mathcal{O}_{K}$ for some $l$ and some $\alpha\in K^{\times}$. Thus, $\alpha^{-1}\mathcal{O}_{K}\mathfrak{a}_{x_{0},\ldots ,x_{d}}=\mathfrak{a}_{l}$. We conclude that $(x_{0}',\ldots ,x_{d}'):=(\alpha^{-1}x_{0},\ldots ,\alpha^{-1}x_{d})$ are coordinates of $\boldsymbol x$ in $\mathcal{O}_{K}^{d+1}$ and $\mathfrak{a}_{x_{0}',\ldots ,x_{d}'}=\mathfrak{a}_{l}$. In particular, since $\mathcal{N}_{K}(\mathfrak{a}_{l})\leq c_{2}$, all the prime ideals $\mathfrak{p}$ dividing $\mathfrak{a}_{l}$ have norm $\mathcal{N}_{K}(\mathfrak{p})\leq c_{2}$. We conclude that 
\begin{equation*}
\text{for all prime } \mathfrak{p}\notin M_{K,\infty}\text{ with } \mathcal{N}_{K}(\mathfrak{p})>c_{2},\; \mathfrak{p}\not | \sum_{i=0}^{d}x_{i}'\mathcal{O}_{K}.
\end{equation*}
Moreover, if $c_{3}=c_{2}^{-\frac{1}{d_{K}}}$, we see that $\prod_{v\in M_{K,\text{fin}}}\max_{i}|x_{i}'|_{v}=\frac{1}{\mathcal{N}_{K}(\mathfrak{a}_{l})^{\frac{1}{d_{K}}}}\geq c_{3}$. 

If $K$ is a function field, $M_{K,\infty}$ consists only of the place $v_{\infty}$, thus condition (1) of Proposition \ref{Serre} is immediately verified with $c_{1}:=c_{3}^{-1}$. If $K$ is a number field, 
the argument follows the same lines given in \cite[Section 13.4]{Serre}, \cite[Proposition 2.1]{P2}, namely we find $c_{1}$ and $\varepsilon\in \mathcal{O}_{K}^{\times}$ such that the set of coordinates $(y_{0},\ldots ,y_{d}):=(\varepsilon^{-1}x_{0}',\ldots ,\varepsilon^{-1}x_{d}')$ is a point in $\mathcal{O}_{K}^{d+1}$, representing $\boldsymbol x$, satisfying for all embedding $\sigma:K\hookrightarrow \mathbb{C}$ the bound  $\max_{i}|\sigma(y_{i})|\leq c_{1} H_{K}(\boldsymbol x)$. More precisely, let $\sigma_{1},\ldots ,\sigma_{r},\sigma_{r+1},\ldots ,\sigma_{r+s}$ be all the non-equivalent embeddings of $K$, rearranged in such a way that $\sigma_{i}$ is a real embedding for all $1\leq i\leq r$, and $\sigma_{i}$ is a complex embedding for all $i>r$. Since the $\mathbb{Z}$-span 
$$W:=\left\langle \left \{ \Big(\log(|\sigma_{1}(\varepsilon)|),\ldots ,\log(|\sigma_{r}(\varepsilon)|),\log(|\sigma_{r+1}(\varepsilon)|^{2}),\ldots ,\log(|\sigma_{r+s}(\varepsilon)|^{2})\Big):\varepsilon\in \mathcal{O}_{K}^{\ast}\right  \}\right \rangle_{\mathbb{Z}}$$ 
has dimension $r+s-1$ by Dirichlet's unit theorem, so the same holds for the $\mathbb{Z}$-span
$$W':=\left\langle \left \{ \Big(\log(|\sigma_{1}(\varepsilon)|),\ldots ,\log(|\sigma_{r}(\varepsilon)|),\log(|\sigma_{r+1}(\varepsilon)|),\ldots ,\log(|\sigma_{r+s}(\varepsilon)|)\Big):\varepsilon\in \mathcal{O}_{K}^{\ast}\right  \}\right \rangle_{\mathbb{Z}}.$$ 
The vector $\boldsymbol 1=(1,\ldots ,1)\in \mathbb{R}^{r+s}$ does not lie in $W'$ because $e$ is transcendental. Then, we have that $W_{1}:=W\oplus \mathbb{Z}\boldsymbol 1$ is a lattice in $\mathbb{R}^{r-s}$. Let $||\cdot ||$ be the $\ell^{\infty}$-norm in $\mathbb{R}^{r-s}$. Let $\Omega\subseteq \mathbb{R}^{r-s}$ be a bounded subset containing a representative of each class of $\mathbb{R}^{r+s}/W_{1}$ (this subset exists since $\mathbb{R}^{r+s}/W_{1}$ is compact) and set $C_{W_{1}}:=\sup_{\boldsymbol y\in \Omega}||\boldsymbol y||$. Then it follows that 
\begin{equation}
\text{for any } \boldsymbol z\in \mathbb{R}^{r+s} \text{ there exists } \boldsymbol w\in W_{1} \text{ verifying } ||\boldsymbol z-\boldsymbol w||\leq C_{W_{1}}.
\label{fundamental domain}
\end{equation}
Now, for any $\boldsymbol x\in \mathbb{P}^{d}(K)$ with coordinates $(x_{0},\ldots ,x_{d})$ such that $\mathfrak{a}_{x_{0},\ldots, x_{d}}=\mathfrak{a}_{l}$ for some $l$, we have
$$H(\boldsymbol x)=\frac{1}{\mathcal{N}_{K}(\mathfrak{a}_{l})^{\frac{1}{d_{K}}}}\prod_{j}\max_{i}|\sigma_{j}(x_{i})|^{\frac{n_{j}}{d_{K}}}\sim_{K}\prod_{j}\max_{i}|\sigma_{j}(x_{i})|^{\frac{n_{j}}{d_{K}}},$$
where $n_{j}=1$ if $\sigma_{j}$ is real and $n_{j}=2$ if it is complex. Let $h_{\sigma}:=\max_{i}|\sigma(x_{i})|>0$. Let $\boldsymbol z:=(\log(h_{\sigma_{1}}),\ldots ,\log(h_{\sigma_{r+s}}))$. By \eqref{fundamental domain} there exists $\boldsymbol w\in W_{1}$ such that $||\boldsymbol z-\boldsymbol w||\leq C_{W_{1}}$. In particular, there exist $\varepsilon\in \mathcal{O}_{K}^{\ast}$ and $m\in \mathbb{Z}$ such that $\boldsymbol w=(\log|\sigma_{1}(\varepsilon)|,\ldots ,\log|\sigma_{r+s}(\varepsilon)|)+m\boldsymbol 1$ and
$$-C_{W_{1}}\leq \log|\sigma_{j}(\varepsilon)|+m-\log(h_{\sigma_{j}})\leq C_{W_{1}} \text{ for all }j.$$
Hence, if $C_{K}:=e^{C_{W_{1}}}, $we have that there is some $\varepsilon\in \mathcal{O}_{K}^{\ast}$ and $t=e^{m}$ such that
\begin{equation*}
C_{K}^{-1}\frac{h_{\sigma_{j}}}{t}\leq |\sigma_{j}(\varepsilon)|\leq C_{K}\frac{h_{\sigma_{j}}}{t} \text{ for all }j.
\end{equation*}
Equivalently,
\begin{equation*}
C_{K}^{-1}\max_{i}|\sigma_{j}(x_{i}\varepsilon^{-1})|\leq t\leq C_{K}\max_{i}|\sigma_{j}(x_{i}\varepsilon^{-1})|\text{ for all }j.
\end{equation*}
Thus, since $\sum_{i}\frac{n_{i}}{d_{K}}=1$, for any $j_{0}$ it follows that
\begin{align*}
 H(\boldsymbol x) & =H(\varepsilon^{-1}\boldsymbol x)\sim_{K} \prod_{j}\max_{i}|\sigma_{j}(x_{i}\varepsilon^{-1})|^{\frac{n_{j}}{d_{K}}}=\prod_{j}\left(\max_{i}|\sigma_{j}(x_{i}\varepsilon^{-1})|C_{K}\right)^{\frac{n_{j}}{d_{K}}}C_{K}^{-\frac{n_{j}}{d_{K}}}\\ & \gtrsim_{K} \prod_{j} t^{\frac{n_{j}}{d_{K}}}\gtrsim_{K} \prod_{j}\left(\max_{i}|\sigma_{j_{0}}(x_{i}\varepsilon^{-1})|C_{K}^{-1}\right)^{\frac{n_{j}}{d_{K}}}\gtrsim_{K} \max_{i}|\sigma_{j_{0}}(x_{i}\varepsilon^{-1})|. 
\end{align*}
We conclude that $H(\boldsymbol x)\gtrsim_{K} \max_{j,i}|\sigma_{j}(x_{i}\varepsilon^{-1})|\gtrsim_{K}\max_{\sigma}|\sigma(x_{i}\varepsilon^{-1})|\text{ for all }i.$ Then $\boldsymbol y=(x_{1}\varepsilon^{-1},\ldots ,x_{d}\varepsilon^{-1})$ is a lift of $\boldsymbol x$ with the required properties.
\end{proof}

\subsection{Small solutions to linear equations} For our purpose, we will need a strong version of Siegel's lemma, given by the well known theorem \cite{Bombieri2} of Bombieri and Vaaler, and its function field generalization \cite{Thunder} of Thunder. In order to state these results, we recall the notion of the height of a matrix. 
Let $A=(a_{ij})_{i,j}$ be an $m\times n$ matrix with entries in $K$ and rank $m<n$. If $J\subseteq \{1,2,\ldots, n\}$ is a subset with $|J|=m$ elements, we write $A_{J}:=(a_{i,j})_{1\leq i\leq m,j\in J}$ for the corresponding submatrix. For each place $v\in M_{K}$ we define
\begin{equation*}
H_{v}(A):=\begin{cases} \max_{|J|=m}|\det(A_{J})|_{v} & \text{ if }v\text{ is non-archimedean}, \\ |\det(A A^{\ast})|_{v}^{\frac{1}{2}} & \text{ if }v\text{ is archimedean}, \end{cases}
\end{equation*}
where $A^{\ast}$ is the complex conjugate transpose of $A$. The Arakelov height of $A$ is defined as the product
$$H_{\text{Ar}}(A):=\prod_{v\in M_{K}}H_{v}(A).$$
Let us state the theorems of Bombieri and Vaaler, and Thunder.

\begin{theorem}[{\cite[Theorem 9]{Bombieri2}, \cite[Corollary 2]{Thunder}}]
Let $K$ be a global field. Suppose that $A=(a_{ij})$ is an $m\times n$ matrix of rank $m<n$ with entries in $K$. Then there exist an effective constant $C=C(K)$ and linearly independent $\boldsymbol  b_{1},\ldots \boldsymbol b_{n-m}\in K^{n}$ with $A\boldsymbol b_{i}^{t}=0$ for all $i$, satisfying 
$$\prod_{i=1}^{n-m}H(\boldsymbol b_{i})\leq C(K)^{n-m}H_{\emph{\text{Ar}}}(A).$$
\label{Bombieri-Vaaler}
\end{theorem}

In this paper we will use a different formulation of Theorem \ref{Bombieri-Vaaler}, requiring a more explicit definition of the Arakelov height of a matrix. For that, let $A$ be an $m\times n$ matrix of rank $m<n$ with entries in $\mathcal{O}_{K}$. Let $\Delta$ be the greatest common divisor in $\mathcal{O}_{K}$ of the determinants $\det(A_{J})$ with $|J|=m$. Since for any non-archimedean place $v\notin M_{K,\infty}$ with corresponding prime $\mathfrak{p}$, it holds
$$H_{v}(A)=\max_{|J|=m}\mathcal{N}_{K}(\mathfrak{p})^{-\frac{\text{ord}_{\mathfrak{p}}(\det(A_{J}))}{d_{K}}}=\mathcal{N}_{K}(\mathfrak{p})^{-\min_{|J|=m}\frac{\text{ord}_{\mathfrak{p}}(\det(A_{J}))}{d_{K}}}=\mathcal{N}_{K}(\mathfrak{p})^{-\frac{\text{ord}_{\mathfrak{p}}(\Delta)}{d_{K}}},$$  
we conclude that
$$\prod_{v\notin M_{K,\infty}}H_{v}(A)=\prod_{\mathfrak{p}|\Delta}\mathcal{N}_{K}(\mathfrak{p})^{-\frac{\text{ord}_{\mathfrak{p}}(\Delta)}{d_{K}}}=\mathcal{N}_{K}(\Delta)^{-\frac{1}{d_{K}}}.$$
Hence, we may express the Arakelov height of a matrix as the product:
\begin{equation*}
H_{\text{Ar}}(A)=\begin{cases} \mathcal{N}_{K}(\Delta)^{-\frac{1}{d_{K}}} \displaystyle \prod_{v\in M_{K,\infty}}|\det(A A^{\ast})|_{v}^{\frac{1}{2}} & \text{ if }K\text{ is a number field},\\ \mathcal{N}_{K}(\Delta)^{-\frac{1}{d_{K}}} q^{-\frac{\deg(\mathfrak{p}_{\infty})}{d_{K}}\min_{|J|=m}\text{ord}_{\mathfrak{p}_{\infty}}(\det(A_{J}))} & \text{ if }K \text{ is a function field}. \end{cases}
\end{equation*}
Using this, we conclude the following implication of Theorem \ref{Bombieri-Vaaler}.

\begin{theorem}
Let $K$ be a global field of degree $d_{K}$. Suppose that $A=(a_{ij})$ is an $m\times n$ matrix of rank $m<n$ with entries in $K$. Then there exist an effective constant $C=C(K)$ and a non-zero solution $\boldsymbol b=(b_{1},\ldots ,b_{n})\in \mathcal{O}_{K}^{n}$ with $A\boldsymbol b^{t}=0$, satisfying
$$H_{K}(b_{1}:\ldots :b_{n})^{n-m}\leq \begin{cases}C(K)^{(n-m)}\mathcal{N}_{K}(\Delta)^{-1} \displaystyle \prod_{v\in M_{K,\infty}}|\det(A A^{\ast})|_{v}^{\frac{d_{K}}{2}} & \text{ if }K\text{ is a number field},\\ C(K)^{(n-m)}\mathcal{N}_{K}(\Delta)^{-1} q^{-\deg(\mathfrak{p}_{\infty})\min_{|J|=m}\text{ord}_{\mathfrak{p}_{\infty}}(\det(A_{J}))} & \text{ if }K\text{ is a function field}.\end{cases}$$
\label{Bombieri-Vaaler2}
\end{theorem}

Another useful consequence of Theorem \ref{Bombieri-Vaaler} (which is slightly stronger than Siegel's lemma) is the following result. Given an $m\times n$ matrix $A$ with entries in $K$, we denote by $H_{K}(A)$ the $K$-relative height of the point in $\mathbb{P}_{K}^{mn-1}$ corresponding to the matrix $A$. 

\begin{corollary}
Let $K$ be a global field. Suppose that $A=(a_{i,j})$ is an $m\times n$ matrix of rank $r$ with entries in $K$. Then there exists an effective constant $C=C(K)$ and a non-zero solution  $\boldsymbol x\in \mathcal{O}_{K}^{n}$ of $A\boldsymbol x=0$ with
$$H_{K}(1:\boldsymbol x)\leq C(K)\left( n^{\frac{d_{K}}{2}}H_{K}(A)\right)^{\frac{r}{n-r}}.$$ 
\label{Bombieri-Vaaler a la Siegel}
\end{corollary}
\begin{proof}
When $K$ is a number field, in \cite[Corollary 2.9.9]{Bombieri} it is proved that Theorem \ref{Bombieri-Vaaler} implies the existence of a non-zero $\boldsymbol x\in \mathcal{O}_{K}^{n}$ of $A\boldsymbol x=0$ with
\begin{equation}
H_{K}(\boldsymbol x)\leq C(K)\left( n^{\frac{d_{K}}{2}}H_{K}(A)\right)^{\frac{r}{n-r}}.
\label{aa}
\end{equation}
By the same argument, if $K$ is a function field one may find a non-zero solution $\boldsymbol x\in \mathcal{O}_{K}^{n}$ verifying \eqref{aa}. Then a solution verifying the statement of the lemma can be found by lifting $\boldsymbol x$ with Proposition \ref{Serre}.
\end{proof}

\subsection{Polynomial heights} 
\label{different polynomial heights}
We will use a notion of height for a polynomial that is defined in \cite[$\S$ 1.6]{Bombieri}. Specifically, given a global field $K$ of degree $d_{K}$, and $f=\sum_{I}a_{I}\boldsymbol X^{I}\in K[X_{1},\ldots ,X_{n}]$ a non-zero polynomial, for any $v\in M_{K}$ we let
$$|f|_{v}:=\max_{I} |a_{I}|_{v}.$$
Also, for $f=\sum_{I}a_{I}\boldsymbol X^{I}\in \mathbb{C}[X_{1},\ldots ,X_{n}]$, we define $\ell_{\infty}(f):=\max_{I}|a_{I}|_{\infty}$ and $\ell_{1}(f):=\sum_{I}|a_{I}|_{\infty}$, where $|\cdot |_{\infty}$ is the usual absolute value of $\mathbb{C}$. 

Following \cite[$\S$ 1.6]{Bombieri}, the height of $f\in K[X_{1},\ldots ,X_{n}]$ is defined as
\begin{equation*}
H(f):=\prod_{v\in M_{K}}|f|_{v},
\end{equation*}
thus, the height of $f$ is just the projective height of the projective point defined by the coefficients of $f$. In particular, for any $\lambda \in K$ we have $H(\lambda f)=H(f)$.

Similarly, in this article we define the affine height of $f$ by
\begin{equation}
H_{\text{aff}}(f):=\prod_{v\in M_{K}}\max\{1,|f|_{v}\}.
\label{affine height of a polynomial}
\end{equation}
 It holds that for any non-zero polynomial $f$, $H_{\text{aff}}(f)\geq H(f)\geq 1$. 

We will also use the $K$-relative versions of the projective and affine heights of polynomials:
\begin{equation*}
H_{K}(f):=H(f)^{d_{K}}=\left(\prod_{v\in M_{K}}|f|_{v}\right)^{d_{K}},\; H_{K,\text{aff}}(f):=H_{\text{aff}}(f)^{d_{K}}=\left(\prod_{v\in M_{K}}\max\{1,|f|_{v}\}\right)^{d_{K}}.
\end{equation*}

We will need the following easy variant of  the inequality \eqref{polynomialheight2}. First, let us suppose that $K$ is a number field. Let $f=\sum_{I}a_{I}\boldsymbol X^{I}\in \mathbb{Z}[X_{1},\ldots ,X_{n}]$ be a non-zero polynomial, and let $\boldsymbol x=(x_{1},\ldots ,x_{n})\in K^{n}$. For any place $v\in M_{K,\infty}$, it holds 
$$|f(\boldsymbol x)|_{v}\leq \sum_{I}|a_{I}|_{v}|\boldsymbol x^{I}|_{v}\leq \sum_{I}|a_{I}|_{v}\max_{i}\{1,|x_{i}|_{v}\}^{\deg(f)}.$$
Using that the coefficients $a_{I}$ are integers, for any $v\notin M_{K,\infty}$, we have
$$|f(\boldsymbol x)|_{v}\leq \max_{I}\{|a_{I}|_{v}|\boldsymbol x^{I}|_{v}\}\leq \max_{i}\{1,|x_{i}|_{v}\}^{\deg(f)}.$$
It holds
\begin{align}
H(f(\boldsymbol x)) & =\prod_{v\in M_{K}}\max\{1,|f(\boldsymbol x)|_{v}\}  \leq \left(\prod_{v\in M_{K,\infty}}\sum_{I}|a_{I}|^{\frac{n_{v}}{d_{K}}}\right)H(1:\boldsymbol x)^{\deg(f)}\leq \left[\prod_{v\in M_{K,\infty}}\left(\sum_{I}|a_{I}|\right)\right]H(1:\boldsymbol x)^{\deg(f)} \label{inequality between norms} \\ & \leq \ell_{1}(f)^{d_{K}}H(1:\boldsymbol x)^{\deg(f)}, \nonumber
\end{align}
where we used that $|a_{I}|^{\frac{n_{v}}{d_{K}}}\leq |a_{I}|$ since $n_{v}\leq d_{K}$ and $|a_{I}|\geq 1$. Hence
\begin{equation}
H_{K}(f(\boldsymbol x))\leq \ell_{1}(f)^{d_{K}^{2}}H_{K}(1:\boldsymbol x)^{\deg(f)},\; \text{ for all } f\in \mathbb{Z}[X_{1},\ldots ,X_{n}].
\label{auxiliar remark}
\end{equation}

Now, suppose that $K$ is a function field over $\mathbb{F}_{q}(T)$. Let $f=\sum_{I}a_{I}\boldsymbol X^{I}\in \mathbb{F}_{q}[X_{1},\ldots ,X_{n}]$ be a non-zero polynomial, and let $\boldsymbol x=(x_{1},\ldots ,x_{n})\in K^{n}$. Since all places in $K$ are non-archimedean, and trivial on $\mathbb{F}_{q}$, for any $v\in M_{K}$ we have
$$|f(\boldsymbol x)|_{v}\leq \max\{|a_{I}|_{v}|\boldsymbol x^{I}|_{v}\}\leq \max_{i}\{1,|x_{i}|_{v}\}^{\deg(f)},$$  
from where it follows immediately the bound
\begin{equation}
H_{K}(f(\boldsymbol x))\leq H_{K}(1:\boldsymbol x)^{\deg(f)},\; \text{ for } f\in \mathbb{F}_{q}[X_{1},\ldots, X_{n}].
\label{inequality between norms2}
\end{equation}

\subsection{Distribution of primes in global fields}\label{distribution of primes in global fields}
Let $K$ be a global field of degree $d_{K}$. For any $Q>0$, let us denote 

$$\mathcal{P}:=\{\mathfrak{p}\text{ prime in }\mathcal{O}_{K}\},\;\; \mathcal{P}(Q):=\{\mathfrak{p}\in \mathcal{P}:\mathcal{N}_{K}(\mathfrak{p})\leq Q\}.$$
There exist constants 
$c_{1,K}, c_{2,K} $ and $c_{3,K}$ such that

\begin{align}
\sum_{\mathfrak{p}\in \mathcal{P}(Q)}\frac{\log(\mathcal{N}_{K}(\mathfrak{p}))}{\mathcal{N}_{K}(\mathfrak{p})}  =\log(Q)+O_{K}(1), \label{landau1} \\  c_{1,K}Q  \leq \displaystyle \sum_{\mathfrak{p}\in \mathcal{P}(Q)}\log(\mathcal{N}_{K}(\mathfrak{p}))\leq c_{2,K}Q, \label{landau2} \\  \frac{1}{2}Q \leq \smashoperator[r]{\sum_{\substack{\mathfrak{p}\\ Q  \leq \mathcal{N}_{K}(\mathfrak{p})\leq 2Q}}}\log(\mathcal{N}_{K}(\mathfrak{p}))  \text{ for all }Q\geq c_{3,K},\label{Bertrand}
\end{align} 
where the logarithms are in base $q$ and $Q$ is of the form $q^{h}$ with $h\in \mathbb{N}$ when $K$ is a function field. Indeed, if $K$ is a number field, \eqref{landau1}, \eqref{landau2} and \eqref{Bertrand} follow from Landau Prime Ideal Theorem (see \cite[Theorem 5.33]{Iwaniec}) which states that $|\mathcal{P}(Q)|\sim \frac{Q}{\log(Q)}$. Meanwhile, if $K$ is a function field over $\mathbb{F}_{q}$ of genus $g$, this follows from the Riemann Hypothesis over function fields (see \cite[Theorem 5.12]{Rosen}), which states that the number of primes in $K$ of degree $N$ is $\frac{q^{N}}{N}+O_{K}(\frac{q^{\frac{N}{2}}}{N})$. 

By summation by parts, there exists a constant $c_{4,K}$ such that for all $Q>0$ it holds
\begin{equation}
\sum_{\mathfrak{p}\in \mathcal{P}\backslash \mathcal{P}(Q)}\frac{\log(\mathcal{N}_{K}(\mathfrak{p}))}{\mathcal{N}_{K}(\mathfrak{p})^{\frac{3}{2}}}\leq c_{4,K}Q^{-\frac{1}{2}}.
\label{landau3}
\end{equation}

We will use the following lemma, which is similar to \cite[Lemma 1.10]{Salberger}. We provide a simple proof of it, which holds for global fields.
\begin{lemma}
Let $\mathfrak{I}\neq (1)$ be a non-zero ideal. Then
$$\sum_{\mathfrak{p}|\mathfrak{I}}\frac{\log(\mathcal{N}_{K}(\mathfrak{p}))}{\mathcal{N}_{K}(\mathfrak{p})}\leq \max\{\log(\log(\mathcal{N}_{K}(\mathfrak{I}))),0\}+O_{K}(1).$$
\label{divisors of a fixed ideal}
\end{lemma}
\begin{proof}
Let $c>0$. By \eqref{landau1} we have
\begin{align*}
\sum_{\mathfrak{p}|\mathfrak{I}}\frac{\log(\mathcal{N}_{K}(\mathfrak{p}))}{\mathcal{N}_{K}(\mathfrak{p})}  \leq \sum_{\substack{\mathfrak{p}|\mathfrak{I}\\ \mathcal{N}_{K}(\mathfrak{p})\leq c}}\frac{\log(\mathcal{N}_{K}(\mathfrak{p}))}{\mathcal{N}_{K}(\mathfrak{p})}+\sum_{\substack{\mathfrak{p}|\mathfrak{I}\\ \mathcal{N}_{K}(\mathfrak{p})>c}}\frac{\log(\mathcal{N}_{K}(\mathfrak{p}))}{c} \leq \max\{\log(c),0\}+O_{K}(1)+\frac{\log(\mathcal{N}_{K}(\mathfrak{I}))}{c}.
\end{align*}
The lemma follows by taking $c=\log(\mathcal{N}_{K}(\mathfrak{I}))$.
\end{proof}

\begin{remark}
Let $K$ be a function field of degree $d_{K}$. By inspecting the proof of \cite[Theorem 5.12]{Rosen} one sees that $|\{\mathfrak{p}\text{ prime in }K:\deg(\mathfrak{p})=N\}|=\frac{q^{N}}{N}+O_{g}(\frac{q^{\frac{N}{2}}}{N})$ is effective, with the implict constant depending only on the genus of the function field $K$. Since the prime $v_{\infty}$ has degree at most $d_{K}$, it follows that all the bounds in $\S$ \ref{distribution of primes in global fields} for function fields are effective, with constants depending only on $q,d_{K}$ and the genus of $K$.

On the other hand, if $K$ is a number field of degree $d_{K}$ and discriminant $\Delta_{K}$, by \cite[Theorem 5.33]{Iwaniec} if one assumes the Generalized Riemman Hypothesis there exists an effective absolute constant $c>0$ such that it holds $\sum_{\mathfrak{a}:\mathcal{N}_{K}(\mathfrak{a})\leq Q}\Lambda_{K}(\mathfrak{a})=Q+O(\sqrt{|\Delta_{K}|}Q\exp(-cd_{K}^{2}\sqrt{\log(Q)}))$, where $\Lambda_{K}(\mathfrak{a})$ is the usual von Mangoldt function on ideals of $\mathcal{O}_{K}$. Then, all the bounds in $\S$ \ref{distribution of primes in global fields} for number fields would be effective, with constants depending only on $d_{K},\Delta_{K}$ (see also \cite{Grenie} for  explicits versions of Landau Prime Ideal Theorem under the Generalized Riemann Hypothesis).
\end{remark}

\section{The local and the global determinant method over global fields}
\label{section3}

In this section we start by recalling some uniform estimates for Hilbert functions given by Salberger in \cite{Salberger1}. Then, in Lemma \ref{local determinant} we generalize the local determinant method \cite[Main Lemma 2.5]{Salberger1} as it is presented in \cite[Corollary 2.5]{Cluckers} to varieties over global fields. Afterwards, our next main result is Theorem \ref{global determinant method lower bound}, which is a generalization of \cite[Proposition 3.2.26]{Cluckers} where the authors made explicit the dependence on the degree in the bounds of \cite[Theorem 2.3]{Walsh3} when $K=\mathbb{Q}$. Finally, in Lemma \ref{b(f)}, we bound the contribution of the primes $\mathfrak{p}$ for which the reduction modulo $\mathfrak{p}$ of a homogeneous absolutely irreducible polynomial $f\in \mathcal{O}_{K}[X_{0},\ldots ,X_{n+1}]$ is not absolutely irreducible. This generalizes  \cite[inequality (2.1)]{Walsh3}, \cite[Corollary 3.2.3]{Cluckers} and \cite[Lemma 2.3]{Vermeulen} to global fields. In order to do that, we use the theory of effective Noether forms as in \cite{Cluckers, Vermeulen}, forcing us to distinguish the characteristic of the global field.

We remark that a difference between this article and \cite{Walsh3, Cluckers, Vermeulen} is the introduction of a finite subset of primes in the bounds. The reason to do so is rather technical and will be apparent much later (see the introduction to Section \ref{Section 6}); in a few words it will be needed to prove the dimension growth conjecture for small values of $d$.

\subsection{Uniform estimates for Hilbert functions} 
We start by recalling the definition of the Hilbert function of a projective variety.

\begin{definition}
Let $K$ be a field and $I\subseteq K[X_{0},\ldots ,X_{n}]$ be a homogeneous ideal. Given a non-negative integer $k$, we will denote the elements of degree $k$ in $K[X_{0},\ldots ,X_{n}]/I$ by $(K[X_{0},\ldots ,X_{n}]/I)_{k}$. Then we define the Hilbert function as 
$$h_{I}(k):=\dim_{K}(K[X_{0},\ldots ,X_{n}]/I)_{k}.$$ 
If $X\subseteq \mathbb{P}_{K}^{n}=\text{Proj}(K[X_{0},\ldots ,X_{n}])$ is a closed subscheme, then we define $h_{X}(k):=h_{I}(k)$ for the saturated homogeneous ideal $I$ corresponding to $X=\text{Proj}(K[X_{0},\ldots ,X_{n}]/I)$.
\label{Hilbert function}
\end{definition}

By a classic theorem \cite[p. 51-52]{Hartshorne}, there is a unique polynomial $p_{I}(t)\in \mathbb{Q}[t]$, the Hilbert polynomial, such that $h_{I}(k)=p_{I}(k)$ for all sufficiently large positive integers, verifying
$$p_{I}(t)=d\frac{t^{r}}{r!}+O(t^{r-1}),$$
where $r$ is the dimension of $X=\text{Proj}(K[X_{0},\ldots ,X_{n}]/I)$ and $d$ is the degree of $X\subseteq \mathbb{P}^{n}$.

For our purposes, we need to have a uniform estimate in the implicit constant appearing in the error term of the Hilbert polynomial. For this, we require the following two lemmas of Salberger.

\begin{lemma}[{\cite[Lemma 1.3]{Salberger1}}]
Let $n$ be a positive integer and let $K$ be a field. Let $X=\emph{\text{Proj}}(K[X_{0},\ldots ,X_{n}]/I)\subseteq \mathbb{P}_{K}^{n}$ be a closed subscheme, with $I$ a homogeneous saturated ideal. Then $X$ is defined by forms in $K[X_{0},\ldots ,X_{n}]$ of degree bounded in terms of $n$ and $p_{I}(t)$.
\label{uniform bound1}
\end{lemma}

\begin{lemma}[{\cite[Lemma 1.4]{Salberger1}}]
Let $d$ and $n$ be positive integers and let $K$ be a field. Then there are only finitely many possibilities for the Hilbert function $h_{X}:\mathbb{N}\to \mathbb{Z}$ of closed geometrically reduced equidimensional subschemes $X\subseteq \mathbb{P}_{K}^{n}$ of degree $d$.
\label{uniform bound2}
\end{lemma}

In order to obtain uniform results in the degree and the dimension, we use Gr\"obner basis as in  Broberg's article \cite{Broberg}. We recall the following classic definition from \cite[Chapter 2, $\S$ 2, Definition 1 and Chapter 8, $\S$ 4, p. 416]{Cox} :
\begin{definition}[Graded monomial ordering and leading monomial]
A graded monomial ordering $<$ is a well-ordering on the set of monomials ${\bf X}^{{\boldsymbol \alpha}}=X_{0}^{\alpha_{0}}\cdots X_{n}^{\alpha_{n}}$ satisfying the following conditions for the corresponding order on the set of exponents ${\boldsymbol \alpha}=(\alpha_{0},\ldots ,\alpha_{n})\in \mathbb{Z}_{\geq 0}^{n+1}$:
\begin{enumerate}
\item If $\boldsymbol \alpha>\boldsymbol \beta$ and $\boldsymbol \gamma\in \mathbb{Z}_{\geq 0}^{n+1}$, then $\boldsymbol\alpha+\boldsymbol\gamma>\boldsymbol\beta+\boldsymbol\gamma$;
\item if $\boldsymbol \alpha,\boldsymbol \beta\in \mathbb{Z}_{\geq 0}^{n+1}$ and $\alpha_{0}+\ldots +\alpha_{n}< \beta_{0}+\ldots +\beta_{n}$, then $\boldsymbol \alpha<\boldsymbol \beta$.
\end{enumerate} 
Moreover, if $f=\sum_{\boldsymbol \alpha}a_{\boldsymbol \alpha}\boldsymbol X^{\boldsymbol \alpha}$, the leading monomial of $f$, denoted by $LM(f)$, is the maximum of all $\boldsymbol X^{\boldsymbol \alpha}$ with $a_{\boldsymbol \alpha}\neq 0$.
\end{definition}

Let $I\subseteq K[X_{0},\ldots ,X_{n}]$ be a homogeneous monomial ideal, i.e. an ideal generated by monomials. By the same argument given in the proof of \cite[Chapter 9, $\S$ 3, Proposition 3]{Cox}, $h_{I}(k)$ is the number of monomials not in $I$ of degree $k$. Furthermore, given a homogeneous ideal $I\in K[X_{0},\ldots ,X_{n}]$, let $<$ be a graded monomial ordering, and consider the set $\{LM(f):f\in I_{k}\}$, where $I_{k}$ denotes the elements of degree $k$ in $I$. Then $\{LM(f):f\in I_{k}\}=\{LM(f_{1}),\ldots ,LM(f_{m})\}$ for some polynomials $f_{1},\ldots ,f_{m}$ of degree $k$. It holds that $\{f_{1},\ldots ,f_{m}\}$ is a basis for $I_{k}$ and $\{LM(f_{1}),\ldots ,LM(f_{m})\}$ is a basis of the $k$-degree graded part of the ideal $\left\langle LT(I)\right\rangle$, where
$$LT(I):=\{c\boldsymbol X^{\boldsymbol \alpha}:\text{ exists }f\in I\backslash\{0\}\text{ satisfying }LT(f)=c\boldsymbol X^{\boldsymbol \alpha}\},$$ 
(see \cite[Chapter 9, $\S$ 3, Proposition 9]{Cox}).
We conclude
$$(K[X_{0},\ldots ,X_{n}]/I)_{k}\cong (K[X_{0},\ldots ,X_{n}]/\left\langle LT(I)\right\rangle)_{k},$$
and thus 
$$h_{I}(k)=h_{\left\langle LT(I)\right\rangle}(k).$$ 
This is quite useful, since this reduces the problem of computing $h_{I}(k)$ to the problem of computing the Hilbert function of a monomial ideal, which is easy to do, and moreover, the monomial ideal $\left\langle LT(I)\right\rangle$ can be computed by means of a Gr\"obner basis of $I$. Furthermore, $(K[X_{0},\ldots ,X_{n}]/\left\langle LT(I)\right\rangle)_{k}$ has as a $K$-basis those monomials which are not the leading term of some form of degree $k$ in $I$. From this, it follows easily that the same does hold for $(K[X_{0},\ldots ,X_{n}]/I)_{k}$. This is because  the monomials of degree $k$ which are not the leading term of some form of degree $k$ in $I$ are linearly independent on $(K[X_{0},\ldots ,X_{n}]/I)_{k}$, and have cardinal $h_{I}(k)$, thus they are a $K$-basis.

We will also require:

\begin{definition}
Let $K$ be a field and $I\subseteq K[X_{0},\ldots ,X_{n}]$ be a homogeneous ideal. Let $<$ be a graded monomial order on $\{\boldsymbol X^{\boldsymbol \alpha}:\boldsymbol \alpha\in \mathbb{Z}_{\geq 0}^{n+1}\}$ and $m\in \{0,\ldots ,n\}$. Define
$$\sigma_{I,m}(k):=\sum_{\boldsymbol \alpha}\alpha_{m},$$
where $\boldsymbol \alpha=(\alpha_{0},\ldots ,\alpha_{n})$ runs over the exponent set of all monomials $\boldsymbol X^{\boldsymbol \alpha}$ which are not leading monomials of any form of degree $k$ in $I$.
\label{sigma(k)}
\end{definition}

Since the monomials of degree $k$ which are not leading monomials of any form of degree $k$ in $I$ form a $K$-basis of $(K[X_{0},\ldots, X_{n}]/I)_{k}$, it holds
\begin{equation}
\sigma_{I,0}(k)+\ldots +\sigma_{I,n}(k)=kh_{I}(k).
\label{elemental equality}
\end{equation}

It turns out that the functions $\sigma_{I,m}(k)$ are numerical polynomials. More precisely: 

\begin{lemma}[{\cite[Lemma 1]{Broberg}}]
Let $X=\emph{\text{Proj}}(K[X_{0},\ldots ,X_{n}]/I)\subseteq \mathbb{P}_{K}^{n}$ be a closed subscheme of dimension $r$ and degree $d$, such that $I\subseteq K[X_{0},\ldots ,X_{n}]$ is generated by forms of degrees at most $\delta$. There is a positive integer $k_{0}\lesssim_{n,\delta}1$, such that for all $k\geq k_{0}$, $h_{X}(k)$ is equal to some polynomial of degree $r$, and $\sigma_{I,0}(k),\ldots ,\sigma_{I,n}(k)$ are equal to some polynomials of degree $r+1$. Furthermore, the coefficients of these polynomials can be bounded in terms of $n$ and $\delta$.
\label{Broberg}
\end{lemma}

As a consequence of Lemma \ref{Broberg}, the limit $\lim_{k\to \infty}\frac{\sigma_{I,m}(k)}{kh_{I}(k)}$ exists. This, together with Lemma \ref{uniform bound1}, Lemma \ref{uniform bound2}, and equation \eqref{elemental equality} gives:

\begin{lemma}[{\cite[Lemma 1.9]{Salberger1}}]
Let $K,I,<,\sigma_{I,m}$ be as above and suppose that $X=\emph{\text{Proj}}(K[X_{0},\ldots ,X_{n}]/I)$ is reduced and equidimensional. Denote $d=\deg(X)$. For all $m\in \{0,\ldots, n\}$, there exists $0\leq a_{I,m}\leq 1$ with $a_{I,0}+\ldots +a_{I,n}=1$ such that
$$\frac{\sigma_{I,m}(k)}{kh_{I}(k)}=a_{I,m}+O_{d,n}\left(\frac{1}{k}\right),\; m\in \{0,\ldots ,n\}.$$
\label{sigma(k)2}
\end{lemma}

\subsection{Prime powers dividing certain determinants}
We begin with the following lemma. 

 \begin{lemma}
Let $K$ be a global field, and let $X=\emph{\text{Proj}}(K[X_{0},\ldots ,X_{n}]/I)$ be a closed reduced subscheme of $\mathbb{P}_{K}^{n}$.  The schematic closure of  $X$ in $\mathbb{P}_{\mathcal{O}_{K}}^{n}=\emph{\text{Proj}}(\mathcal{O}_{K}[X_{0},\ldots ,X_{n}])$ is equal to 
$$\Xi:=\emph{\text{Proj}}(\mathcal{O}_{K}[X_{0},\ldots ,X_{n}]/(I\cap \mathcal{O}_{K}[X_{0},\ldots ,X_{n}])).$$
\label{schematic closure}
\end{lemma}
\begin{proof}
Due to the reduced hypothesis, $I\cap \mathcal{O}_{K}[X_{0},\ldots ,X_{n}]$ is reduced, thus $\Xi$ is reduced. We have a ring morphism 
$$\varphi:\mathcal{O}_{K}[X_{0},\ldots ,X_{n}]/(I\cap \mathcal{O}_{K}[X_{0},\ldots ,X_{n}])\to K[X_{0},\ldots ,X_{n}]/I.$$ 
Since, for any homogeneous prime ideal $\mathfrak{P}\subseteq K[X_{0},\ldots ,X_{n}]/I$ which does not contain $(X_{0},\ldots, X_{n})$, it holds that $\varphi^{-1}(\mathfrak{P})$ is an homogeneous prime ideal which does not contain $(X_{0},\ldots ,X_{n})$, we see that there exists a morphism $\text{Proj}(\varphi):X\to \Xi$
given by $\mathfrak{P}\to \varphi^{-1}(\mathfrak{P})$. Now, note that the minimal primes of $K[X_{0},\ldots ,X_{n}]/I$ correspond to the primes appearing in the primary decomposition of $I$. If $I=\mathfrak{P}_{1}\cap \cdots \cap \mathfrak{P}_{r}$, then $I\cap \mathcal{O}_{K}[X_{0},\ldots ,X_{n}]=(\mathfrak{P}_{1}\cap \mathcal{O}_{K}[X_{0},\ldots ,X_{n}])\cap \cdots \cap (\mathfrak{P}_{r}\cap \mathcal{O}_{K}[X_{0},\ldots ,X_{n}])$ and $\{\mathfrak{P}_{1}\cap \mathcal{O}_{K}[X_{0},\ldots ,X_{n}],\ldots ,\mathfrak{P}_{r}\cap \mathcal{O}_{K}[X_{0},\ldots ,X_{n}]\}$ is the set of minimal primes of $I\cap \mathcal{O}_{K}[X_{0},\ldots ,X_{n}]$, thus they are the set of generic points of $\Xi$. Since $\text{Proj}(\varphi)(\mathfrak{P}_{i})=\mathfrak{P}_{i}\cap \mathcal{O}_{K}[X_{0},\ldots ,X_{n}]$ for all $i$, the image of $\text{Proj}(\varphi)$ contains all the generic points of $\Xi$, thus its closure is $\Xi$. Since $\Xi$ is reduced, by \cite[(10.8)]{Gortz} we conclude that $\Xi$ is the schematic closure of $X$.
\end{proof}

From here on, $K$ will denote a global field. We shall work with closed subschemes $X=\text{Proj}(K[X_{0},\ldots ,X_{n}]/I)$ of $\mathbb{P}_{K}^{n}$, which are geometrically reduced, equidimensional of dimension $r$, and degree $d$.  Also, we will denote 
$$\Xi:=\text{Proj}(\mathcal{O}_{K}[X_{0},\ldots ,X_{n}]/(I\cap \mathcal{O}_{K}[X_{0},\ldots ,X_{n}])).$$ 
Furthermore, given a prime $\mathfrak{p}\notin M_{K,\infty}$, if $\mathbb{F}_{\mathfrak{p}}:=\mathcal{O}_{K}/\mathfrak{p}$, we let $I_{\mathfrak{p}}\subset \mathbb{F}_{\mathfrak{p}}[X_{0},\ldots ,X_{n}]$ be the image of $I\cap \mathcal{O}_{K}[X_{0},\ldots ,X_{n}]$ in $\mathbb{F}_{\mathfrak{p}}[X_{0},\ldots ,X_{n}]$, and we set
$$X_{\mathfrak{p}}:=\text{Proj}(\mathbb{F}_{\mathfrak{p}}[X_{0},\ldots ,X_{n}]/I_{\mathfrak{p}})=\Xi\times_{\mathbb{Z}}\mathbb{F}_{\mathfrak{p}}.$$

The following lemmas, that are global field variants of \cite[Lemma 2.2, Lemma 2.3]{Salberger} show that the reductions modulo $\mathfrak{p}$ of $X$, i.e. the $\mathbb{F}_{\mathfrak{p}}$-schemes $X_{\mathfrak{p}}$, retain the same uniform bounds of Lemma \ref{uniform bound1} and Lemma \ref{uniform bound2}. 

\begin{lemma}
Let $d$ and $n$ be given and $X,X_{\mathfrak{p}}$ as above. Then $X_{\mathfrak{p}}$ is defined by forms of degree $\lesssim_{d,n}1$.
\label{uniform bound3}
\end{lemma}

\begin{proof}
In order to prove Lemma \ref{uniform bound3}, we will use the key property that a flat morphism has all its fibres with the same Hilbert polynomial. Let $X$ and $\Xi$ be as above, and let $f:\Xi\to \text{Spec}(\mathcal{O}_{K})$ be the structural morphism. If $\eta$ is the generic point of $\text{Spec}(\mathcal{O}_{K})$, by Lemma \ref{schematic closure} the schematic closure of $f^{-1}(\eta)=X$ is $\Xi$. Then by \cite[Proposition 14.14]{Gortz} $f$ is flat.

By hypothesis, $X$ is geometrically reduced, so we may apply Lemma \ref{uniform bound2}. There are thus $\lesssim_{d,n}1$ possible Hilbert polynomials for the schemes $X\subseteq \mathbb{P}_{K}^{n}$. But the Hilbert polynomials of $X=\Xi\times_{\mathcal{O}_{K}}K$ and $X_{\mathfrak{p}}=\Xi\times_{\mathcal{O}_{K}}\mathbb{F}_{\mathfrak{p}}$ coincide since they are fibres of $f$, which is flat (this follows from \cite[Theorem 9.9]{Hartshorne}). There are thus $\lesssim_{d,n} 1$ possible Hilbert polynomials also for the $\mathbb{F}_{\mathfrak{p}}$-subschemes $X_{\mathfrak{p}}$. The result follows from Lemma \ref{uniform bound1}.  
\end{proof}

\begin{lemma}
Let $d$ and $n$ be given and $X,X_{\mathfrak{p}}$ be as above. Let $A$ be the stalk of $X_{\mathfrak{p}}$ at some $\mathbb{F}_{\mathfrak{p}}$-point $P$ on $X_{\mathfrak{p}}$ of multiplicity $\mu$. Let $\mathfrak{m}$ be the maximal ideal of $A$. Let $g_{X,P}:\mathbb{Z}_{>0}\to \mathbb{Z}$ be the function where $g_{X,P}(k)=\dim_{A/\mathfrak{m}}\mathfrak{m}^{k}/\mathfrak{m}^{k+1}$. Then there are $\lesssim_{d,n}1$ different functions $g_{X,P}$ among all pairs $(X,P)$ as above, and we have 
$$g_{X,P}(k)\leq \mu\frac{k^{r-1}}{(r-1)!}+O_{d,n}(k^{r-2})$$
for all $k\geq 1$, where $r=\dim(X)$. 
Furthermore, If $X$ is a hypersurface, the implicit constant in the second summand does not depend on $d$.
\label{uniform bound4}
\end{lemma}

\begin{proof}
The function $g_{X,P}$ coincides with the Hilbert function of the projective tangent cone $\text{PTC}_{P}X\subseteq \mathbb{P}_{\mathbb{F}_{\mathfrak{p}}}^{n-1}$ of $X_{\mathfrak{p}}$ at $P$. Moreover, from Lemma \ref{uniform bound3} $X_{\mathfrak{p}}$ is defined by forms of degree $\lesssim_{d,n}1$, thus this is also the case for $\text{PTC}_{P}X$ (this follows from the fact that the tangent cone can be effectively computed in terms of a Gr\"obner basis of $I$, and the fact that the degrees of the elements in a Gr\"obner basis of $I$ are bounded in terms of $n$ and  the degrees of the elements in a basis of $I$, see \cite[III.3]{Red} and \cite[Chapter 9, $\S$ 7, Proposition 4 ]{Cox}). By Lemma \ref{Broberg}, there is a positive integer $k_{0}\lesssim_{n,d}1$ such that the Hilbert function $g_{X,P}(k)$ coincides with a polynomial $p_{X,P}(k)$ for all $k\geq k_{0}$, with coefficients bounded in terms of $n$ and $d$. In particular, $g_{X,P}(k)\leq \mu \frac{k^{r-1}}{(r-1)!}+O_{d,n}(k^{r-2})$ for all $k\geq 1$. 

If $X$ is a hypersurface, the projective tangent cone $\text{PTC}_{P}X\subseteq \mathbb{P}_{\mathbb{F}_{\mathfrak{p}}}^{n-1}$ is a hypersurface of degree $\mu$ in $\mathbb{P}^{n}$, thus the function $g_{X,P}$ has the explicit formula
$$g_{X,P}(k)=\begin{cases}{r+k\choose r}\text{ for }k<\mu,\\ {r+k\choose r}-{r+k-\mu\choose r}\text{ for }k\geq \mu.\end{cases}$$ 
Now, the same argument in \cite[Lemma 2.1]{Cluckers} finishes the proof.
\end{proof}

We will need the following lemmas.

\begin{lemma}[{\cite[Proof of Main Lemma 2.5]{Salberger1},\cite[Lemma 2.2]{Cluckers}}]
Let $c,r,\mu>0$ be integers. Let $g:\mathbb{Z}_{\geq 0}\to \mathbb{Z}_{\geq 0}$ be a function with $g(0)=1$ and satisfying $g(k)\leq \mu\frac{k^{r-1}}{(r-1)!}+ck^{r-2}$ for $k>0$. Let $(n_{i})_{i\geq 1}$ be the non-decreasing sequence of integers $m\geq 0$ where $m$ occurs exactly $g(m)$ times. Then for any $s\geq 0$ we have
$$n_{1}+\cdots +n_{s}\geq \left( \frac{r!}{\mu} \right)^{\frac{1}{r}}\frac{r}{r+1}s^{1+\frac{1}{r}}-O_{r,c}(s).$$
\label{Cluckers1}
\end{lemma}

\begin{lemma}[{\cite[Lemma 2.3]{Cluckers}, see also \cite[Proposition A.1]{Liu2} for explicit bounds}]
Consider $A$ as in Lemma \ref{uniform bound4}, and let $(n_{i}(A))_{i\geq 1}$ be the non-decreasing sequence of integers $k\geq 0$ where $k$ occurs exactly $\dim_{A/\mathfrak{m}}\mathfrak{m}^{k}/\mathfrak{m}^{k+1}$ times. Write $A(s)=n_{1}(A)+\cdots +n_{s}(A)$. Then
$$A(s)\geq \left(\frac{r!}{\mu}\right)^{\frac{1}{r}}\frac{r}{r+1}s^{1+\frac{1}{r}}-O_{d,n}(s),$$ 
where $r=\dim(X)$.
Furthermore, if $X$ is a hypersurface, the implicit constant in the second summand does not depend on $d$.
\label{Cluckers2}
\end{lemma}

In order to provide our generalization of \cite[Lemma 2.4 and Main Lemma 2.5]{Salberger1} to global fields, we follow the exposition in \cite[$\S$2]{Cluckers}. We denote $\mathcal{O}_{\mathfrak{p}}$ for the localization of $\mathcal{O}_{K}$ at the prime $\mathfrak{p}$.

\begin{lemma}
Let $R$ be a commutative noetherian local ring containing $\mathcal{O}_{\mathfrak{p}}$ as a subring and let $A=R/\mathfrak{p}R$. Let $\mathfrak{m}$ be the maximal ideal of $A$ and let $(n_{i}(A))_{i\in \mathbb{N}}$ be the non-decreasing sequence of integers $k\geq 0$ where $k$ occurs exactly $\dim_{A/\mathfrak{m}}\mathfrak{m}^{k}/\mathfrak{m}^{k+1}$ times. Let $r_{1},\ldots ,r_{s}$ be elements of $R$ and $\phi_{1},\ldots \phi_{s}:R\to \mathcal{O}_{\mathfrak{p}}$ be ring homomorphisms such that $\phi_{i}|_{\mathcal{O}_{\mathfrak{p}}}=Id$. Then the determinant of the $s\times s$ matrix $(\phi_{i}(r_{j}))_{i,j}$ is divisible by $\mathfrak{p}^{A(s)}$ for $A(s):=n_{1}(A)+\ldots +n_{s}(A)$. 
\label{lemma valuation}
\end{lemma}

\begin{proof}
Choose one of the homomorphisms $\phi_{i}$, say $\phi_{1}$. Let $I:=\text{Ker}(\phi_{1})$. Since the map $\mathcal{O}_{\mathfrak{p}}\hookrightarrow R\xrightarrow[]{\phi_{1}}\mathcal{O}_{\mathfrak{p}}$ is the identity, it holds $\phi_{1}(\mathcal{O}_{\mathfrak{p}})=\mathcal{O}_{\mathfrak{p}}$, hence $R/I\cong \mathcal{O}_{\mathfrak{p}}$. If $\widetilde{\phi_{1}}:R/I\to\mathcal{O}_{\mathfrak{p}}$ is the induced homomorphism, then $\widetilde{\phi_{1}}^{-1}(\mathfrak{p}\mathcal{O}_{\mathfrak{p}})=(\mathfrak{p}R+I)/I$ is a maximal ideal, from where it follows that $\mathfrak{p}R+I$ is the maximal ideal of $R$. Furthermore, $\phi_{i}(I+\mathfrak{p}R)\mathcal{O}_{\mathfrak{p}}\neq \mathcal{O}_{\mathfrak{p}}$, otherwise we would find $x\in I+\mathfrak{p}R$ and $a\in \mathcal{O}_{\mathfrak{p}}$ with $a\phi_{i}(x)=1$, hence $\phi_{i}(ax-1)=0$, with contradicts that $I+\mathfrak{p}R$ is the only maximal ideal of $R$. We conclude that $\phi_{i}(I)$ is a proper ideal of $\mathcal{O}_{\mathfrak{p}}$, and hence $\phi_{i}(I)\subseteq \mathfrak{p}\mathcal{O}_{\mathfrak{p}}$. In particular, for all $i$, $\phi_{i}(I^{k})\subseteq \mathfrak{p}^{k}\mathcal{O}_{\mathfrak{p}}$.
From Nakayama's lemma, it follows that for all $k\geq 0$, $I^{k}/I^{k+1}$ is an $\mathcal{O}_{\mathfrak{p}}$-module generated with at most $\dim_{A/\mathfrak{m}}\mathfrak{m}^{k}/\mathfrak{m}^{k+1}$ elements. 
Now, suppose that there are more than $g(k):=\dim_{A/\mathfrak{m}}\mathfrak{m}^{k}/\mathfrak{m}^{k+1}$ elements in $I^{k}-I^{k+1}$ among $r_{1},\ldots ,r_{s}$, say $r_{1},\ldots ,r_{q}$. There exists an $\mathcal{O}_{\mathfrak{p}}$-module homomorphism $\lambda:\mathcal{O}_{\mathfrak{p}}^{q}\to I^{k}/I^{k+1}$ where $\lambda(\beta_{1},\ldots ,\beta_{q})=\beta_{1}r_{1}+\ldots +\beta_{q}r_{q}\;(\text{mod } I^{k+1})$. Since $I^{k}/I^{k+1}$ is generated by at most $g(k)<q$ elements we conclude that there exists $(\beta_{1},\ldots ,\beta_{q})\in \text{Ker}(\lambda)$ with, say, $\beta_{q}=1$ (this is because $I^{k}/I^{k+1}$ is a finite $\mathcal{O}_{\mathfrak{p}}$-module generated by at most $g(k)$ elements and $\mathcal{O}_{\mathfrak{p}}$ is local, then by Nakayama's lemma there exists $j\in \{1,\ldots ,q\}$, say $j=q$, such that $\text{span}_{\mathcal{O}_{\mathfrak{p}}}\{r_{1},\ldots ,r_{q}\}=\text{span}_{\mathcal{O}_{\mathfrak{p}}}\{r_{1},\ldots ,r_{q-1}\}$). Denoting $\rho_{q}=\beta_{1}r_{1}+\ldots +\beta_{q-1}r_{q-1}+r_{q}$, we see that $\rho_{q}\in I^{k+1}$, furthermore, since the homomorphisms $\phi_{1},\ldots ,\phi_{s}$ restrict to the identity in $\mathcal{O}_{\mathfrak{p}}$, we see that the determinant of $(\phi_{i}(r_{j}))_{i,j}$ will not change if we replace $r_{q}$ with $\rho_{q}$. If we continue making such elementary transformations, we arrive to a situation where there are at most $\dim_{A/\mathfrak{m}}\mathfrak{m}^{k}/\mathfrak{m}^{k+1}$ elements in $I^{k}-I^{k+1}$ among $r_{1},\ldots ,r_{s}$, for each $k\geq 0$.  

After rearranging, we have that $r_{j}$ is in $I^{n_{j}(A)}$ for all $j=1,\ldots ,s$. This implies that all the elements $\phi_{1}(r_{j}),\ldots ,\phi_{s}(r_{j})$ are divisible by $\mathfrak{p}^{n_{j}(A)}$, hence $\mathfrak{p}^{A(s)}$ divides the determinant of $(\phi_{i}(r_{j}))_{i,j}$.
\end{proof}

\begin{lemma}
Let $\mathfrak{p}$ be a prime and let $X,r,d,n,\Xi,X_{\mathfrak{p}}$ be as above. Let $P$ be an $\mathbb{F}_{\mathfrak{p}}$-point of multiplicity $\mu_{P}$ on $X_{\mathfrak{p}}$ and let $\boldsymbol \xi_{1},\ldots ,\boldsymbol \xi_{s}$ be points in $\mathcal{O}_{K}^{n+1}$, such that, for all $i$, ${\boldsymbol \xi_{i}}$ has its coordinates not all divisible by $\mathfrak{p}$ and it represents a $K$-rational point of $X$ which reduces modulo $\mathfrak{p}$ to $P$. Let $F_{1},\ldots ,F_{s}$ be forms in $(X_{0},\ldots ,X_{n})$ with coefficients in $\mathcal{O}_{K}$, and let $\det(F_{j}(\boldsymbol \xi_{l})_{j,l})$ be the determinant of the $s\times s$ matrix $(F_{j}(\boldsymbol \xi_{l}))_{j,l}$. Then there exists a non-negative integer $N=(\frac{r!}{\mu_{P}})^{\frac{1}{r}}\frac{r}{r+1}s^{1+\frac{1}{r}}+O_{d,n}(s)$ such that $\mathfrak{p}^{N}|\det(F_{j}(\boldsymbol \xi_{l})_{j,l})$. Furthermore, if $X$ is a hypersurface, the implicit constant in the second summand does not depend on $d$.
\label{local determinant}
\end{lemma}
\begin{proof}
Let $P=(z_{0}:z_{1}:\ldots :z_{n})$ with $z_{i}\in \mathbb{F}_{\mathfrak{p}}$ for all $i$. It holds that there is some $i$ with  $z_{i}\neq 0$, which we may suppose that it is $i=0$. By hypothesis, it also follows that the $x_{0}$-coordinates of the points $\boldsymbol \xi_{1},\ldots \boldsymbol \xi_{s}$ are not divisible by $\mathfrak{p}$ and we replace the forms $F_{j},j=1,\ldots ,s$ by the rational functions $f_{j}=F_{j}(1,\frac{x_{1}}{x_{0}},\ldots ,\frac{x_{n}}{x_{0}})$ without changing the $\mathfrak{p}$-adic valuation of the determinant. 

Now, we define a prime ideal in $\mathcal{O}_{K}[X_{0},\ldots ,X_{n}]$ by $\mathfrak{P}:=\left\langle \mathfrak{p},X_{0}-z_{0},\ldots ,X_{n}-z_{n}\right\rangle$. Let $R:=\mathcal{O}_{K}[X_{0},\ldots ,X_{n}]_{\mathfrak{P}}$. It is clear that $R$ is noetherian local ring containing $\mathcal{O}_{\mathfrak{p}}$ as a subring, and the rational functions $f_{j},j=1,\ldots ,s$ are elements in $R$. Moreover, the evaluations $\text{ev}_{\boldsymbol \xi_{1}},\ldots \text{ev}_{\boldsymbol \xi_{s}}:\mathcal{O}_{K}[X_{0},\ldots ,X_{n}]\to \mathcal{O}_{K}$ are ring homomorphisms, which restrict to the identity on $\mathcal{O}_{K}$, and can be extended to ring homomorphisms $\overline{\text{ev}_{\boldsymbol \xi_{1}}},\ldots ,\overline{\text{ev}_{\boldsymbol \xi_{s}}}:R\to \mathcal{O}_{\mathfrak{p}}$ which restrict to the identity on $\mathcal{O}_{\mathfrak{p}}$. Taking $\phi_{i}:=\overline{\text{ev}_{\boldsymbol \xi_{i}}}$ for all $i$, we see that $R$ and $\phi_{1},\ldots ,\phi_{s}$ are as in Lemma \ref{lemma valuation}, hence $\mathfrak{p}^{A(s)}|\det(\phi_{l}(f_{j}))_{j,l}=\det(F_{j}(\boldsymbol \xi_{l}))_{j,l}$. The proof finishes using Lemma \ref{Cluckers2}.
\end{proof}

\subsection{From local estimates to global estimates}

In this section, $X\subseteq \mathbb{P}_{K}^{n+1}$ will denote a geometrically integral hypersurface of degree $d$. Hence, $X$ is defined by an absolutely irreducible homogeneous polynomial $f=\sum_{I}a_{I}\boldsymbol X^{I}\in \mathcal{O}_{K}[X_{0},\ldots ,X_{n+1}]$ of degree $d$. We shall begin with an important remark concerning the polynomial $f$ defining $X$. 

\begin{remark}[Primitive polynomials in global fields]
We recall that $f$ is said to be primitive if $\mathfrak{a}:=\sum_{I}a_{I}\mathcal{O}_{K}=(1)$. Also, given a prime $\mathfrak{p}$ of $\mathcal{O}_{K}$, we will say that $f$ is $\mathfrak{p}$-primitive if $\mathfrak{p}$ does not divide $\mathfrak{a}$. Note that in general we can not take the polynomial $f$ defining $X$ to be primitive as in the case $\mathbbm{k}=\mathbb{Q}$ or $\mathbb{F}_{q}(T)$. Indeed, by the proof of Proposition \ref{Serre}, the ideal class of $\mathfrak{a}=\sum_{I}a_{I}\mathcal{O}_{K}$ is determined by the projective point with coordinates $(a_{I})_{I}$. Hence, if $\mathcal{O}_{K}$ is not a principal domain, it may occur that the ideal class of $\mathfrak{a}$ is not principal, thus there is no non-zero scalar $\lambda\in \mathcal{O}_{K}$ such that $\lambda f$ is primitive. Nevertheless, by Proposition \ref{Serre} we may find a non-zero $\lambda\in \mathcal{O}_{K}$ such that $\lambda f$  is $\mathfrak{p}$-primitive for all primes $\mathfrak{p}\subseteq \mathcal{O}_{K}$ with $\mathcal{N}_{K}(\mathfrak{p})>c_{2}$, where $c_{2}$ is the constant in Proposition \ref{Serre}. Also, by Proposition \ref{Serre}, for all $v\in M_{K,\infty}$, 
 it holds $\max_{I}|\lambda a_{I}|_{v}\leq c_{1} H(f)^{\frac{n_{v}}{d_{K}}}$ if $K$ is a number field, and $\max_{I}|\lambda a_{I}|_{v_{\infty}}\leq c_{1}H(f)$ if $K$ is a function field, for some effective computable constants $c_{1}$, and $c_{3}$ which depend on $K$. By definition \ref{affine height of a polynomial}, it holds
\begin{align*}
H_{K,\text{aff}}(\lambda f)=\left(\prod_{v\in M_{K}}\max_{I}\{1,|\lambda a_{I}|_{v}\}\right)^{d_{K}}  =\left(\prod_{v\in M_{K,\infty}}\max_{I}\{1,|\lambda a_{I}|_{v}\}\right)^{d_{K}} \nonumber  \leq c_{1}^{d_{K}}H_{K}(f).
\end{align*}

Furthermore, since the height is invariant under multiplication by non-zero scalars, it holds $H_{K}(\lambda f)=H_{K}(f)$. For this reason we may assume that $X$ is defined by an absolutely irreducible homogeneous polynomial $f=\sum_{I}a_{I}\boldsymbol X^{I}\in \mathcal{O}_{K}[X_{0},\ldots ,X_{n+1}]$ of degree $d$, which is $\mathfrak{p}$-primitive for all primes $\mathfrak{p}$ with $\mathcal{N}_{K}(\mathfrak{p})>c_{2}$, and
\begin{equation*}
H_{K,\text{aff}}(f)\leq c_{1}^{d_{K}}H_{K}(f).
\end{equation*}
\label{remark about primitive polynomials}
\end{remark}

We proceed to generalize \cite[Theorem 2.2]{Walsh3}, \cite[Corollary 2.9]{Cluckers} and \cite[Lemma 2.1]{Vermeulen} to global fields.

\begin{theorem}
Let $\mathfrak{p}$ be a prime for which the reduction $X_{\mathfrak{p}}$ is absolutely irreducible, and for which either $\emph{\text{char}}(K)=0$ and $\mathcal{N}_{K}(\mathfrak{p})>\max\{27d^{4},c_{2}\}$, or $0<\emph{\text{char}}(K)\leq \max\{27d^{4},c_{2}\}$ and $\mathcal{N}_{K}(\mathfrak{p})\geq d^{\frac{14}{3}}$, or $\emph{\text{char}}(K)>\max\{27d^{4},c_{2}\}$. Let $(\boldsymbol \xi_{1},\ldots ,\boldsymbol \xi_{s})$ be a tuple of points in $\mathcal{O}_{K}^{n+1}$, such that, for all $i$, $\boldsymbol \xi_{i}$ has its coordinates not all divisible by $\mathfrak{p}$ and it represents a $K$-rational point of $X$. Let $F_{1},\ldots ,F_{s}\in \mathcal{O}_{K}[X_{0},\ldots ,X_{n+1}]$ be forms with $\mathcal{O}_{K}$-coefficients, and write $\Delta$ for the determinant of $(F_{i}(\boldsymbol \xi_{j}))_{i,j}$. Then, there exists some
$$e_{\mathfrak{p}}\geq n!^{\frac{1}{n}}\frac{n}{n+1}\frac{s^{1+\frac{1}{n}}}{\mathcal{N}_{K}(\mathfrak{p})+O_{n}(d^{2}\mathcal{N}_{K}(\mathfrak{p})^{\frac{1}{2}})}-O_{n}(s)$$
such that $\mathfrak{p}^{e_{\mathfrak{p}}}|\Delta$.
\label{local determinant2}
\end{theorem}

\begin{proof}
Let $P$ be an $\mathbb{F}_{\mathfrak{p}}$-point on $X_{\mathfrak{p}}$ and let us write $s_{P}$ for the cardinal of the subset $I_{P}\subseteq \{1,\ldots ,s\}$ of indices $l$ such that $\boldsymbol \xi_{l}$ reduces to $P$ modulo $\mathfrak{p}$. Let  
$$N_{P}=\left(\frac{n!}{\mu_{P}}\right)^{\frac{1}{n}}\frac{n}{n+1}s_{P}^{1+\frac{1}{n}}+O_{n}(s_{P})$$
be the integer in Lemma \ref{local determinant}. Then by Lemma \ref{local determinant}, for each $s_{P}\times s_{P}$ minor $A=(a_{k,l})_{k,l}$ of $(F_{i}(\boldsymbol \xi_{j}))_{i,j}$ with $l\in I_{P}$, it holds that $\mathfrak{p}^{N_{P}}|\det(A)$. If we apply this to all the $\mathbb{F}_{\mathfrak{p}}$-points on $X_{\mathfrak{p}}$ and use the Laplace expansion,
we see that $\mathfrak{p}^{e_{\mathfrak{p}}}|\Delta$, where

\begin{align}
e_{\mathfrak{p}}\geq \sum_{P\in X(\mathbb{F}_{\mathfrak{p}})}N_{P} =n!^{\frac{1}{n}}\frac{n}{n+1}\left( \sum_{P\in X(\mathbb{F}_{\mathfrak{p}})}\frac{s_{P}^{1+\frac{1}{n}}}{\mu_{P}^{\frac{1}{n}}}\right)-O_{n}\left(\sum_{P\in X(\mathbb{F}_{\mathfrak{p}})}s_{P}\right) = n!^{\frac{1}{n}}\frac{n}{n+1}\left( \sum_{P\in X(\mathbb{F}_{\mathfrak{p}})}\frac{s_{P}^{1+\frac{1}{n}}}{\mu_{P}^{\frac{1}{n}}}\right)-O_{n}(s).\label{exponent e}
\end{align}
To conclude the proof, we need to estimate $\sum_{P\in X(\mathbb{F}_{\mathfrak{p}})}\frac{s_{P}^{1+\frac{1}{n}}}{\mu_{P}^{\frac{1}{n}}}$. By H\"older's inequality, 
$$s=\sum_{P\in X(\mathbb{F}_{\mathfrak{p}})}s_{P}\leq \left( \sum_{P\in X(\mathbb{F}_{\mathfrak{p}})}\mu_{P} \right)^{\frac{1}{n+1}}\left(\sum_{P\in X(\mathbb{F}_{\mathfrak{p}})}\frac{s_{P}^{1+\frac{1}{n}}}{\mu_{P}^{\frac{1}{n}}}\right)^{\frac{n}{n+1}},$$
hence
\begin{equation}
\sum_{P\in X(\mathbb{F}_{\mathfrak{p}})}\frac{s_{P}^{1+\frac{1}{n}}}{\mu_{P}^{\frac{1}{n}}}\geq \frac{s^{1+\frac{1}{n}}}{\left(\displaystyle \sum_{P\in X(\mathbb{F}_{\mathfrak{p}})}\mu_{P}\right)^{\frac{1}{n}}}.
\label{Holder}
\end{equation}
If we denote $n_{\mathfrak{p}}:=\sum_{P\in X(\mathbb{F}_{\mathfrak{p}})}\mu_{P}$, from \eqref{exponent e} and \eqref{Holder} we see that
\begin{equation}
e_{\mathfrak{p}}\geq n!^{\frac{1}{n}}\frac{n}{n+1}\frac{s^{1+\frac{1}{n}}}{n_{\mathfrak{p}}^{\frac{1}{n}}}-O_{n}(s).
\label{exponent e2}
\end{equation}
In order to bound $n_{\mathfrak{p}}$, namely the number of $\mathbb{F}_{\mathfrak{p}}$-points in $X_{\mathfrak{p}}$ counted with multiplicity, we observe that the singular points of $X_{\mathfrak{p}}$, say $X_{\mathfrak{p},\text{sing}}$, consist of the intersection of $f$ and the partial derivatives $\frac{\partial f}{\partial x_{i}}$. Hence, we may suppose that $X_{\mathfrak{p},\text{sing}}$ is contained in the intersection of $f$ and $\frac{\partial f}{\partial x_{0}}$, which is the intersection of two hypersurfaces of degree at most $d$.  By B\'ezout's Theorem \cite[Example 8.4.6]{Fulton}, the number of irreducible components of $\mathcal{Z}(f)\cap\mathcal{Z}(\frac{\partial f}{\partial x_{0}})$ is bounded by $d^{2}$, and their degree are also bounded by $d^{2}$. Then the standard Schwarz-Zippel estimate for $\mathcal{Z}(f)\cap\mathcal{Z}(\frac{\partial f}{\partial x_{0}})$ yields $|X_{\mathfrak{p},\text{sing}}(\mathbb{F}_{\mathfrak{p}})|\leq \left|\left(\mathcal{Z}(f)\cap \mathcal{Z}(\frac{\partial f}{\partial x_{0}})\right)(\mathbb{F}_{\mathfrak{p}})\right|\lesssim_{n}d^{2}\mathcal{N}_{K}(\mathfrak{p})^{n-1}$. Hence
\begin{equation}
n_{\mathfrak{p}}-|X(\mathbb{F}_{\mathfrak{p}})|=\sum_{P\in X(\mathbb{F}_{\mathfrak{p}})}(\mu_{P}-1)\leq (d-1)|X_{\mathfrak{p},\text{sing}}(\mathbb{F}_{\mathfrak{p}})|\lesssim_{n} d^{3}\mathcal{N}_{K}(\mathfrak{p})^{n-1}\lesssim_{n}d\mathcal{N}_{K}(\mathfrak{p})^{n-\frac{1}{2}}.
\label{multiplicity estimate}
\end{equation}

\begin{claim}
$|X(\mathbb{F}_{\mathfrak{p}})|\leq \mathcal{N}_{K}(\mathfrak{p})^{n}+O_{n}(d^{2}\mathcal{N}_{K}(\mathfrak{p})^{n-\frac{1}{2}})$.
\label{cafure lang-weil}
\end{claim}

\begin{proof}[Proof of Claim \ref{cafure lang-weil}]
This follows from an effective version of the Lang-Weil estimate. If $K$ is a number field, by hypothesis, we have $\mathcal{N}_{K}(\mathfrak{p})>\max\{27d^{4},c_{2}\}$. On the other hand, if $K$ is a function field over $\mathbb{F}_{q}(T)$ with $\text{char}(K)>\max\{27d^{4},c_{2}\}$, it holds $\mathcal{N}_{K}(\mathfrak{p})\geq q\geq \text{char}(K)>\max\{27d^{4},c_{2}\}$. In either case, since $X_{\mathfrak{p}}$ is absolutely irreducible and $\mathcal{N}_{K}(\mathfrak{p})>27d^{4}$, from \cite[Corollary 5.6]{Cafure} it follows that the number of $\mathbb{F}_{\mathfrak{p}}$-points of $X_{\mathfrak{p}}$ counted without multiplicity is bounded by
$$\frac{\mathcal{N}_{K}(\mathfrak{p})^{n+1}+(d-1)(d-2)\mathcal{N}_{K}(\mathfrak{p})^{n+\frac{1}{2}}+(5d^{2}+d+1)\mathcal{N}_{K}(\mathfrak{p})^{n}-1}{\mathcal{N}_{K}(\mathfrak{p})-1}\leq \mathcal{N}_{K}(\mathfrak{p})^{n}+O_{n}(d^{2}\mathcal{N}_{K}(\mathfrak{p})^{n-\frac{1}{2}}).$$
Similarly, if $K$ is a function field over $\mathbb{F}_{q}(T)$ with $\text{char}(K)<\max\{27d^{4},c_{2}\}$, by hypothesis $\mathcal{N}_{K}(\mathfrak{p})>d^{\frac{14}{3}}$, thus \cite[Theorem 5.2]{Cafure}  gives that the number of $\mathbb{F}_{\mathfrak{p}}$-points of $X_{\mathfrak{p}}$ counted without multiplicity is bounded by
\begin{equation*}
\frac{\mathcal{N}_{K}(\mathfrak{p})^{n+1}+(d-1)(d-2)\mathcal{N}_{K}(\mathfrak{p})^{n+\frac{1}{2}}+5d^{\frac{13}{3}}\mathcal{N}_{K}(\mathfrak{p})^{n}-1}{\mathcal{N}_{K}(\mathfrak{p})-1}\leq \mathcal{N}_{K}(\mathfrak{p})^{n}+O_{n}(d^{2}\mathcal{N}_{K}(\mathfrak{p})^{n-\frac{1}{2}}). \qedhere
\end{equation*}
\vspace{-0.3cm}
\phantom\qedhere
\end{proof}
\noindent Combining \eqref{multiplicity estimate} and  Claim \eqref{cafure lang-weil} we obtain $n_{\mathfrak{p}}\leq \mathcal{N}_{K}(\mathfrak{p})^{n}+O_{n}(d^{2}\mathcal{N}_{K}(\mathfrak{p})^{n-\frac{1}{2}})$. Applying the bound $|x^{\frac{1}{n}}-1|\leq |x-1|$ to $x=\frac{n_{\mathfrak{p}}}{\mathcal{N}_{K}(\mathfrak{p})}$, we conclude that $n_{\mathfrak{p}}^{\frac{1}{n}}\leq \mathcal{N}_{K}(\mathfrak{p})+O_{n}(d^{2}\mathcal{N}_{K}(\mathfrak{p})^{-\frac{1}{2}})$. Replacing this in \eqref{exponent e2} finishes the proof.
\end{proof}

Our next objective is to obtain a global bound for the determinant $\Delta$ in Theorem \ref{local determinant2}, without making any assumptions over its reductions $X_{\mathfrak{p}}$. We follow the presentation given in  \cite[Proposition 3.2.26]{Cluckers} where the authors made explicit the dependence of the degree in the bounds of \cite[Theorem 2.3]{Walsh3}. 

\begin{definition}
Let $c_{2}$ be the constant in Proposition \ref{Serre}. We let 
$$\beta:=\begin{cases} 27d^{4} &\text{ if char}(K)=0,  \\ d^{\frac{14}{3}} & \text{ if }0<\text{char}(K)\leq \max\{27d^{4},c_{2}\},\\ 1 & \text{ if char}(K)>\max\{27d^{4},c_{2}\}. \end{cases}$$
Given $f\in \mathcal{O}_{K}[X_{0},\ldots ,X_{n+1}]$ we define $b(f):=0$ if $f$ is not absolutely irreducible, otherwise we set 
\begin{equation*}
b(f):=\displaystyle \prod_{\mathfrak{p}\in \mathcal{P}_{f}}\exp\left( \frac{\log(\mathcal{N}_{K}(\mathfrak{p}))}{\mathcal{N}_{K}(\mathfrak{p})} \right),
\end{equation*}
where 
\begin{equation}
\mathcal{P}_{f}:=\left\{\mathfrak{p}\notin M_{K,\infty}:\mathcal{N}_{K}(\mathfrak{p})>\max\{\beta,c_{2}\} \text{ and } f \text{ mod } \mathfrak{p} \text{ is not absolutely irreducible}\right\}.
\label{p(f)}
\end{equation}
\label{definition of b(f)}
\end{definition}

\begin{remark}
Observe that if $X\subseteq \mathbb{P}_{K}^{n+1}$ is a geometrically integral hypersurface defined by an absolutely irreducible homogeneous polynomial $f\in \mathcal{O}_{K}[X_{0},\ldots ,X_{n+1}]$ which is $\mathfrak{p}$-primitive, then $X_{\mathfrak{p}}$ is geometrically irreducible if and only if $f\text{ mod } \mathfrak{p}$ is absolutely irreducible. 
\end{remark}

\begin{theorem}
Let $(\boldsymbol \xi_{1},\ldots ,\boldsymbol \xi_{s})$ be a tuple of $K$-rational points in $X$, let $F_{li}\in \mathcal{O}_{K}[X_{0},\ldots ,X_{n+1}]$, $1\leq l\leq L,1\leq i\leq s$, be homogeneous polynomials with integer coefficients, and write $\Delta_{l}$ for the determinant of $(F_{li}(\boldsymbol \xi_{j}))_{ij}$. Let $\Delta$ be the greatest common divisor  of the $\Delta_{l}$, and assume that $\Delta\neq 0$. Let $\{\mathfrak{p}_{1},\ldots ,\mathfrak{p}_{u}\}$ be (a possibly empty) subset of primes and set $\mathfrak{q}:=\prod_{i=1}^{u}\mathfrak{p}_{i}$ if $u\geq 1$ and $\mathfrak{q}:=(1)$ if $u=0$. Let $\mathcal{P}_{X}$  be the collection of primes $\mathfrak{p}\notin M_{K,\infty}$ such that either $\mathcal{N}_{K}(\mathfrak{p})\leq \max\{\beta,c_{2}\}$, $X_{\mathfrak{p}}$ is not geometrically irreducible, or $\mathfrak{p}\in \{\mathfrak{p}_{1},\ldots ,\mathfrak{p}_{u}\}$, namely $\mathcal{P}_{X}=\mathcal{P}_{f}\cup \{\mathfrak{p}:\mathcal{N}_{K}(\mathfrak{p})\leq \max\{\beta,c_{2}\}\}\cup\{\mathfrak{p}_{1},\ldots ,\mathfrak{p}_{u}\}$.  Then there is some non-zero ideal $\mathfrak{I}$, relative prime with all the primes lying in $\mathcal{P}_{X}$, such that $\mathfrak{I}|\Delta$, and 
$$\log(\mathcal{N}_{K}(\mathfrak{I}))\geq \frac{n!^{\frac{1}{n}}}{n+1}s^{1+\frac{1}{n}}(\log(s)-O_{n, K}(1)-n(\log(\beta)+\max\{\log(\log(\mathcal{N}_{K}(\mathfrak{q}))),0\}+\log b(f))).$$
\label{global determinant method lower bound}
\end{theorem}

\begin{proof}
We fix $\mathcal{O}_{K}^{n+1}$-coordinates for the tuple $(\boldsymbol \xi_{1},\ldots ,\boldsymbol \xi_{s})$ satisfying the conditions of Proposition \ref{Serre}.  Let us suppose that all prime $\mathfrak{p}$ with $\mathcal{N}_{K}(\mathfrak{p})\leq s^{\frac{1}{n}}$ is contained in $\mathcal{P}_{X}$. Then by \eqref{landau1} and Lemma \ref{divisors of a fixed ideal} we have
\begin{align*}
\frac{1}{n}\log(s)+O_{K}(1)=\sum_{\substack{\mathfrak{p}\\ \mathcal{N}_{K}(\mathfrak{p})\leq s^{\frac{1}{n}}}}\frac{\log(\mathcal{N}_{K}(\mathfrak{p}))}{\mathcal{N}_{K}(\mathfrak{p})} & \leq \sum_{\substack{\mathfrak{p}\\ \mathcal{N}_{K}(\mathfrak{p})\leq \max\{\beta,c_{2}\}}}\frac{\log(\mathcal{N}_{K}(\mathfrak{p}))}{\mathcal{N}_{K}(\mathfrak{p})}+\sum_{\substack{\mathfrak{p}\\ \mathfrak{p}\in \mathcal{P}_{f}}}\frac{\log(\mathcal{N}_{K}(\mathfrak{p}))}{\mathcal{N}_{K}(\mathfrak{p})}+\sum_{\substack{\mathfrak{p}\\ \mathfrak{p}| \mathfrak{q}}}\frac{\log(\mathcal{N}_{K}(\mathfrak{p}))}{\mathcal{N}_{K}(\mathfrak{p})}\\ & \leq \log(\beta)+\log(b(f))+\max\{\log(\log(\mathcal{N}_{K}(\mathfrak{q}))),0\}+O_{K}(1).
\end{align*}
Thus, if the term $O_{n,K}(1)$ is large enough, we may take $\mathfrak{I}=(1)$, since the right hand side of the bound in Theorem \ref{global determinant method lower bound} is negative.

Hence we may suppose that there is some prime $\mathfrak{p}\notin \mathcal{P}_{X}$ with $\mathcal{N}_{K}(\mathfrak{p})\leq s^{\frac{1}{n}}$. Applying Theorem \ref{local determinant2} for each prime $\mathfrak{p}\notin \mathcal{P}_{X}$ with $\mathcal{N}_{K}(\mathfrak{p})\leq s^{\frac{1}{n}}$, we obtain  $\prod_{\mathfrak{p}\notin \mathcal{P}_{X}:\mathcal{N}_{K}(\mathfrak{p})\leq s^{\frac{1}{n}}}\mathfrak{p}^{e_{\mathfrak{p}}}|\Delta$. Let $\mathfrak{I}:=\prod_{\mathfrak{p}\notin \mathcal{P}_{X}:\mathcal{N}_{K}(\mathfrak{p})\leq s^{\frac{1}{n}}}\mathfrak{p}^{e_{\mathfrak{p}}}$. Then the Landau estimate \eqref{landau2} yields

\begin{align}
\log(\mathcal{N}_{K}(\mathfrak{I})) & = \sum_{\substack{\mathfrak{p}\notin \mathcal{P}_{X}\\ \mathcal{N}_{K}(\mathfrak{p})\leq s^{\frac{1}{n}}}}e_{\mathfrak{p}}\log(\mathcal{N}_{K}(\mathfrak{p})) 
 \geq \frac{n!^{\frac{1}{n}}n}{n+1}s^{1+\frac{1}{n}}\sum_{\substack{\mathfrak{p}\notin \mathcal{P}_{X}\\ \mathcal{N}_{K}(\mathfrak{p})\leq s^{\frac{1}{n}}}}\frac{\log(\mathcal{N}_{K}(\mathfrak{p}))}{\mathcal{N}_{K}(\mathfrak{p})+O_{n}(d^{2}\mathcal{N}_{K}(\mathfrak{p})^{\frac{1}{2}})}-O_{n}(s)\sum_{\substack{\mathfrak{p}\\ \mathcal{N}_{K}(\mathfrak{p})\leq s^{\frac{1}{n}}}}\log(\mathcal{N}_{K}(\mathfrak{p}))\nonumber\\ & \geq \frac{n!^{\frac{1}{n}}n}{n+1}s^{1+\frac{1}{n}}\sum_{\substack{\mathfrak{p}\notin \mathcal{P}_{X}\\ \mathcal{N}_{K}(\mathfrak{p})\leq s^{\frac{1}{n}}}}\frac{\log(\mathcal{N}_{K}(\mathfrak{p}))}{\mathcal{N}_{K}(\mathfrak{p})+O_{n}(d^{2}\mathcal{N}_{K}(\mathfrak{p})^{\frac{1}{2}})}-O_{n,K}(s^{1+\frac{1}{n}}). \nonumber
\end{align} 
The inequality $\frac{1}{x+O_{n}(d^{2}x^{\frac{1}{2}})}\geq \frac{1}{x}-O_{n}(d^{2})\frac{1}{x^{\frac{3}{2}}}$ for all $x\geq 0$, the Landau estimates \eqref{landau1}, \eqref{landau2}, and \eqref{landau3}, and Lemma \ref{divisors of a fixed ideal},  allow us to deduce

\begin{align*}
 \sum_{\substack{\mathfrak{p}\notin \mathcal{P}_{X}\\ \mathcal{N}_{K}(\mathfrak{p})\leq s^{\frac{1}{n}}}} & \frac{\log(\mathcal{N}_{K}(\mathfrak{p}))}{\mathcal{N}_{K}(\mathfrak{p})+O_{n}(d^{2}\mathcal{N}_{K}(\mathfrak{p})^{\frac{1}{2}})}  \geq \sum_{\substack{\mathfrak{p}\\ \mathcal{N}_{K}(\mathfrak{p})\leq s^{\frac{1}{n}}}}\frac{\log(\mathcal{N}_{K}(\mathfrak{p}))}{\mathcal{N}_{K}(\mathfrak{p})}-\sum_{\mathfrak{p}\in \mathcal{P}_{X}}\frac{\log(\mathcal{N}_{K}(\mathfrak{p}))}{\mathcal{N}_{K}(\mathfrak{p})}-O_{n}(d^{2})\sum_{\substack{\mathfrak{p}\notin \mathcal{P}_{X}\\ \mathcal{N}_{K}(\mathfrak{p})\leq s^{\frac{1}{n}}}} \frac{\log(\mathcal{N}_{K}(\mathfrak{p}))}{\mathcal{N}_{K}(\mathfrak{p})^{\frac{3}{2}}}\\ & \geq \frac{\log(s)}{n}-O_{K}(1)-\! \! \! \! \! \! \! \! \! \! \! \! \! \! \! \! \! \! \smashoperator[r]{\sum_{\substack{\mathfrak{p}\\ \mathcal{N}_{K}(\mathfrak{p})\leq \max\{\beta,c_{2}\}}}}\frac{\log(\mathcal{N}_{K}(\mathfrak{p}))}{\mathcal{N}_{K}(\mathfrak{p})}-\sum_{\mathfrak{p}|\mathfrak{q}}\frac{\log(\mathcal{N}_{K}(\mathfrak{p}))}{\mathcal{N}_{K}(\mathfrak{p})}-\log b(f)-O_{n}(d^{2})\! \! \! \!  \! \! \! \smashoperator[r]{\sum_{\substack{\mathfrak{p}\notin \mathcal{P}_{X}\\ \mathcal{N}_{K}(\mathfrak{p})\leq s^{\frac{1}{n}}}}} \frac{\log(\mathcal{N}_{K}(\mathfrak{p}))}{\mathcal{N}_{K}(\mathfrak{p})^{\frac{3}{2}}}\\ & \geq \frac{\log(s)}{n}-\log(\beta)-\max\{\log\log(\mathcal{N}_{K}(\mathfrak{q})),0\}-\log b(f)-O_{K}(1)-O_{n}(d^{2})\!\!\! \! \smashoperator[r]{\sum_{\substack{\mathfrak{p}\\ \mathcal{N}_{K}(\mathfrak{p})>\beta}}} \frac{\log(\mathcal{N}_{K}(\mathfrak{p}))}{\mathcal{N}_{K}(\mathfrak{p})^{\frac{3}{2}}}\\ & \geq \frac{\log(s)}{n}-\log(\beta)-\max\{\log(\log(\mathcal{N}_{K}(\mathfrak{q}))),0\}-\log b(f)-O_{K}(1)+O_{n,K}(d^{2}\beta^{-\frac{1}{2}})\\ & \geq  \frac{\log(s)}{n}-\log(\beta)-\max\{\log(\log(\mathcal{N}_{K}(\mathfrak{q}))),0\}-\log b(f)-O_{n,K}(1).
\end{align*}
 
The bound in Theorem \ref{global determinant method lower bound} follows.
\end{proof}

\begin{remark}
When $\mathfrak{q}=(1)$, Theorem \ref{global determinant method lower bound} gives the bounds in \cite[Proposition 3.2.26]{Cluckers} and \cite[Proposition 2.4]{Vermeulen} in the cases $K=\mathbb{Q}$ and $K=\mathbb{F}_{q}(T)$, respectively. 
\end{remark}

In \cite{Walsh3} Walsh proved the bound $b(f)\lesssim_{d}\max\{\log(H_{K,\text{aff}}(f)),1\}$ for $K=\mathbb{Q}$. In order to keep track of the dependence on the degree of $f$, in \cite[Corollary 3.2.3]{Cluckers} and \cite[Lemma 2.3]{Vermeulen} the authors gave an effective bound of $b(f)$ for $K=\mathbb{Q}$ and $K=\mathbb{F}_{q}(T)$, respectively. We follow their strategy to give an effective bound for global fields. We will use the following notation

\begin{definition}
Let $d\in \mathbb{N}$, and let $K$ be a global field. Then for any $a_{1},a_{2},a_{3},a_{4}\in \mathbb{R}$ we denote
$$[a_{1},a_{2},a_{3},a_{4}]:=\begin{cases} a_{1} & \text{ if char}(K)=0,\\ a_{2} & \text{ if }0<\text{char}(K)\leq d(d-1),\\ a_{3} & \text{ if }d(d-1)<\text{char}(K)\leq \max\{27d^{4},c_{2}\},\\ a_{4} & \text{ if char}(K)>\max\{27d^{4},c_{2}\}.\end{cases}$$
\label{multicharacteristic}
\end{definition}

\begin{lemma}
Let $f=\sum_{I}a_{I}\boldsymbol X^{I}\in \mathcal{O}_{K}[X_{0},\ldots ,X_{n+1}]$ be an absolutely irreducible homogeneous polynomial of degree $d\geq 2$.  We have
\begin{equation*}
b(f)\lesssim_{K,n} \max\{d^{[-2,\frac{4}{3},-\frac{8}{3},2]}\log(H_{K,\emph{\text{aff}}}(f)),1\}.
\end{equation*}
\label{b(f)}
\end{lemma}

\begin{proof}
We will use Noether forms as they are presented in \cite{Ruppert, Kaltofen, Gao}. Let $\mathcal{P}_{f}$ be as in \eqref{p(f)}. First, let us suppose that $K$ is a number field. By \cite[Satz 4]{Ruppert}, there is a homogeneous form $\Phi$ with coefficients in $\mathbb{Z}$, of degree $d^{2}-1$ and 
\begin{equation}
\ell_{1}(\Phi)\leq d^{3(d^{2}-1)}\left[ {n+d\choose n}3^{d} \right]^{d^{2}-1},
\label{bound for the norm}
\end{equation} 
such that $\Phi$ applied to the coefficients of $f$ is non-zero, but is divisible by any prime in $\mathcal{P}_{f}$. In particular, by inequalities \eqref{norm}, \eqref{auxiliar remark} we have 
\begin{align}
\prod_{\mathfrak{p}\in \mathcal{P}_{f}}\mathcal{N}_{K}(\mathfrak{p})\leq \mathcal{N}_{K}(\Phi(a_{I})_{I})\leq H_{K}(\Phi(a_{I})_{I}) & \leq  \ell_{1}(\Phi)^{d_{K}^{2}} H_{K}(1:(a_{I})_{I})^{\deg(\Phi)} = \ell_{1}(\Phi)^{d_{K}^{2}} H_{K,\text{aff}}(f)^{\deg(\Phi)}. \label{no geo integro1} 
\end{align}
Let $c:=\ell_{1}(\Phi)^{d_{K}^{2}} H_{K,\text{aff}}(f)^{\deg(\Phi)}$. Using the above inequality and \eqref{landau1}, we have
\begin{align*}
\log (b(f)) & =\sum_{\mathfrak{p}\in \mathcal{P}_{f}}\frac{\log(\mathcal{N}_{K}(\mathfrak{p}))}{\mathcal{N}_{K}(\mathfrak{p})}\leq \sum_{\mathclap{\substack{\mathfrak{p}\\ \max\{27d^{4},c_{2}\}<\mathcal{N}_{K}(\mathfrak{p})\leq \log c}}}\frac{\log(\mathcal{N}_{K}(\mathfrak{p}))}{\mathcal{N}_{K}(\mathfrak{p})}+\sum_{\substack{\mathfrak{p}\in \mathcal{P}_{f}\\ \log c<\mathcal{N}_{K}(\mathfrak{p})}}\frac{\log(\mathcal{N}_{K}(\mathfrak{p}))}{\log c}\\ & \leq \max\{\log (\log c)-4\log d,0\}+O_{K}(1)+\frac{\log c}{\log c}\\ & \leq \max\{\log(\deg(\Phi)\log(H_{K,\text{aff}}(f)))-4\log d,\log (d_{K}^{2}\log \ell_{1}(\Phi))-4\log d\}+O_{K}(1).
\end{align*}
Using that $\deg(\Phi)= d^{2}-1$ and \eqref{bound for the norm}, the result follows in the number field case. Indeed, note that
\begin{align}
\log(d_{K}\deg(\Phi)\log(H_{K,\text{aff}}(f)))-4\log d &\leq \log(d_{K})+\log (d^{2}\log(H_{K,\text{aff}}(f)))-4\log d\label{b(f)1}\\ & \leq \log(d_{K})+ \log (\log(H_{K,\text{aff}}(f)))-2\log d. \nonumber
\end{align}
Meanwhile,
\begin{align}
\log(d_{K}^{2}\log \ell_{1}(\Phi)) & \leq \log \left(d_{K}^{2}\log \left( d^{3(d^{2}-1)}\left[ {n+d\choose n}3^{d} \right]^{d^{2}-1} \right)\right)\label{b(f)2} \\ & \leq  2\log(d_{K})+2\log(d)+\log\left( 3\log(d)+d\log(3)+d\log(n+d)-\log(d!)\right).\nonumber
\end{align}
Now we bound $\log (d_{K}^{2}\log\ell_{1}(\Phi))-4\log d$. Since
\begin{align*}
\frac{3\log (d)+d\log(3)+d\log(n+d)-\log(d!)}{d^{2}}\lesssim_{n}1,
\end{align*}
from \eqref{b(f)2} we obtain $\log(d_{K}^{2}\log\ell_{1}(\Phi))-4\log(d)\lesssim_{n,K}1$. This together with \eqref{b(f)1} gives the bound for $b(f)$ in the number field case.

Now let us suppose that $K$ is a function field over $\mathbb{F}_{q}(T)$. We will follow \cite{Vermeulen}, and consider separately  the cases of small and large characteristic.

Let $K$ be with $0<\text{char}(K)\leq d(d-1)$. By \cite[Theorem 7]{Kaltofen}, there is a homogeneous form $\Phi$ with coefficients in $\mathbb{Z}$, of degree $12d^{6}$ such that $\Phi$ applied to the coefficients of $f$ is non-zero, but divisible by any prime in $\mathcal{P}_{f}$. On the other hand, if $\text{char}(K)>d(d-1)$, by \cite[Lemma 2.4]{Gao}, the result in \cite[Satz 4]{Ruppert} still holds in the function field case and then we may take $\Phi$ of degree $d^{2}-1$. In any case, since the coefficients of $\Phi$ are integers, their height in $K$ is $1$, then $\ell_{1}(\Phi)=1$. Then by inequalities \eqref{polynomialheight2}, \eqref{norm}, \eqref{inequality between norms2} we have 
\begin{equation}
\prod_{\mathfrak{p}\in \mathcal{P}_{f}}\mathcal{N}_{K}(\mathfrak{p})\leq \mathcal{N}_{K}(\Phi(a_{I})_{I})\leq H_{K}(\Phi(a_{I})_{I})\leq \begin{cases}H_{K,\text{aff}}(f)^{12d^{6}}\text{ if }0<\text{char}(K)\leq d(d-1),\\ H_{K,\text{aff}}(f)^{d^{2}-1}\text{ if char}(K)> d(d-1).\end{cases}
\label{no geo integro2}
\end{equation}  
Now, denoting $c:=H_{K,\text{aff}}(f)^{\deg(\Phi)}$, by \eqref{landau1} and \eqref{landau2} it follows that
\begin{align*}
\log (b(f)) & =\sum_{\mathfrak{p}\in \mathcal{P}_{f}}\frac{\log(\mathcal{N}_{K}(\mathfrak{p}))}{\mathcal{N}_{K}(\mathfrak{p})}\leq \sum_{\mathclap{\substack{\mathfrak{p}\\ \max\{\beta,c_{2}\}<\mathcal{N}_{K}(\mathfrak{p})\leq \log c}}}\frac{\log(\mathcal{N}_{K}(\mathfrak{p}))}{\mathcal{N}_{K}(\mathfrak{p})}+\sum_{\substack{\mathfrak{p}\in \mathcal{P}_{f}\\ \log c<\mathcal{N}_{K}(\mathfrak{p})}}\frac{\log(\mathcal{N}_{K}(\mathfrak{p}))}{\log c}\\ & \leq \max\{\log (\deg(\Phi)(\log H_{K,\text{aff}}(f)))-\log(\beta),0\}+O_{K}(1)\\ & \leq \max\{\log(\deg(\Phi))+\log(\log(H_{K,\text{aff}}(f)))-\log(\beta),0\}+O_{K}(1).
\end{align*}
From this last estimate it follows the bounds for function fields.
\end{proof}

\section{An adequate change of coordinates}
\label{section change of variables}

For technical reasons that will appear in the next section, given two non-zero polynomials $f,h\in \mathcal{O}_{K}[X_{0},\ldots ,X_{n+1}]$, we will need to obtain a lower bound for the quantity $H_{K}(fh)$ in terms of $H_{K}(f)$ and $H_{K}(h)$. This turns out to be rather easy in the function field case. Indeed, suppose that $K$ is a function field. Then every place $v\in M_{K}$ is non-archimedean. Thus, by Gauss' lemma, (e.g. see \cite[Lemma 1.6.3]{Bombieri}), it holds
\begin{equation}
H_{K}(fh)=H_{K}(f)H_{K}(h).
\label{lower bound for polynomials function field case}
\end{equation}

In the number field case, while it is possible to bound $H_{K}(fh)$ in terms of $H_{K}(f)$ and $H_{K}(h)$, this bound will necessarily depend on the degrees of $f$ and $g$, which is not good enough for our purposes (see Remark \ref{remark sobre la cota}). In Claim \ref{lower bound3.6} we will prove that, if $f$ is of an adequate form, it is possible to bound from below $H_{K}(fh)$ in terms of $H_{K}(f)$ and of the degree of $f$. By making an adequate change of variables as in \cite[$\S 3$]{Walsh3}, in this section we will show that any polynomial can be taken into a polynomial that verifies the hypothesis of Claim \ref{lower bound3.6}.

Let us suppose that $K$ is a number field. Let $f=\sum_{I}c_{I}\boldsymbol X^{I}\in \mathcal{O}_{K}[X_{0},\ldots ,X_{n+1}]$ be a homogeneous polynomial of degree $d$. Let $a_{0},\ldots ,a_{n}\in \mathcal{O}_{K}$ and consider the new variables $X_{i}':=X_{i}-a_{i}X_{n+1}$ for all $i=0,\ldots ,n+1$. Then the linear change of variables gives
\begin{align*}
 \tilde{f}(X_{0},\ldots ,X_{n+1}) & :=f(X_{0}',\ldots ,X_{n}',X_{n+1}) =\sum_{i_{0}+\cdots +i_{n+1}=d}c_{i_{0},\ldots ,i_{n+1}}(X_{0}-a_{0}X_{n+1})^{i_{0}}\cdots (X_{n}-a_{n}X_{n+1})^{i_{n}}X_{n+1}^{i_{n+1}}\\ & =\pm\left(\sum_{i_{0}+\ldots +i_{n+1}=d}c_{i_{0},\ldots ,i_{n+1}}a_{0}^{i_{0}}\cdots a_{n}^{i_{n}}\right)X_{n+1}^{d}+\cdots =\pm f(a_{0},\ldots ,a_{n},1)X_{n+1}^{d}+\cdots .
\end{align*}
So, the coefficient of $X_{n+1}^{d}$ in $\tilde{f}$ is $\pm f(a_{0},\ldots ,a_{n},1)$. Next, we will require that this coefficient is large enough. More precisely, we shall prove the following generalization of \cite[$\S$ 3]{Walsh3} and \cite[Lemma 3.4.2]{Cluckers}.

\begin{lemma}
Let $K$ be a number field and let $f=\sum_{I}c_{I}\boldsymbol X^{I}\in \mathcal{O}_{K}[X_{0},\ldots ,X_{n+1}]$ be a homogeneous polynomial of degree $d$. There exist integers $a_{0},\ldots ,a_{n}$ with $0\leq a_{i}\leq d$ for all $i$ such that
$$\prod_{v\in M_{K,\infty}}|f(a_{0},\ldots ,a_{n},1)|_{v}\geq 3^{-\frac{(n+1)d}{d_{K}}}2^{-d}\prod_{v\in M_{K,\infty}}|f|_{v}.$$
\label{an adequate integer tuple}
\end{lemma}

\begin{proof}
Given an embedding $\sigma:K\to \mathbb{C}$, we define $\sigma(f):=\sum_{I}\sigma(c_{I})\boldsymbol X^{I}$ and $f^{G}=\prod_{\sigma}\sigma(f)$ where the product runs over all the embeddings of $K$ in $\mathbb{C}$. By \eqref{definition of arq places} it holds 
\begin{equation}
\prod_{v\in M_{K,\infty}}|f(a_{0},\ldots ,a_{n},1)|_{v}=|f^{G}(a_{0},\ldots ,a_{n},1)|^{\frac{1}{d_{K}}}.
\label{sigma bound}
\end{equation}
By \cite[Lemma 3.4.2]{Cluckers}, there exist integers $a_{0},\ldots ,a_{n}$ with $0\leq a_{i}\leq d$ for all $i$ such that
\begin{equation*}
|f^{G}(a_{0},\ldots ,a_{n},1)|\geq 3^{-(n+1)d}\ell_{\infty}(f^{G}).
\end{equation*}
Since $f^{G}$ is a polynomial of degree $d_{K}d$, by \cite[Lemma 1.6.11]{Bombieri} we have
\begin{equation*}
|f^{G}(a_{0},\ldots ,a_{n},1)|\geq 3^{-(n+1)d} 2^{-dd_{K}}\prod_{\sigma}\ell_{\infty}(\sigma(f)).
\end{equation*}
If $v$ is the place corresponding to $\sigma$, we have $\ell_{\infty}(\sigma(f))=|f|_{v}^{\frac{d_{K}}{n_{v}}}$. Then, since the complex embeddings come in conjugate pairs, both corresponding to the same place, it follows that
\begin{equation}
|f^{G}(a_{0},\ldots, a_{n},1)|\geq 3^{-(n+1)d} 2^{-dd_{K}}\left(\prod_{v\in M_{K,\infty}}|f|_{v}\right)^{d_{K}}.
\label{sigma bound2}
\end{equation}
Combining \eqref{sigma bound} and \eqref{sigma bound2} yields the desired result.
\end{proof}
From Lemma \ref{an adequate integer tuple} we find integers $a_{0},\ldots ,a_{n}$ with $0\leq a_{i}\leq d$ for all $i$, such that $\tilde{f}(X_{0},\ldots ,X_{n},X_{n+1})=f(X_{0}',\ldots ,X_{n}',X_{n+1})=c_{\tilde{f}}X_{n+1}^{d}+\cdots $, with $\prod_{v\in M_{K,\infty}}|c_{\tilde{f}}|_{v}\geq 3^{-\frac{(n+1)d}{d_{K}}}2^{-d}\prod_{v\in M_{K,\infty}}|f|_{v}$. Since for all $v\in M_{K,\infty}$, it holds
\begin{equation}
|\tilde{f}|_{v}\leq \left| n+d+1\choose n+1 \right|_{v}|d|_{v}^{d}|f|_{v}\leq |2(n+1)^{d}d^{d}|_{v}|f|_{v},
\label{comparison of f and f1}
\end{equation}
we conclude that
\begin{equation}
\prod_{v\in M_{K,\infty}}|c_{\tilde{f}}|_{v}\geq 3^{-\frac{(n+1)d}{d_{K}}} 2^{-(d+1)}(n+1)^{-d}d^{-d}\prod_{v\in M_{K,\infty}}|\tilde{f}|_{v}\geq C^{-nd^{1+\frac{1}{n}}}\prod_{v\in M_{K,\infty}}|\tilde{f}|_{v},
\label{important bound}
\end{equation} 
where $C$ is a positive constant depending on $K$ and $n$. The bound is not sharp, but having it in this form will be convenient for our purposes. Now, if $\mathfrak{p}$ is a prime such that $f$ is $\mathfrak{p}$-primitive, it also follows that $\tilde{f}$ is $\mathfrak{p}$-primitive. Furthermore, if $\tilde{f}=\sum_{J}b_{J}\boldsymbol X^{J}$, the proof of Proposition \ref{Serre} shows that there exists a constant $c_{1}=c_{1}(K)$ and a unit $\varepsilon\in \mathcal{O}_{K}^{\times}$ such that $\varepsilon \tilde{f}$ satisfies 
$$H_{K}(1:(\varepsilon b_{J})_{J})\leq c_{1}^{d_{K}}H_{K}(\varepsilon \tilde{f}).$$
Since the $\mathfrak{p}$-adic valuation of any unit is $0$, and by the product formula \eqref{product formula} it holds  $\prod_{v\in M_{K,\infty}}|\varepsilon|_{v}=1$, we see that the coefficient of $X_{n+1}^{d}$ in $f_{1}:=\varepsilon \tilde{f}$ also satisfies \eqref{important bound}. Moreover, from the identity $f(X_{0},\ldots ,X_{n+1})=\varepsilon^{-1}f_{1}(X_{0}+a_{0}X_{n+1},\ldots ,X_{n}+a_{n}X_{n+1},X_{n+1})$, proceeding as in \eqref{comparison of f and f1}, we have that for all $v\in M_{K,\infty}$
\begin{equation}
|f|_{v}\leq |\varepsilon|_{v}^{-1} |2(n+1)^{d}d^{n+1}|_{v}|f_{1}|_{v}.
\label{comparison of f1 and f}
\end{equation}
Also, for all $v\in M_{K,\text{fin}}$, it holds
\begin{equation}
|f|_{v}=|f_{1}|_{v},
\label{comparison of f and f1 2}
\end{equation}
hence \eqref{comparison of f1 and f} and \eqref{comparison of f and f1 2} imply
\begin{equation*}
H_{K}(f)\leq \prod_{v\in M_{K,\infty}}|\varepsilon|_{v}^{-1}|2(n+1)^{d}d^{n+1}|_{v}|f_{1}|_{v}\prod_{v\in M_{K,\text{fin}}}|f_{1}|_{v}\leq C'^{d^{1+\frac{1}{n}}}H_{K}(f_{1}), 
\end{equation*}
for some constant $C'$ depending only on $n$. We summarize all this in the next proposition.

\begin{proposition}
Let $K$ be a number field, and let $f=\sum_{I}a_{I}\boldsymbol X^{I}\in \mathcal{O}_{K}[X_{0},\ldots ,X_{n+1}]$ be an absolutely irreducible homogeneous polynomial of degree $d\geq 2$, which is $\mathfrak{p}$-primitive for all primes $\mathfrak{p}$ with $\mathcal{N}_{K}(\mathfrak{p})>c_{2}$ and
\begin{equation}
H_{K,\emph{\text{aff}}}(f)\leq c_{1}^{d_{K}}H_{K}(f).
\label{auxiliar b}
\end{equation}
Then there are integers $a_{0},\ldots ,a_{n}$ with $0\leq a_{i}\leq d$ and a unit $\varepsilon\in \mathcal{O}_{K}^{\times}$ such that $f_{1}(X_{0},\ldots ,X_{n},X_{n+1}):=\varepsilon f(X_{0}-a_{0}X_{n+1},\ldots ,X_{n}-a_{n}X_{n+1},X_{n+1})$ is a polynomial with coefficient $c_{f_{1}}$ in $X_{n+1}^{d}$ verifying
$$\prod_{v\in M_{K,\infty}}|c_{f_{1}}|_{v}\geq C^{-nd^{1+\frac{1}{n}}}\prod_{v\in M_{K,\infty}}|f_{1}|_{v}\qquad \text{ and }\qquad H_{K}(f)\leq C'^{d^{1+\frac{1}{n}}}H_{K}(f_{1})$$
for some positive constants $C=C(n,K)$, $C'=C(n)$. Furthermore, $f_{1}$ verifies \eqref{auxiliar b} and it holds $b(f)=b(f_{1})$.
\label{auxiliar change}
\end{proposition}

\section{Construction of the auxiliary hypersurface and bounds for the number of rational points of curves over global fields} 
\label{Section 5}

This section is divided in two parts, the first one dealing with the proof of Theorem \ref{teorema1} extending and improving  the bound of Heath-Brown \cite{Heath-Brown} and the second one dealing with Theorem \ref{teorema1.5} extending and improving the bounds of Bombieri and Pila. 

The main technical tool in the first part is Theorem \ref{polynomial method2} (which is a generalization of \cite[Theorem 1.3]{Walsh3}, \cite[Theorem 3.1.1]{Cluckers} and \cite[Theorem 3.1]{Vermeulen} to global fields) where we construct an auxiliary hypersurface of bounded degree that vanishes on the rational points of the hypersurface $X$ of prescribed height. From Theorem \ref{polynomial method2} and B\'ezout's theorem, Theorem \ref{teorema1} follows for plane projective curves. By means of a projection argument used in \cite{P2} we deduce the general form of Theorem \ref{teorema1}. In the proof of Theorem \ref{polynomial method2}, the change of variables done in Section \ref{section change of variables} will allow us to assume that in the number field case  the polynomial defining $X$ has its height concentrated on a prescribed coefficient. Unlike \cite{Vermeulen}, in the function field case the proof presented here does not need a change of variables. As in \cite{Walsh3, Cluckers, Vermeulen}, the bound on the degree of the auxiliary hypersurface improves as the height of the polynomial defining $X$ gets larger. A difference with \cite{Walsh3, Cluckers, Vermeulen} is that the hypersurface we construct vanishes on the rational points of $X$ of bounded height and prescribed reduction modulo $\mathfrak{p}$ for many primes $\mathfrak{p}$, and this produces a saving in the degree of the hypersurface, which will be needed in Section \ref{Section 6}.  

The proof of Theorem \ref{teorema1.5} follows the same strategy of the proof of Theorem \ref{teorema1}, the main technical tool being Theorem \ref{ell-venk}, which is an adaptation of Theorem \ref{polynomial method2} to affine hypersurfaces. In the proof of Theorem \ref{ell-venk} we follow the strategy devised in \cite[Remark 2.3]{Ellenberg} and developed in \cite[Proposition 4.2.1]{Cluckers}. 

\subsection{The projective case}\label{the projective case}
We proceed to generalize \cite[Theorem 1.3]{Walsh3}, \cite[Theorem 3.1.1]{Cluckers} and \cite[Theorem 3.1]{Vermeulen} to global fields. Moreover, this generalization also improves \cite[Theorem 2.2]{Salberger} when $(B_{0},\ldots, B_{n+1})=(B,\ldots ,B)$ by removing a logarithmic factor. 
We will require the following notation.

\begin{notation}
Let $X\subseteq \mathbb{P}_{K}^{n+1}$ be a projective variety. Given $B\in \mathbb{R}_{>0}$, we will denote
$$X(K,B):=\left\{ \boldsymbol x\in X(K):H_{K}(\boldsymbol x)\leq B\right\},$$
Moreover, if $\{\mathfrak{p}_{1},\ldots ,\mathfrak{p}_{u}\}$ is a subset of primes such that for all $1\leq i\leq u$ we let $P_{i}$ be a non-singular $\mathbb{F}_{\mathfrak{p}_{i}}$--point on $X_{\mathfrak{p}_{i}}$, we denote
$$X(K,B;P_{1},\ldots ,P_{u}):=\{\boldsymbol x\in X(K,B):\boldsymbol x\text{ specialises to }P_{i} \text{ in }X_{\mathfrak{p}_{i}}\text{ for all }i\}.$$
We denote 
$$N(X,K,B):=|X(K,B)|.$$
\end{notation}

\begin{theorem}
Let $K$ be a global field of degree $d_{K}$. Let $X\subseteq \mathbb{P}_{K}^{n+1}$ be an integral hypersurface of degree $d\geq 2$, defined by an irreducible homogeneous polynomial $f\in \mathcal{O}_{K}[X_{0},\ldots ,X_{n+1}]$ of degree $d$. Let $\{\mathfrak{p}_{1},\ldots ,\mathfrak{p}_{u}\}$ be (a possibly empty) subset of primes and let $P_{i}$ be a non-singular $\mathbb{F}_{\mathfrak{p}_{i}}$--point on $X_{\mathfrak{p}_{i}}$ for each $i\in \{1,\ldots ,u\}$. Set $\mathfrak{q}:=\prod_{i=1}^{u}\mathfrak{p}_{i}$ if $u\geq 1$ and $\mathfrak{q}:=(1)$ if $u=0$.
Then for any $B\geq 1$, there exists a homogeneous $g\in \mathcal{O}_{K}[X_{0},\ldots ,X_{n+1}]$ of degree 
\begin{align*}
M\lesssim_{K,n}& B^{\frac{n+1}{nd^{\frac{1}{n}}}}\frac{d^{[4,\frac{14}{3},\frac{14}{3},0]-\frac{1}{n}}b(f)\max\{\log(\mathcal{N}_{K}(\mathfrak{q}),1\})}{H_{K}(f)^{\frac{1}{n}\frac{1}{d^{1+\frac{1}{n}}}}\mathcal{N}_{K}(\mathfrak{q})}+d^{1-\frac{1}{n}}\log(B\mathcal{N}_{K}(\mathfrak{q}))+d^{2-\frac{1}{n}}\log(\mathcal{N}_{K}(\mathfrak{q}))\\ & +d^{[4-\frac{1}{n},7,\frac{14}{3}-\frac{1}{n},3]}\left(\frac{\max\{\log(\mathcal{N}_{K}(\mathfrak{q})),1\}}{\mathcal{N}_{K}(\mathfrak{q})}+1\right),
\end{align*}
not divisible by $f$ and vanishing on all $X(K,B;P_{1},\ldots ,P_{u})$.
\label{polynomial method2} 
\end{theorem}

\begin{remark}
Observe that if a polynomial $f$ is irreducible but not absolutely irreducible over $K$, then by the argument in \cite[Corollary 1]{Heath-Brown}, there exists a homogeneous polynomial $g\in \mathcal{O}_{K}[X_{0},\ldots ,X_{n+1}]$ of degree $d$, not divisible by $f$ and vanishing on all $K$-rational points of $X$. Thus, in the proof of Theorem \ref{polynomial method2} we may assume that $f$ is absolutely irreducible.
\label{abs irred assumption}
\end{remark}

In order to prove Theorem \ref{polynomial method2}, first we will prove a seemingly weaker version of it where the polynomial $f$ defining the hypersurface $X$ satisfies the following assumption, afterwards the work done in Section \ref{section change of variables} will be used to deduce the general case.

\begin{assumption}
$f\in \mathcal{O}_{K}[X_{0},\ldots ,X_{n+1}]$ is an absolutely irreducible homogeneous polynomial of degree $d$ that verifies:
\begin{enumerate}
\item $f$ is $\mathfrak{p}$-primitive for any prime $\mathfrak{p}$ of $\mathcal{O}_{K}$ with $\mathcal{N}_{K}(\mathfrak{p})>c_{2}$;
\item $H_{K,\emph{\text{aff}}}(f)\leq c_{1}^{d_{K}}H_{K}(f)$. \label{assumption2}
\item If $K$ is a number field, the coefficient $c_{f}$ of $X_{n+1}^{d}$ in $f$ verifies:
$$\prod_{v\in M_{K,\infty}}|c_{f}|_{v}\geq C^{-nd^{1+\frac{1}{n}}}\prod_{v\in M_{K,\infty}}|f|_{v},$$
where $C$ is a constant that depends only on $K$ and $n$.\label{assumption3} 
\end{enumerate}
\label{assumption}
\end{assumption} 
  
\begin{proof}[Proof of Theorem \ref{polynomial method2} assuming $f$ verifies Assumption \ref{assumption}]
The proof follows the usual strategy of the polynomial method: first we find a small characteristic subset $\mathcal{C}\subseteq X(K,B;P_{1},\ldots ,P_{u})$, which detects the algebraic structure of $X(K,B;P_{1},\ldots ,P_{u})$. Then, we use a dimensional argument (which usually is a variation of the Siegel's lemma) to construct an adequate polynomial $g$ of small complexity that vanishes on $\mathcal{C}$. From the nature of the subset $\mathcal{C}$, it will follow that $g$ vanishes on $X(K,B;P_{1},\ldots ,P_{u})$ and has the desired  property. 

In order to accomplish the above strategy, we let $M$ be a positive integer of size
\begin{align}
M\sim _{K,n}& B^{\frac{n+1}{nd^{\frac{1}{n}}}}\frac{d^{[4,\frac{14}{3},\frac{14}{3},0]-\frac{1}{n}}b(f)\max\{\log(\mathcal{N}_{K}(\mathfrak{q})),1\}}{H_{K}(f)^{\frac{1}{n}\frac{1}{d^{1+\frac{1}{n}}}}\mathcal{N}_{K}(\mathfrak{q})}+d^{1-\frac{1}{n}}\log(B\mathcal{N}_{K}(\mathfrak{q})) \label{choice of M}+d^{2-\frac{1}{n}}\log(\mathcal{N}_{K}(\mathfrak{q}))\\ & +d^{[4-\frac{1}{n},7,\frac{14}{3}-\frac{1}{n},3]}\left(\frac{\max\{\log(\mathcal{N}_{K}(\mathfrak{q})),1\}}{\mathcal{N}_{K}(\mathfrak{q})}+1\right). \nonumber
\end{align}

This integer will be the degree of the polynomial $g$ to be constructed. Now we construct the characteristic subset, which should be a subset such that any polynomial of degree $M$, vanishing on it also vanishes on $X(K,B;P_{1},\ldots ,P_{u})$, without vanishing completely on $X$. The construction is rather natural. Using Proposition \ref{Serre} we represent points in $X(K,B;P_{1},\ldots ,P_{u})$ with coordinates in $\mathcal{O}_{K}^{n+2}$. Afterwards, we take $ \mathcal{C}:=\{\boldsymbol \xi_{1},\ldots ,\boldsymbol \xi_{s}\}$ to be a maximal subset of $X(K,B;P_{1},\ldots ,P_{u})$ with the property that if $A$ is the matrix whose $i^{\text{th}}$ row is the evaluation of the different monomials of degree $M$ at $\boldsymbol \xi_{i}$, then $A$ has rank $s$. 

The next step is to construct an adequate polynomial by means of the Bombieri-Vaaler theorem. Given an integer $D$, we write $\mathcal{B}[D]$ for the set of monomials of degree $D$; it follows that $|\mathcal{B}[D]|={D+n+1\choose n+1}$. Since the elements of $f\mathcal{B}[M-d]$ provide linearly independent polynomials vanishing on $X(K,B;P_{1},\ldots ,P_{u})$, it follows that 
\begin{equation}
s=\text{rank}(A)\leq |\mathcal{B}[M]|-|\mathcal{B}[M-d]|.
\label{rank inequality0}
\end{equation}

Now, let us note that since $A$ has coefficients in $\mathcal{O}_{K}$, if the inequality \eqref{rank inequality0} is strict we can find a polynomial in the $\mathcal{O}_{K}$-span of $\mathcal{B}[M]$ which vanishes on $\mathcal{C}$ (so it also vanishes on $X(K,B;P_{1},\ldots ,P_{u})$, by the maximality of $\mathcal{C}$), and it is not divisible by $f$. This would conclude the proof. Then we may suppose that
\begin{equation}
s=|\mathcal{B}[M]|-|\mathcal{B}[M-d]|.
\label{rank inequality}
\end{equation}
The remaining of the proof is to show that \eqref{rank inequality} does not hold if the implicit constant in \eqref{choice of M} is large enough. 

Let us note that \eqref{rank inequality} implies
\begin{equation}
s=\frac{dM^{n}}{n!}+O_{n}(d^{2}M^{n-1}).
\label{estimate for s}
\end{equation}

From \eqref{estimate for s} and the fact that $M\gtrsim d^{2}$ we obtain
\begin{equation}
\log(s)=\log(d)+n\log(M)-O_{n}(1),
\label{estimate for the log of s}
\end{equation}
and
\begin{equation}
\frac{s^{\frac{1}{n}}}{M}=\frac{d^{\frac{1}{n}}}{n!^{\frac{1}{n}}}+O_{n}\left(\frac{d^{2}}{M}\right).
\label{estimate for the 1/n-rooth}
\end{equation}

Now, let $\Delta$ be the greatest common divisor in $\mathcal{O}_{K}$ of the determinants of the $s\times s$-minors of $A$ and let us denote 
\begin{equation*}
r:=|\mathcal{B}[M]|.
\end{equation*} 
To find an $\mathcal{O}_{K}$- homogeneous polynomial of degree $M$ that vanishes on $ \mathcal{C}$ amounts to find a non-trivial solution in $\mathcal{O}_{K}^{r}$ of the system of linear equations given by $A$. For that, using Theorem \ref{Bombieri-Vaaler2}, there exists a polynomial $g=\sum_{I}a_{I}\boldsymbol X^{I}\in \mathcal{O}_{K}[X_{0},\ldots ,X_{n+1}]$ of degree $M$, vanishing on $ \mathcal{C}$ and satisfying the inequality
\begin{equation}
H_{K}(g)^{r-s}\leq \begin{cases}C(K)^{(r-s)d_{K}}\mathcal{N}_{K}(\Delta)^{-1} \displaystyle \prod_{v\in M_{K,\infty}}|\det(A A^{\ast})|_{v}^{\frac{d_{K}}{2}} & \text{ if }K\text{ is a number field},\\ C(K)^{(r-s)d_{K}}\mathcal{N}_{K}(\Delta)^{-1} q^{-\deg(\mathfrak{p}_{\infty})\min_{|J|=s}\text{ord}_{\mathfrak{p}_{\infty}}(\det(A_{J}))} & \text{ if }K\text{ is a function field}.\end{cases}
\label{application of bombieri-vaaler}
\end{equation}
\begin{claim}
There is some positive constant $\tilde{c_{1}}=\tilde{c_{1}}(K)$ such that it holds the inequality
\begin{equation}
H_{K}(g)^{r-s}\leq C(K)^{(r-s)}\mathcal{N}_{K}(\Delta)^{-1}\tilde{c_{1}}^{Ms}s!^{\frac{d_{K}}{2}}r^{\frac{d_{K}s}{2}}B^{Ms}.
\label{application of bombieri-vaaler2}
\end{equation}
\label{claim 5.5}
\end{claim}
\begin{proof}[Proof of Claim \ref{claim 5.5}]
The proof amounts to bound the contribution of the places $v\in M_{K,\infty}$ in the right-hand side of \eqref{application of bombieri-vaaler}. Since the coordinates of the points $\boldsymbol \xi_{i}$ were chosen according to Proposition \ref{Serre}, given any $1\leq i\leq s$ and any monomial of degree $M$ corresponding to an index set $\boldsymbol J\subseteq \mathbb{N}^{n+2}$, it holds that for any place $v\in M_{K,\infty}$, 
\begin{equation}
|\boldsymbol \xi_{i}^{\boldsymbol J}|_{v}=|\boldsymbol \xi_{i,0}^{j_{0}}\cdots \boldsymbol \xi_{i,n+1}^{j_{n+1}}|_{v}\leq \begin{cases}  c_{1}^{M}(B^{j_{0}}\cdots B^{j_{n+1}})^{\frac{n_{v}}{d_{K}^{2}}}=c_{1}^{M}B^{\frac{Mn_{v}}{d_{K}^{2}}} & \text{ if }K\text{ is a number field},\\  c_{1}^{M}(B^{j_{0}}\cdots B^{j_{n+1}})^{\frac{1}{d_{K}}}=c_{1}^{M}B^{\frac{M}{d_{K}}} & \text{ if }K\text{ is a function field}. \end{cases}
\label{bound for the entries of the matrix}
\end{equation}

Let us suppose that $K$ is a number field. Let $v\in M_{K,\infty}$. If $\sigma\in S_{s}$ then 
$$|(A A^{\ast})_{i,\sigma(i)}|_{v}=\left|\sum_{k=1}^{r}a_{i,k}\overline{a_{\sigma(i),k}}\right|_{v}\leq c_{1}^{2M}|r|_{v}B^{\frac{2Mn_{v}}{d_{K}^{2}}} .$$

Hence
$$|\det(A A^{\ast})|_{v}=\left|\sum_{\sigma\in S_{s}}(-1)^{\sigma}\prod_{i=1}^{s}(AA^{\ast})_{i,\sigma(i)}\right|_{v}\leq |s!|_{v}\left( c_{1}^{2M}|r|_{v}B^{\frac{2Mn_{v}}{d_{K}^{2}}}\right)^{s}.$$

The above inequality allows us to conclude that
\begin{align}
 \prod_{v\in M_{K,\infty}}|\det(A A^{\ast})|_{v}^{\frac{d_{K}}{2}} \leq \prod_{v\in M_{K,\infty}}c_{1}^{d_{K}Ms}|s!|_{v}^{\frac{d_{K}}{2}}B^{\frac{Msn_{v}}{d_{K}}}|r|_{v}^{\frac{d_{K}s}{2}} \leq c_{1}^{d_{K}^{2}Ms}s!^{\frac{d_{K}}{2}}r^{\frac{d_{K}s}{2}}B^{Ms}.\label{arquimedean bound}
\end{align}

On the other hand, if $K$ is a function field over $\mathbb{F}_{q}(T)$, recalling that for all $x\in K$, $|x|_{v_{\infty}}=\mathcal{N}_{K}(\mathfrak{p}_{\infty})^{-\frac{\text{ord}_{\mathfrak{p}_{\infty}}(x)}{d_{K}}}$, \eqref{bound for the entries of the matrix} implies that every entry $a_{ij}$ of $A$ verifies $\text{ord}_{\mathfrak{p}_{\infty}}(a_{ij})\geq -\frac{M}{\deg(\mathfrak{p}_{\infty})}(d_{K}\log_{q}(c_{1})+\log_{q}(B))$. Then, for any minor $A_{J}$ it holds
\begin{align*}
\text{ord}_{\mathfrak{p}_{\infty}}(\det(A_{J})) & =\text{ord}_{\mathfrak{p}_{\infty}}\left( \sum_{\sigma\in S_{s}}(-1)^{\sigma}\prod_{i=1}^{s}(A_{J})_{i,\sigma(i)} \right)\geq \min_{\sigma\in S_{s}}\left\{\text{ord}_{\mathfrak{p}_{\infty}}\left((-1)^{\sigma}\prod_{i=1}^{s}(A_{J})_{i,\sigma(i)}\right)\right\}\\ & \geq  \min_{\sigma\in S_{s}}\left\{ \sum_{i=1}^{s}\text{ord}_{\mathfrak{p}_{\infty}}((A_{J})_{i,\sigma(i)}) \right\}\geq -\frac{sM}{\deg(\mathfrak{p}_{\infty})}(d_{K}\log_{q}(c_{1})+\log_{q}(B)),
\end{align*}
and from this, we get the bound
\begin{equation}
q^{-\deg(\mathfrak{p}_{\infty})\min_{|J|=s}\text{ord}_{\mathfrak{p}_{\infty}}(\det(A_{J}))}\leq q^{sM(d_{K}\log_{q}(c_{1})+\log_{q}(B))}=c_{1}^{d_{K}sM}B^{sM}.
\label{q bound}
\end{equation}
Denoting $\tilde{c_{1}}:=c_{1}^{d_{K}^{2}}$ or $\tilde{c_{1}}:=c_{1}^{d_{K}}$ according to $K$ being a number field or a function field, respectively, the inequalities \eqref{arquimedean bound} and \eqref{q bound} applied in \eqref{application of bombieri-vaaler} yields the desired result. 
\end{proof}

Now, the standing hypothesis \eqref{rank inequality} implies that $g$ must be of the form $g=fh$ with $h$ homogeneous of degree $M-d$. We shall bound $H_{K}(g)=H_{K}(fh)$ in terms of the height $H_{K}(f)$.

\begin{remark}
If $K$ is a number field, for $f,h\in K[X_{1},\ldots ,X_{n}]$, it holds (e.g. see \cite[Theorem 1.6.13]{Bombieri}) that $H_{K}(fh)\gtrsim_{\deg(f),\deg(h)} H_{K}(f)H_{K}(h)$. Since in our case the polynomial $h$ has degree $M-d$, such a bound would depend on the parameter $M$. The next claim shows that under the Assumption \ref{assumption}\eqref{assumption3} it is possible to obtain a bound independent of the degree of $h$.
\label{remark sobre la cota}  
\end{remark}

\begin{claim}
There exists some positive constant $C$, depending only on $n,K$, such that $H_{K}(g)\geq C^{-nd^{1+\frac{1}{n}}} H_{K}(f)$.
\label{lower bound3.6}
\end{claim}
\begin{proof}[Proof of Claim \ref{lower bound3.6}]
If $K$ is a function field, then by identity \eqref{lower bound for polynomials function field case} and the fact that $H_{K}(h)\geq 1$ the claim follows by taking $C=1$.

Now, let us suppose that $K$ is a number field. If $h=\sum_{I}c_{I}\boldsymbol X^{I}$, let $\boldsymbol c_{h}=(c_{I})_{I}$ be the projective point defined by the coordinates of $h$.  Following the proof of Proposition \ref{Serre}, we find a non-zero scalar $\lambda \in\mathcal{O}_{K}$ such that $(\lambda c_{I})_{I}$ satisfies the conditions of Proposition \ref{Serre}. Thus, we get $\prod_{v\in M_{K,\text{fin}}}|\lambda h|_{v}\geq c_{3}$. Since the polynomial height is invariant under multiplication by scalars, we get $H(g)=H(\lambda g)=H(f (\lambda h))$. Hence, after multiplication by a scalar, we may suppose that $g=f h$ with $h\in \mathcal{O}_{K}[X_{0},\ldots ,X_{n+1}]$ satisfying 
\begin{equation}
\prod_{v\in M_{K,\text{fin}}}|h|_{v}\geq c_{3}.
\label{lower bound6}
\end{equation}
Let $W$ be the greatest monomial (in right to left lexicographical order) appearing in $h$ with non-zero coefficient and let $w$ its coefficient. Then the monomial $WX_{n+1}^{d}$ appears in $g$ with $c_{f}w$ as its coefficient, where $c_{f}$ is the coefficient of $X_{n+1}^{d}$ in $f$. By Gauss' lemma (see \cite[Lemma 1.6.3]{Bombieri}), it holds
\begin{equation}
H(g)=H(fh)=\prod_{v\in M_{K}}|fh|_{v}=\prod_{v\in M_{K,\text{fin}}}|f|_{v} \prod_{v\in M_{K,\text{fin}}}|h|_{v}  \prod_{v\in M_{K,\infty}}|fh|_{v}.
\label{lower bound7}
\end{equation}
Now we bound $|fh|_{v}$ for any infinite place $v$. 
Since $WX_{n+1}^{d}$ appears in $g$ with coefficient $c_{f}w$, from Assumption \ref{assumption} we have
\begin{equation}
\prod_{v\in M_{K,\infty}}|fh|_{v}\geq \prod_{v\in M_{K,\infty}}|c_{f}w|_{v}\geq C^{-nd^{1+\frac{1}{n}}}\prod_{v\in M_{K,\infty}}|f|_{v}\prod_{v\in M_{K,\infty}}|w|_{v}.
\label{lower bound3}
\end{equation}
Moreover, if given $v\in M_{K,\infty}$, we let $\sigma_{v}$ be the embedding corresponding to $v$, then
\begin{equation*}
\prod_{v\in M_{K,\infty}}|w|_{v}=\prod_{v\text{ real}}|\sigma_{v}(x)|^{\frac{1}{d_{K}}} \prod_{v\text{ complex}}|\sigma_{v}(x)|^{\frac{2}{d_{K}}}=\left(\prod_{\sigma}|\sigma(w)|\right)^{\frac{1}{d_{K}}}\geq 1
\end{equation*}
where the last inequality is because $w$ is an algebraic integer and $\prod_{\sigma}\sigma(w)$ is the constant term of the minimal polynomial of $w$. Thus from \eqref{lower bound3} it follows
\begin{equation}
\prod_{v\in M_{K,\infty}}|fh|_{v}\geq C^{-nd^{1+\frac{1}{n}}}\prod_{v\in M_{K,\infty}}|f|_{v}.
\label{lower bound5}
\end{equation}

Using inequalities \eqref{lower bound6} and \eqref{lower bound5} in \eqref{lower bound7} we deduce
\begin{equation}
H(g)\geq c_{3}C^{-nd^{1+\frac{1}{n}}}\prod_{v\in M_{K}}|f|_{v}=c_{3}C^{-nd^{1+\frac{1}{n}}}H(f).
\label{lower bound3.5}
\end{equation}
Thus, from \eqref{lower bound3.5}, after relabelling $C$, the claim follows.
\end{proof}
Multiplying \eqref{application of bombieri-vaaler2} by $\mathcal{N}_{K}(\Delta)$ and using Claim \ref{lower bound3.6} it follows that
\begin{equation}
\left(C^{-nd^{1+\frac{1}{n}}}H_{K}(f)\right)^{r-s}\mathcal{N}_{K}(\Delta)\leq C(K)^{(r-s)}\tilde{c_{1}}^{Ms}s!^{\frac{d_{K}}{2}}r^{\frac{d_{K}s}{2}}B^{Ms}.
\label{app BV}
\end{equation}
Now we take logarithm in \eqref{app BV}. By Theorem \ref{global determinant method lower bound}, there is some non-zero ideal $\mathfrak{I}$, relative prime with all the primes lying in $\mathcal{P}_{X}$, such that $\mathcal{J}|\Delta$, and
\begin{equation}
\log(\mathcal{N}_{K}(\mathfrak{I}))\geq \frac{n!^{\frac{1}{n}}}{n+1}s^{1+\frac{1}{n}}(\log(s)-O_{n, K}(1)-n(\log(\beta)+\max\{\log(\log(\mathcal{N}_{K}(\mathfrak{q}))),0\}+\log b(f))).
\label{another factor dividing Delta0}
\end{equation}

Recall that given $i\in \{1,\ldots ,u\}$, $P_{i}$ is a non-singular $\mathbb{F}_{\mathfrak{p}_{i}}$-point on $X_{\mathfrak{p}_{i}}$. So, by Lemma \ref{local determinant}, there exists a non-negative $N_{i}=n!^{\frac{1}{n}}\frac{n}{n+1}s^{1+\frac{1}{n}}{-}O_{n}(s)$ such that $\mathfrak{p}_{i}^{N_{i}}|\Delta$. 
Hence, if we set $\mathfrak{I}':=\prod_{i=1}^{u}\mathfrak{p}_{i}^{N_{i}}$, then $\mathfrak{I}'|\Delta$, it is relative prime with $\mathfrak{I}$, and
\begin{equation}
\log(\mathcal{N}_{K}(\mathfrak{I}'))=n!^{\frac{1}{n}}\frac{n}{n+1}s^{1+\frac{1}{n}}\log(\mathcal{N}_{K}(\mathfrak{q}))+O_{n}(s\log(\mathcal{N}_{K}(\mathfrak{q}))).
\label{another factor dividing Delta}
\end{equation}

On the one hand, bounding $\log(\mathcal{N}_{K}(\Delta))$ by means of \eqref{another factor dividing Delta0} and \eqref{another factor dividing Delta}, the logarithm of the left hand side of \eqref{app BV} has as lower bound 

\begin{align}
& \frac{n!^{\frac{1}{n}}}{n+1}s^{1+\frac{1}{n}}\left(\log(s)+n\log(\mathcal{N}_{K}(\mathfrak{q}))-O_{n, K}(1)-n(\log(\beta)+\max\{\log(\log(\mathcal{N}_{K}(\mathfrak{q}))),0\}+\log b(f))\right) \label{lh side} \\ & -O_{n}(s\log(\mathcal{N}_{K}(\mathfrak{q})))+(r-s)\log(H_{K}(f))-(r-s)nd^{1+\frac{1}{n}}\log(C). \nonumber
\end{align}

On the other hand, the logarithm of the right hand side of \eqref{app BV} is
\begin{equation}
(r-s)\log(C(K))+Ms\log(\tilde{c_{1}})+\frac{d_{K}}{2}\log(s!)+Ms\log(B)+\frac{sd_{K}}{2}\log(r).
\label{rh side}
\end{equation}

Since $\log(s!)\leq s\log(s)$, $r=|\mathcal{B}[M]|={M+n+1\choose n+1}\leq (M+1)^{n+1}$ and $M\gtrsim d$, by \eqref{estimate for the log of s} it holds $\log(r)\lesssim_{n} \log (M)\lesssim_{n} \log (s)$, and so the terms $\frac{d_{K}}{2}\log(s!)$ and $\frac{sd_{K}}{2}\log(r)$ are majorized by $s^{1+\frac{1}{n}}$. Replacing \eqref{lh side} and \eqref{rh side} in \eqref{app BV} gives

\begin{align}
& \frac{n!^{\frac{1}{n}}}{n+1}s^{1+\frac{1}{n}}(\log(s)+n\log(\mathcal{N}_{K}(\mathfrak{q}))-O_{n, K}(1)-n(\log(\beta)+\max\{\log(\log(\mathcal{N}_{K}(\mathfrak{q}))),0\}+\log b(f)))\label{auxiliar inequality}  \\ & \leq Ms\log(\tilde{c_{1}})+ Ms\log(B)-(r-s)\log(H_{K}(f))+(r-s)\log(C_{K})+(r-s)nd^{1+\frac{1}{n}}\log(C)+O_{n}(s\log(\mathcal{N}_{K}(\mathfrak{q}))). \nonumber
\end{align}

We will now estimate the right hand side of \eqref{auxiliar inequality}. By \eqref{choice of M}, $M\gtrsim d^{2}$ and this implies
\begin{equation*}
r-s=|\mathcal{B}[M-d]|={M-d+n+1\choose n+1}=\frac{M^{n+1}}{(n+1)!}+O_{n}(dM^{n}),
\end{equation*}
hence from \eqref{estimate for s} we obtain
\begin{equation}
\frac{r-s}{Ms}=\frac{\frac{1}{(n+1)!}+O_{n}(dM^{-1})}{\frac{d}{n!}+O_{n}(d^{2}M^{-1})}=\frac{1}{d(n+1)}\frac{1+O_{n}(dM^{-1})}{1+O_{n}(dM^{-1})}=\frac{1}{d(n+1)}+O_{n}\left(\frac{1}{M}\right).
\label{estimate for r-s divided by Ms}
\end{equation}
Since $\log(\tilde{c_{1}})$ and $\frac{d^{\frac{1}{n}}n}{n+1}\log(C)$ are majorized by $\frac{d^{\frac{1}{n}}n}{n+1}O_{n,K}(1)$, and since by \eqref{choice of M}, $M\gtrsim \log(\mathcal{N}_{K}(\mathfrak{q}))$, using \eqref{estimate for r-s divided by Ms} and that $M\gtrsim d^{2}\gtrsim d^{1+\frac{1}{n}}$, the quotient by $Ms$ of the right hand side of \eqref{auxiliar inequality} is equal to 
\begin{align}
& \log(\tilde{c_{1}}) + \log(B)-\left(\frac{1}{d(n+1)}+O_{n}\left(\frac{1}{M}\right) \right)\log(H_{K}(f))+\left(\frac{1}{d(n+1)}+O_{n}\left(\frac{1}{M}\right) \right)nd^{1+\frac{1}{n}}\log(C)\label{auxiliar inequality2.5}\\ &  \nonumber \\ & +O_{n}\left(\frac{1}{M}\log(\mathcal{N}_{K}(\mathfrak{q}))\right)= \log(B)-\left(\frac{1}{d(n+1)}+O_{n}\left(\frac{1}{M}\right) \right)\log(H_{K}(f))+\frac{d^{\frac{1}{n}}n}{n+1}O_{n,K}(1). \nonumber
\end{align}

We will now estimate the left hand side of \eqref{auxiliar inequality}. Using the inequalities \eqref{estimate for the log of s} and \eqref{estimate for the 1/n-rooth} the quotient by $Ms$ in the left hand side of \eqref{auxiliar inequality} is equal to

\begin{align}
& \frac{n!^{\frac{1}{n}}}{n+1}\left( \frac{d^{\frac{1}{n}}}{n!^{\frac{1}{n}}}+O_{n}\left(\frac{d^{2}}{M}\right)\right)\Big(\log(s)+n\log(\mathcal{N}_{K}(\mathfrak{q}))+O_{n,K}(1)-n(\log(\beta)+\max\{\log(\log(\mathcal{N}_{K}(\mathfrak{q}))),0\}+\log(b(f)))\Big) \label{auxiliar inequality2} \\ & =\frac{d^{\frac{1}{n}}n}{n+1}\left(1+O_{n}\left(\frac{d^{2-\frac{1}{n}}}{M}\right)\right)\big(\log(M)+\frac{1}{n}\log(d)+\log(\mathcal{N}_{K}(\mathfrak{q}))+O_{n,K}(1)-\log(\beta)-\max\{\log(\log(\mathcal{N}_{K}(\mathfrak{q}))),0\}-\log(b(f))\Big)\nonumber \\ & =\frac{d^{\frac{1}{n}}n}{n+1}\left(\log(M)+\log(\mathcal{N}_{K}(\mathfrak{q}))-\log(\beta d^{-\frac{1}{n}})+O_{n,K}(1)-\left(1+O_{n}\left(\frac{d^{2-\frac{1}{n}}}{M}\right)\right)\log(b(f))-\max\{\log(\log(\mathcal{N}_{K}(\mathfrak{q}))),0\}\right),\nonumber
\end{align}
where the terms $O_{n}(\frac{d^{2-\frac{1}{n}}}{M}\log(M))$, $O_{n}(\frac{d^{2-\frac{1}{n}}}{M}\log(\mathcal{N}_{K}(\mathfrak{q})))$, $O_{n}(\frac{d^{2-\frac{1}{n}}}{M}\max\{\log(\log(\mathcal{N}_{K}(\mathfrak{q}))),0\})$, and $O_{n}(\frac{d^{2-\frac{1}{n}}}{M}\log(d))$ were merged in the term $O_{n,K}(1)$. 

Replacing the inequalities \eqref{auxiliar inequality2.5} and \eqref{auxiliar inequality2} in \eqref{auxiliar inequality} gives
\begin{align}
 \frac{d^{\frac{1}{n}}n}{n+1}& \Bigg( \log(M)+\log(\mathcal{N}_{K}(\mathfrak{q}))-\log(\beta d^{-\frac{1}{n}})-O_{n,K}(1)-\left(O_{n}\left(\frac{d^{2-\frac{1}{n}}}{M}\right)+1\right)\log(b(f)) \label{auxiliar inequality2.6} -\max\{\log(\log(\mathcal{N}_{K}(\mathfrak{q}))),0\}\Bigg)\\ &  \leq \log(B)-\left(\frac{1}{d(n+1)}+O_{n}\left(\frac{1}{M}\right) \right)\log(H_{K}(f)).\nonumber 
\end{align}

The leading term in the right hand side \eqref{auxiliar inequality2.6} depends on the size of $H_{K}(f)$. In the case when $H_{K}(f)$ is small, namely $H_{K}(f)\leq B^{2d(n+1)}$, we use \eqref{choice of M} to bound $M\geq d^{1-\frac{1}{n}}\log(B)$. From this it follows 
$$\frac{\log(H_{K}(f))}{M}\leq \frac{2d(n+1)\log(B)}{M}\lesssim_{n}d^{\frac{1}{n}}.$$
Assumption \ref{assumption}\eqref{assumption2}  and Lemma \ref{b(f)} yield 

\begin{equation}
b(f)\lesssim_{K,n} \max\{d^{[-2,\frac{4}{3},-\frac{8}{3},2]}\log(H_{K}(f)),1\}.
\label{cons assump}
\end{equation}

Then it holds that $d^{2-\frac{1}{n}}\frac{\log(b(f))}{M}\lesssim_{K,n}1$, hence in \eqref{auxiliar inequality2.6} the terms $O_{K,n}\left(\frac{\log(H_{K}(f))}{M}\right)$ and $O_{n}\left(d^{2-\frac{1}{n}}\frac{\log(b(f))}{M}\right)$ can be merged in the term $O_{n,K}(1)$. Then, rearranging \eqref{auxiliar inequality2.6}, we obtain

\begin{align*}
\log(M) \leq  & \frac{(n+1)}{nd^{\frac{1}{n}}}\log(B)-\log(\mathcal{N}_{K}(\mathfrak{q}))-\frac{\log(H_{K}(f))}{d^{1+\frac{1}{n}}n}+\log(\beta d^{-\frac{1}{n}})+O_{n,K}(1) +\log(b(f))+\max\{\log(\log(\mathcal{N}_{K}(\mathfrak{q}))),0\}. \nonumber
\end{align*}
Hence we reach a contradiction if the implicit constant in \eqref{choice of M} is large enough. 

In the case when $H_{K}(f)$ is large, namely $H_{K}(f)\geq B^{2d(n+1)}$, let us take the implicit constant in \eqref{choice of M} large enough such that $O_{n}(\frac{\log(H_{K}(f))}{M})$ is bounded by $\frac{\log(H_{K}(f))}{4d(n+1)}$. Then the right hand side of \eqref{auxiliar inequality2.6} is
\begin{align*}
\log(B)-\frac{\log(H_{K}(f))}{d(n+1)}-O_{n}\left(\frac{\log(H_{K}(f))}{M}\right) & \leq -\frac{\log(H_{K}(f))}{2d(n+1)}+\frac{\log(H_{K}(f))}{4d(n+1)} \leq -\frac{\log(H_{K}(f))}{4d(n+1)}.
\end{align*}
Recalling that by \eqref{choice of M}, $M\gtrsim d^{2}$, we conclude that $O(\frac{d^{2-\frac{1}{n}}}{M})=O(d^{-\frac{1}{n}})$. Then, rearranging \eqref{auxiliar inequality2.6} we obtain 

\begin{align}
\log(M) \leq & \, O_{n,K}(1)+\log(\beta d^{-\frac{1}{n}})-\log(\mathcal{N}_{K}(\mathfrak{q}))+\left(1+O_{n}(d^{-\frac{1}{n}})\right)\log(b(f)) \label{auxiliar inequality3} \\ & +\max\{\log(\log(\mathcal{N}_{K}(\mathfrak{q}))),0\}-\frac{\log(H_{K}(f))}{4nd^{1+\frac{1}{n}}}. \nonumber
\end{align}

By \eqref{cons assump},
\begin{align}
 O_{n}(1)\log(b(f))-\frac{\log(H_{K}(f))}{4nd^{1+\frac{1}{n}}} & \leq O_{n,K}(1)\log(\max\{d^{[-2,\frac{4}{3},-\frac{8}{3},2]}\log(H_{K}(f)),1\})-\frac{\log(H_{K}(f))}{4nd^{1+\frac{1}{n}}} \label{auxiliar inequality4} \\ & \leq O_{n,K}(1)\max\{\log\log(H_{K}(f))+\left[-2,\frac{4}{3},-\frac{8}{3},2\right]\log(d),0\}-\frac{\log(H_{K}(f))}{4nd^{1+\frac{1}{n}}}\nonumber\\ & \leq \max\left\{0,\left[-2,\frac{4}{3},-\frac{8}{3},2\right]\log(d)+O_{n,K}(1)\log\log(H_{K}(f))-\frac{\log(H_{K}(f))}{4nd^{1+\frac{1}{n}}}\right\} \nonumber \\ & \leq \max\left\{0,\left[-2,\frac{4}{3},-\frac{8}{3},2\right]\log(d)+\left(1+\frac{1}{n}\right)\log(d)+O_{n,K}(1)\right\}. \nonumber
\end{align}

In the last line, we used that given $c>0$, for any $x>1$ we have $\log\log(x)-\frac{\log(x)}{c}\leq \log(c)+O(1)$ (see \cite[Lemma 3.3.5]{Cluckers}). Using \eqref{auxiliar inequality4} in \eqref{auxiliar inequality3} we get
\begin{align*}
\log(M) & \leq O_{n,K}(1)+\log(\beta d^{-\frac{1}{n}})-\log(\mathcal{N}_{K}(\mathfrak{q}))+\max\{\log(\log(\mathcal{N}_{K}(\mathfrak{q}))),0\} \nonumber\\ & \qquad+\max\left\{0,\left(\left[-1,\frac{7}{3},-\frac{5}{3},3\right]+\frac{1}{n}\right)\log(d)+O_{n,K}(1)\right\} \nonumber\\ &  \leq O_{n,K}(1)+\max\left\{\log\left(\beta d^{-\frac{1}{n}}\right),\log(\beta d^{[-1,\frac{7}{3},-\frac{5}{3},3]})+O_{n,K}(1)\right\}-\log(\mathcal{N}_{K}(\mathfrak{q}))+\max\{\log(\log(\mathcal{N}_{K}(\mathfrak{q}))),0\}. \nonumber
\end{align*}
Hence if the implicit constant in \eqref{choice of M} is large enough, we arrive at a contradiction. 
\end{proof}

\begin{proof}[Proof of Theorem \ref{polynomial method2}]
By Remark \ref{abs irred assumption}, we may assume that $f$ is absolutely irreducible. Then, by Remark \ref{remark about primitive polynomials}, we may assume, after multiplication by a non-zero scalar in $\mathcal{O}_{K}$ that $f$ satisfies conditions (1) and (2) in Assumption \ref{assumption}. This is enough to conclude the proof in the function field case. 

For number fields, we use Proposition \ref{auxiliar change} to find integers $a_{0},\ldots ,a_{n}$ with $0\leq a_{i}\leq d$ and an unit $\varepsilon\in \mathcal{O}_{K}^{\times}$ such that $f_{1}(X_{0},\ldots ,X_{n},X_{n+1}):=\varepsilon f(X_{0}-a_{0}X_{n+1},\ldots ,X_{n}-a_{n}X_{n+1},X_{n+1})$ satisfies all the conditions in Assumption \ref{assumption}. For each $i\in \{1,\ldots ,u\}$, let $P_{i}'$ be the image of $P_{i}$ under the linear transformation $(x_{0},\ldots, x_{n+1})\mapsto (x_{0}+a_{0}x_{n+1},\ldots ,x_{n}-a_{n}x_{n+1},x_{n+1})$; it is a non-singular $\mathbb{F}_{\mathfrak{p}_{i}}$-rational point on $\mathcal{Z}(f_{1})_{\mathfrak{p}}$. Then we can apply Theorem \ref{polynomial method2} to $f_{1}$ and $P_{1}',\ldots ,P_{u}'$ to conclude that for any $B'\geq 1$ there exists a homogeneous polynomial $g_{1}\in \mathcal{O}_{K}[X_{0},\ldots ,X_{n+1}]$ of degree 
\begin{align}
M \lesssim_{K,n} & {B'}^{\frac{n+1}{nd^{\frac{1}{n}}}}\frac{d^{[4,\frac{14}{3},\frac{14}{3},0]-\frac{1}{n}}b(f_{1})\max\{\log(\mathcal{N}_{K}(\mathfrak{q})),1\}}{\mathcal{N}_{K}(\mathfrak{q})H_{K}(f_{1})^{\frac{1}{n}\frac{1}{d^{1+\frac{1}{n}}}}}+d^{1-\frac{1}{n}}\log(B'\mathcal{N}_{K}(\mathfrak{q})) \label{M N'}+d^{2-\frac{1}{n}}\log(\mathcal{N}_{K}(\mathfrak{q}))\\ & +d^{[4-\frac{1}{n},7,\frac{14}{3}-\frac{1}{n},3]}\left(\frac{\max\{\log(\mathcal{N}_{K}(\mathfrak{q})),1\}}{\mathcal{N}_{K}(\mathfrak{q})}+1\right)\nonumber,
\end{align}
not divisible by $f_{1}$, which vanishes on $\mathcal{Z}(f_{1})(K,B';P_{1}',\ldots, P_{u}')$. Thus, $g:=g_{1}(X_{0}+a_{0}X_{n+1},\ldots ,X_{n}+a_{n}X_{n+1},X_{n+1})$ is a polynomial of degree at most $M$, vanishing on $X(K,\frac{B'}{(2d)^{d_{K}}};P_{1},\ldots, P_{u})$. Indeed, if $L(X_{0},\ldots ,X_{n+1}):=(X_{0}-a_{0}X_{n+1},\ldots ,X_{n}-a_{n}X_{n+1},X_{n+1})$, by \eqref{polynomialheight}, it holds 
$$H_{K}(L(x_{0}:\ldots :x_{n+1}))\leq (2d)^{d_{K}}H_{K}(x_{0}:\ldots :x_{n+1}).$$

Take $B'=(2d)^{d_{K}}B$. By Proposition \ref{auxiliar change}, $f_{1}$ verifies $b(f)=b(f_{1})$ and $H_{K}(f)\leq C'^{d^{1+\frac{1}{n}}}H_{K}(f_{1})$. Replacing this in \eqref{M N'}, it follows  that $M$ verifies the bound on the statement of the theorem. 
\end{proof}

As a consequence of Theorem \ref{polynomial method2}, Lemma \ref{b(f)}, Remark \ref{remark about primitive polynomials}, and B\'ezout's theorem, we generalize \cite[Theorem 1.2]{Walsh3}, \cite[Corollary 3.1.2]{Cluckers} and \cite[Theorem 1.1]{Vermeulen}:

\begin{corollary}
Let $K$ be a global field of degree $d_{K}$. For any irreducible homogeneous polynomial $f\in \mathcal{O}_{K}[X_{0},X_{1},X_{2}]$ of degree $d$ and for any $B\geq 1$, if $\mathcal{Z}(f)\subseteq \mathbb{P}_{K}^{2}$ is the corresponding curve, it holds
$$N(\mathcal{Z}(f),K,B)\lesssim_{K}B^{\frac{2}{d}}\frac{d^{[4,\frac{14}{3},\frac{14}{3},0]}b(f)}{H_{K}(f)^{\frac{1}{d^{2}}}}+d\log(B)+d^{[4,8,\frac{14}{3},4]}\lesssim_{K}d^{[4,8,\frac{14}{3},4]}B^{\frac{2}{d}}.$$
\label{Walsh corollary}
\end{corollary}

Now we are in position to prove Theorem \ref{teorema1} of the introduction  by extending Corollary \ref{Walsh corollary} to general curves using a projection argument. This is accomplished by means of an effective change of variables, as in \cite{P2}. The approach presented here differs from the one given in \cite[$\S 5$]{Cluckers}.

\begin{theorem}
Let $K$ be a global field of degree $d_{K}$. Given $n>1$, for any integral projective curve $C\subseteq \mathbb{P}_{K}^{n}$ of degree $d$ it holds
$$N(C,K,B)\lesssim_{K,n}d^{[4,8,\frac{14}{3},4]}B^{\frac{2}{d}}.$$
\label{Walsh general curves}
\end{theorem}

\begin{proof}
We are going to reduce Theorem \ref{Walsh general curves} to the case of a planar curve, by means of an adequate change of variables, as in \cite{P2}. More specifically, by the method of the proof of \cite[Theorem 4.3]{P2} there exists linear forms $L_{0}, L_{1},L_{2}\in \mathcal{O}_{K}[X_{0},\ldots ,X_{n}]$ of height bounded by $\lesssim_{\mathbbm{k},n}d^{n-2}$ such that $\varphi:C\to \mathbb{P}^{2}$ given by $\varphi(\boldsymbol x):=(L_{0}(\boldsymbol x):L_{1}(\boldsymbol x):L_{2}(\boldsymbol x))$ is a finite morphism with fibres of size at most $d$ and $\varphi(C)$ is a geometrically integral projective curve defined over $K$ of degree $\deg(\varphi(C))\leq d$. Furthermore, by \eqref{polynomialheight}, any point $\boldsymbol x\in C(K)$ of $K$-relative height at most $B$ verifies $H_{K}(\varphi(\boldsymbol x))\lesssim_{\mathbbm{k},n} d^{n-2}H_{K}(\boldsymbol x)\lesssim_{\mathbbm{k},n}d^{n-2}B$. Hence, there is some constant $c=c(K,n)$ such that it holds 
\begin{equation}
N(C,K,B)\leq dN(\varphi(C),K,cd^{n-2}B).
\label{normalization1}
\end{equation}
By Corollary \ref{Walsh corollary}, we have

\begin{equation}
N(\varphi(C),K,cd^{n-2}B)\lesssim_{K,n}d^{[4,8,\frac{14}{3},4]}\left(cd^{n-2}B\right)^{\frac{2}{d}}\lesssim_{K,n}d^{[4,8,\frac{14}{3},4]}B^{\frac{2}{d}}.
\label{consequence of normalization1}
\end{equation}

The conclusion of Theorem \ref{Walsh general curves} follows from \eqref{normalization1} and \eqref{consequence of normalization1}. 
\end{proof}

\begin{remark}
By Remark \ref{abs irred assumption} and B\'ezout's theorem, if $f\in \mathcal{O}_{K}[X_{0},X_{1},X_{2}]$ is irreducible but not absolutely irreducible, we have $N(\mathcal{Z}(f),K,B)\leq d^{2}$. In consequence, the argument of the proof of Theorem \ref{Walsh general curves} shows that for any integral projective curve $C\subseteq \mathbb{P}_{K}^{n}$ of degree $d$, which is not geometrically irreducible, it holds $N(C,K,B)\leq d^{2}$. 
\label{curvas no absolutamente irreducibles} 
\end{remark}

\subsection{The affine case} \label{subsection 5.2}

In Theorem \ref{polynomial method2} we found a hypersurface of small degree vanishing on the rational points up to height $B$  of a given projective hypersurface $X$. Now we will obtain a similar result in the affine case. For this purpose, given a number field $K$, for any $x\in \mathcal{O}_{K}$ we define 
$$\house{x}:=\max_{\sigma:K\hookrightarrow \mathbb{C}}|\sigma(x)|,$$
and
$$[B]_{\mathcal{O}_{K}}:=\{x\in \mathcal{O}_{K}:\house{x}\leq B^{\frac{1}{d_{K}}}\}.$$
When $K$ is a function field, we define
$$[B]_{\mathcal{O}_{K}}:=\{x\in \mathcal{O}_{K}:|x|_{v_{\infty}}\leq B^{\frac{1}{d_{K}}}\}.$$
\begin{remark}
The previous definition is to ensure that given $\boldsymbol x\in \mathbb{P}^{n+1}(K)$ with $H_{K}(\boldsymbol x)\leq B$, by Proposition \ref{Serre} it follows that $\boldsymbol x$ has a lift $(y_{1},\ldots ,y_{n+2})\in \mathbb{A}^{n+2}(K)$ which lies in the box $[c_{1}B]_{\mathcal{O}_{K}}^{n+2}$. 
\label{explicacion sobre la caja}
\end{remark}

Similarly to the projective case, for any $X\subseteq \mathbb{A}_{K}^{n}$ affine variety, given $B\in \mathbb{R}_{>0}$ we will use the notation
$$X_{\text{aff}}(\mathcal{O}_{K},B):=X(K)\cap [B]_{\mathcal{O}_{K}}^{n},$$
Moreover, if $\{\mathfrak{p}_{1},\ldots ,\mathfrak{p}_{u}\}$ is a subset of primes such that for all $1\leq i\leq u$ we let $P_{i}$ be a  non-singular $\mathbb{F}_{\mathfrak{p}_{i}}$-point on $X_{\mathfrak{p}_{i}}$, we denote 
$$X_{\text{aff}}(\mathcal{O}_{K},B;P_{1},\ldots ,P_{u}):=\{\boldsymbol x\in X_{\text{aff}}(\mathcal{O}_{K},B):\boldsymbol x\text{ specialises to }P_{i} \text{ in }X_{\mathfrak{p}_{i}}\text{ for all }i\}.$$
We denote
$$N_{\text{aff}}(X,\mathcal{O}_{K},B):=|X_{\text{aff}}(\mathcal{O}_{K},B)|.$$

We need the next application of the bound of Bombieri and Vaaler presented in Corollary \ref{Bombieri-Vaaler a la Siegel}, which will be useful also in Lemma \ref{reduction to singular points} and Lemma \ref{no geo integro3}. A variant of this for $K=\mathbb{Q}$ and $K=\mathbb{F}_{q}(T)$ can be found in \cite[Theorem 4]{Heath-Brown} and \cite[Lemma 3.5]{Vermeulen} with different proofs.

\begin{lemma}
Let $f\in \mathcal{O}_{K}[X_{1},\ldots ,X_{n+2}]$ be a homogeneous polynomial of degree $d$. 
\begin{enumerate}
\item Let $B\geq 1$. Then, either  $H_{K,\emph{\text{aff}}}(f)\lesssim_{K} {d+n+1\choose n+1}^{d_{K}{d+n+1\choose n+1}}B^{d{d+n+1\choose n+1}}$ or there exists a homogeneous polynomial $g$ of degree $d$, vanishing on all $\mathcal{Z}(f)_{\emph{\text{aff}}}(\mathcal{O}_{K},B)$, and not divisible by $f$; \label{interpolation argument a}
\item Let $B\geq 1$. Then, either $H_{K,\emph{\text{aff}}}(f)\lesssim_{K} {d+n+1\choose n+1}^{d_{K}{d+n+1\choose n+1}}B^{d{d+n+1\choose n+1}}$, or there exists a homogeneous polynomial $g$ of degree $d$, vanishing on all $\{\boldsymbol x=(x_{1},\ldots, x_{n+2})\in [B]_{\mathcal{O}_{K}}^{n+2}:x_{1}=1\text{ and }f(1,x_{2},\ldots ,x_{n+2})=0\}$, and not divisible by $f$.  \label{interpolation argument b}
\end{enumerate}
\label{interpolation argument}
\end{lemma}

\begin{proof}
Let $R:={d+n+1\choose n+1}$ and let $\boldsymbol x^{(1)},\ldots ,\boldsymbol x^{(N)}\in \mathcal{Z}(f)_{\text{aff}}(\mathcal{O}_{K},B)$. Consider the $N\times R$ matrix $M$ whose $i^{\text{th}}$ row consists of the $R$ possible monomials of degree $d$ in the variables $X_{1},\ldots ,X_{n+2}$ evaluated in $\boldsymbol x^{(i)}$. We have
$$H_{K}(M)\leq \left(\prod_{v\in M_{K,\infty}}\max_{i,j}\{1,|M_{ij}|_{v}\}\right)^{d_{K}}\leq B^{d}.$$ 
Denoting by $\boldsymbol f$ the vector in $\mathcal{O}_{K}^{n+2}$ consisting of the coefficients of $f$, it holds $M \boldsymbol f=0$. Since $\boldsymbol f\neq 0$, the matrix $M$ must have rank at most $R-1$. Hence, by Corollary \ref{Bombieri-Vaaler a la Siegel} the linear system $M\boldsymbol g$ has a non-zero $\boldsymbol g\in \mathcal{O}_{K}^{R}$ verifying
$$H_{K,\text{aff}}(\boldsymbol g)\leq C(K)\left(R^{\frac{d_{K}}{2}}H_{K}(M)\right)^{\frac{\text{rank}(M)}{R-\text{rank}(M)}}\leq C(K){d+n+1\choose n+1}^{d_{K}{d+n+1\choose n+1}}B^{d{d+n+1\choose n+1}}.$$
Let $g\in \mathcal{O}_{K}[X_{1},\ldots ,X_{n+2}]$ be the polynomial of degree $d$ corresponding to $\boldsymbol g$. Then, either $f$ divides $g$, in which case $H_{K}(f)=H_{K}(g)=H_{K}(\boldsymbol g)\leq H_{K,\text{aff}}(\boldsymbol g)\leq {d+n+1\choose n+1}^{d_{K}{d+n+1\choose n+1}}B^{d{d+n+1\choose n+1}}$, or $g$ is a non-zero homogeneous polynomial of degree $d$, not divisible by $f$ and vanishing on all $\mathcal{Z}(f)_{\text{aff}}(\mathcal{O}_{K},B)$. This proves Lemma \ref{interpolation argument}\eqref{interpolation argument a}. The proof of Lemma \ref{interpolation argument}\eqref{interpolation argument b} is analogous.
\end{proof}

\begin{lemma}
Let $K$ be a global field of degree $d_{K}$. Let $f=\sum_{I}a_{I}\boldsymbol X^{I}\in \mathcal{O}_{K}[X_{1},\ldots ,X_{n+2}]$ be an absolutely irreducible homogeneous polynomial of degree $d$. If we let $1\leq y\leq H_{K,\emph{\text{aff}}}(f)$ then
$$d^{[4,\frac{14}{3},\frac{14}{3},0]-\frac{1}{n}}\frac{b(f)}{H_{K,\emph{\text{aff}}}(f)^{\frac{1}{n}\frac{1}{d^{1+\frac{1}{n}}}}}\lesssim_{n,K} d^{[2,6,2,2]-\frac{1}{n}}\frac{\log(y)+d^{[2,1+\frac{1}{n},\frac{8}{3},1+\frac{1}{n}]}}{y^{\frac{1}{n}\frac{1}{d^{1+\frac{1}{n}}}}}.$$
\label{bound on b(f)}
\end{lemma}

\begin{proof}
The proof is similar to the one of \cite[Lemma 4.2.3 ]{Cluckers} for number fields, and to the one of \cite[Lemma 3.4]{Vermeulen} for function fields, but instead of using \cite[Corollary 3.2.3]{Cluckers} or \cite[Lemma 2.3]{Vermeulen} one uses Lemma \ref{b(f)}. 
\end{proof}

Now we use Lemma \ref{interpolation argument}\eqref{interpolation argument a} and Lemma \ref{bound on b(f)} to prove the following generalization of \cite[Proposition 4.2.1]{Cluckers} and \cite[Lemma 3.6]{Vermeulen}:

\begin{theorem}
Let $K$ be a global field of degree $d_{K}$. Let $n>0$ be an integer. Let $X\subseteq \mathbb{A}_{K}^{n+1}$ be an integral hypersurface of degree $d>0$, defined by an irreducible polynomial $f\in \mathcal{O}_{K}[X_{1},\ldots ,X_{n+1}]$ of degree $d$. For each $i$ write $f_{i}$ for the degree $i$ homogeneous part of $f$. Let $\{\mathfrak{p}_{1},\ldots ,\mathfrak{p}_{u}\}$ be (a possibly empty) subset of primes and $P_{i}$ be a non-singular $\mathbb{F}_{{\mathfrak{p}_{i}}}$-point on $X_{\mathfrak{p}_{i}}$ for each $i\in \{1,\ldots, u\}$. Set $\mathfrak{q}:=\prod_{i=1}^{u}\mathfrak{p}_{i}$ if $u\geq 1$ and $\mathfrak{q}=(1)$ if $u=0$. Fix $B\geq 1$. Then there is a polynomial $g\in \mathcal{O}_{K}[X_{1},\ldots ,X_{n+1}]$ of degree 
\begin{align*}
M\lesssim_{K,n} & B^{\frac{1}{d^{\frac{1}{n}}}}d^{[2,6,2,2]-\frac{1}{n}}\frac{\min\{\log(H_{K}(f_{d}))+d\log(B)+d^{[2,1+\frac{1}{n},\frac{8}{3},1+\frac{1}{n}]},d^{[2,-\frac{4}{3},\frac{8}{3},-2]}b(f)\}\max\{\log(\mathcal{N}_{K}(\mathfrak{q})),1\}}{H_{K}(f_{d})^{\frac{1}{n}\frac{1}{d^{1+\frac{1}{n}}}}\mathcal{N}_{K}(\mathfrak{q})}\\ & +d^{1-\frac{1}{n}}\log(B\mathcal{N}_{K}(\mathfrak{q}))+d^{2-\frac{1}{n}}\log(\mathcal{N}_{K}(\mathfrak{q}))+d^{[4-\frac{1}{n},7,\frac{14}{3}-\frac{1}{n},3]}\left(\frac{\max\{\log(\mathcal{N}_{K}(\mathfrak{q})),1\}}{\mathcal{N}_{K}(\mathfrak{q})}+1\right),
\end{align*}
not divisible by $f$, and vanishing on all $X_{\emph{\text{aff}}}(\mathcal{O}_{K},B;P_{1},\ldots, P_{u})$.
\label{ell-venk}
\end{theorem}

\begin{proof}
The proof follows the strategy devised in \cite[Remark 2.3]{Ellenberg} and developed in \cite[Proposition 4.2.1]{Cluckers}. Recall that $\mathbbm{k}$ denotes the field $\mathbb{Q}$ if $K$ is a number field, or $\mathbb{F}_{q}(T)$ if $K$ is a function field. By Remark \ref{abs irred assumption} and Remark \ref{remark about primitive polynomials}, we may assume, after multiplication by a non-zero scalar in $\mathcal{O}_{K}$ that $f=\sum_{I}a_{I}\boldsymbol X^{I}$ is absolutely irreducible and it is $\mathfrak{p}$-primitive for any prime $\mathfrak{p}$ of $\mathcal{O}_{K}$ with $\mathcal{N}_{K}(\mathfrak{p})>c_{2}$, and $H_{K,\text{aff}}(f)\leq c_{1}^{d_{K}}H_{K}(f)$.

For each $L\in \mathcal{O}_{\mathbbm{k}}$ we consider the polynomial $F_{L}\in \mathcal{O}_{K}[X_{1},\ldots ,X_{n+2}]$ given by $F_{L}(X_{1},\ldots ,X_{n+2}):=\sum_{i=0}^{d}L^{i}f_{i}X_{n+2}^{d-i}$. Then $F_{L}$ is an absolutely irreducible homogeneous polynomial of degree $d$. 
Furthermore, if we denote by $X_{L}$ the projective hypersurface in $\mathbb{P}_{K}^{n+1}$ of degree $d$ defined by $F_{L}$, we see that each point $(x_{1},\ldots ,x_{n+1})\in X(\mathcal{O}_{K})$ gives a $K$-rational point $(x_{1},\ldots ,x_{n+1},L)$ in $X_{L}$. Given $1\leq i\leq u$, we denote $\tilde{P}_{i}:=(P_{i},L)\in (X_{L})_{\mathfrak{p}_{i}}(\mathbb{F}_{\mathfrak{p}_{i}})$. If $\mathfrak{p}_{i}\not | L$, then we see that $\tilde{P}_{i}$ is a non-singular $\mathbb{F}_{\mathfrak{p}_{i}}$-point on $(X_{L})_{\mathfrak{p}_{i}}$.

Observe that given a subset of $\mathcal{O}_{K}^{n+1}$ of size $\lesssim_{K} 1$, there exists a $g\in \mathcal{O}_{K}[X_{1},\ldots ,X_{n+1}]$ of degree $\lesssim_{K,n}1$ vanishing on it which is not divisible by $f$. Applying this to the subset $[B]_{\mathcal{O}_{K}}^{n+1}$ with $B\lesssim_{K}1$, we see that it is enough to prove Theorem \ref{ell-venk} when $B\gtrsim_{K}1$.

Let $\mathcal{P}'$ be the subset of primes $L\in \mathcal{O}_{\mathbbm{k}}$ with $H_{\mathbbm{k}}(L)\in [(\frac{B}{2})^{\frac{1}{d_{K}}},B^{\frac{1}{d_{K}}}]$ such that $\mathfrak{q}=\prod_{i=1}^{u}\mathfrak{p_{i}}$ has no common prime factor with $L$. Let us see that if 
\begin{equation}
\sum_{L\in \mathcal{P}'}\log(\mathcal{N}_{\mathbbm{k}}(L))\leq \frac{1}{4}B^{\frac{1}{d_{K}}},
\label{errata referi0}
\end{equation}
then $|X_{\text{aff}}(\mathcal{O}_{K},B;P_{1},\ldots, P_{u})|\lesssim_{K}1$. Indeed, 
by Bertrand's postulate \eqref{Bertrand} and \eqref{errata referi0} it follows that for $B\gtrsim_{\mathbbm{k}} 1$ 
\begin{equation}
\sum_{\substack{L\text{ prime in }  \mathcal{O}_{\mathbbm{k}}\backslash \mathcal{P}'\\
H_{\mathbbm{k}}(L)\in[(\frac{B}{2})^{\frac{1}{d_{K}}},B^{\frac{1}{d_{K}}}]
}}
\log(\mathcal {N}_{\mathbbm{k}}(L))\geq \frac{1}{4}B^{\frac{1}{d_{K}}}.
\label{errata referi}
\end{equation}
By the definition of $\mathcal{P}'$, for any prime $L$ in $\mathcal{O}_{\mathbbm{k}}\backslash \mathcal{P}'$ such that $H_{\mathbbm{k}}(L)\in [(\frac{B}{2})^{\frac{1}{d_{K}}},B^{\frac{1}{d_{K}}}]$, we choose one $\mathfrak{p}_{L}|\mathfrak{q}$ such that $\mathfrak{p}_{L}|L$. Then $\mathcal{N}_{K}(\mathfrak{p}_{L})=\mathcal{N}_{\mathbbm{k}}(L)^{f_{\mathfrak{p}_{L}}}\geq \mathcal{N}_{\mathbbm{k}}(L)$, where $f_{\mathfrak{p}_{L}}$ is the inertia index of $\mathfrak{p}_{L}$ above $L$. 
Hence, if $\mathcal{P}'':=\{\mathfrak{p}_{L}:L\text { prime in }\mathcal{O}_{\mathbbm{k}}\backslash \mathcal{P}',  H_{\mathbbm{k}}(L)\in[(\frac{B}{2})^{\frac{1}{d_{K}}},B^{\frac{1}{d_{K}}}]\}$ then
\begin{equation}
\sum_{\mathfrak{p}_{L}\in \mathcal{P}''}\log(\mathcal{N}_{K}(\mathfrak{p}_{L}))\
\geq \frac{1}{4}B^{\frac{1}{d_{K}}}.
\label{errata referi2}
\end{equation}
Furthermore, for any $\mathfrak{p}_{L}\in \mathcal{P}''$, since $\mathfrak{p}_{L}|\mathfrak{q}$, from the definition of the subset $X_{\text{aff}}(\mathcal{O}_{K},B;P_{1},\ldots, P_{u})$ it holds that all $\boldsymbol x\in X_{\text{aff}}(\mathcal{O}_{K},B;P_{1},\ldots, P_{u})$ reduces modulo $\mathfrak{p}_{L}$ to the same point in $\mathbb{F}_{\mathfrak{p}_{L}}$. In consequence, if $X_{j}$ is the image of $X_{\text{aff}}(\mathcal{O}_{K},B;P_{1},\ldots, P_{u})$ under the projection on the $j$-th coordinate, for any pair of distinct elements $x,y\in X_{j}$ it holds that $x\equiv y\text{ mod}(\mathfrak{p}_{L})$, i.e $\mathfrak{p}_{L}|x-y$. Thus, $\prod_{\mathfrak{p}_{L}\in \mathcal{P}''}\mathfrak{p}_{L}\displaystyle |x-y$. This and \eqref{norm} imply  the bound
\begin{equation}
\prod_{\mathfrak{p}_{L}\in \mathcal{P}''}\mathcal{N}_{K}(\mathfrak{p}_{L})\leq \mathcal{N}_{K}(x-y)\leq H_{K}(x-y)\text{ for all }x\neq y,x,y\in X_{j}.
\label{errata referi3}
\end{equation}
Since $x,y\in [B]_{\mathcal{O}_{K}}$, if $K$ is a number field then the bound $|x-y|_{v}\leq 2^{\frac{n_{v}}{d_{K}}}B^{\frac{n_{v}}{d_{K}^{2}}}$ holds for all $v\in M_{K,\infty}$, and if $K$ is a function field it holds $|x-y|_{v_{\infty}}\leq \max\{|x|_{v_{\infty}},|y|_{v_{\infty}}\}\leq B^{\frac{1}{d_{K}}}$. Then \eqref{errata referi3} implies
\begin{equation}
\prod_{\mathfrak{p}_{L}\in \mathcal{P}''}\mathcal{N}_{K}(\mathfrak{p}_{L})\leq H_{K}(x-y)=\left(\prod_{v\in M_{K,\infty}}\max\{1,|x-y|_{v}\}\right)^{d_{K}}\leq 2^{d_{K}}B\text{ for all }x\neq y,x,y\in X_{j}.
\label{errata referi4}
\end{equation}
Taking logarithms in \eqref{errata referi4} and using \eqref{errata referi2} we conclude the bound
\[\frac{1}{4}B^{\frac{1}{d_{K}}}\leq \sum_{\mathfrak{p}_{L}\in \mathcal{P}''}\log(\mathcal{N}_{K}(\mathfrak{p}_{L}))\leq d_{K}\log(2)+\log(B)\text{ for all }x\neq y,x,y\in X_{j},\]
which clearly can not hold if $B$ is large enough. Thus, for $B\gtrsim_{K} 1$ we have $|X_{j}|=1$ for all $j$, and in particular $|X_{\text{aff}}(\mathcal{O}_{K},B;P_{1},\ldots, P_{u})|\leq 1$. Hence, for $B\gtrsim_{K}1$ if $\sum_{L\in \mathcal{P}'}\log(\mathcal{N}_{\mathbbm{k}}(L))\leq \frac{1}{4}B^{\frac{1}{d_{K}}}$ there exists $g\in \mathcal{O}_{K}[X_{1},\ldots ,X_{n+1}]$ of degree $\lesssim_{K}1$ vanishing on $X_{\text{aff}}(\mathcal{O}_{K},B;P_{1},\ldots, P_{u})$ which is not divisible by $f$.

Then, we may suppose that 
\begin{equation}
\sum_{L\in \mathcal{P}'}\log(\mathcal{N}_{\mathbbm{k}}(L))\geq \frac{1}{4}B^{\frac{1}{d_{K}}}.
\label{bound for P'}
\end{equation}
We will split the rest of the proof in two cases: when there exists $L\in \mathcal{P}'$ for which $L\not | f_{0}$ and when for all $L\in \mathcal{P}'$, $L| f_{0}$.  

In the first case, $F_{L}$ is $\mathfrak{p}$-primitive for all primes of norm at least $c_{2}$. Moreover, for all $i$, $\tilde{P}_{i}$ is a non-singular $\mathbb{F}_{\mathfrak{p}_{i}}$-point on $(X_{L})_{\mathfrak{p}_{i}}.$ We apply Theorem \ref{polynomial method2}  to $F_{L}$ to find a homogeneous polynomial $G_{L}\in \mathcal{O}_{K}[X_{1},\ldots ,X_{n+2}]$ of degree
\begin{align}
M\lesssim_{K,n} & B^{\frac{n+1}{nd^{\frac{1}{n}}}}\frac{d^{[4,\frac{14}{3},\frac{14}{3},0]-\frac{1}{n}}b(F_{L})\max\{\log(\mathcal{N}_{K}(\mathfrak{q})),1\}}{H_{K}(F_{L})^{\frac{1}{n}\frac{1}{d^{1+\frac{1}{n}}}}\mathcal{N}_{K}(\mathfrak{q})}+d^{1-\frac{1}{n}}\log(B\mathcal{N}_{K}(\mathfrak{q})) \label{affine count1}\\ & +d^{[4-\frac{1}{n},7,\frac{14}{3}-\frac{1}{n},3]}\left(\frac{\max\{\log(\mathcal{N}_{K}(\mathfrak{q})),1\}}{\mathcal{N}_{K}(\mathfrak{q})}+1\right),\nonumber
\end{align}

not divisible by $F_{L}$, and vanishing on all $X_{L}(K,B;\tilde{P}_{1},\ldots \tilde{P}_{u})$. By Remark \ref{remark about primitive polynomials}, there exists $\lambda\in \mathcal{O}_{K}^{\times}$ such that $H_{K,\text{aff}}(\lambda F_{L})\leq c_{1}^{d_{K}}H_{K}(F_{L})$. Since $b(F_{L})=b(\lambda F_{L})$,  inequality \eqref{affine count1} gives
\begin{align}
M\lesssim_{K,n} & B^{\frac{n+1}{nd^{\frac{1}{n}}}}\frac{d^{[4,\frac{14}{3},\frac{14}{3},0]-\frac{1}{n}}b(\lambda F_{L})\max\{\log(\mathcal{N}_{K}(\mathfrak{q})),1\}}{H_{K,\text{aff}}(\lambda F_{L})^{\frac{1}{n}\frac{1}{d^{1+\frac{1}{n}}}}\mathcal{N}_{K}(\mathfrak{q})}+d^{1-\frac{1}{n}}\log(B\mathcal{N}_{K}(\mathfrak{q}))\label{affine count2}\\ & +d^{[4-\frac{1}{n},7,\frac{14}{3}-\frac{1}{n},3]}\left(\frac{\max\{\log(\mathcal{N}_{K}(\mathfrak{q})),1\}}{\mathcal{N}_{K}(\mathfrak{q})}+1\right).\nonumber
\end{align}

 From the definition of $F_{L}$ we see that for all $v\in M_{K}$ it holds
$$|\lambda F_{L}|_{v}=\max_{i}| \lambda L^{i}f_{i}X_{n+2}^{d-i}|_{v}\geq |\lambda L^{d}f_{d}|_{v}= |\lambda |_{v}|L^{d}|_{v}|f_{d}|_{v}.$$
Then 
\begin{align}
H_{K,\text{aff}}(\lambda F_{L})=\left(\prod_{v\in M_{K,\infty}}\max\{1,|\lambda F_{L}|_{v}\}\right)^{d_{K}} & \geq \left(\prod_{v\in M_{K,\infty}}|\lambda|_{v}|L^{d}|_{v}|f_{d}|_{v}\right)^{d_{K}}\geq H_{K}(L^{d})H_{K}(f_{d})\geq \frac{B^{d}}{2^{d}}H_{K}(f_{d}).\label{F_L} 
\end{align}
Hence, from Lemma \ref{bound on b(f)} we have the bound

\begin{equation*}
 d^{[4,\frac{14}{3},\frac{14}{3},0]-\frac{1}{n}}\frac{b(\lambda F_{L})}{H_{K,\text{aff}}(\lambda F_{L})^{\frac{1}{n}\frac{1}{d^{1+\frac{1}{n}}}}}\lesssim_{n,K}\left(\frac{B}{2}\right)^{-\frac{1}{nd^{\frac{1}{n}}}}d^{[2,6,2,2]-\frac{1}{n}}\frac{\log(H_{K}(f_{d}))+d\log(B)+d^{[2,1+\frac{1}{n},\frac{8}{3},1+\frac{1}{n}]}}{H_{K}(f_{d})^{\frac{1}{n}\frac{1}{d^{1+\frac{1}{n}}}}}
\label{affine count3}
\end{equation*}
Replacing this on the first summand on the right hand side of \eqref{affine count2}, it follows that this summand is at most

\begin{equation}
\lesssim_{n,K} B^{\frac{1}{d^{\frac{1}{n}}}}d^{[2,6,2,2]-\frac{1}{n}}\frac{(\log(H_{K}(f_{d}))+d\log(B)+d^{[2,1+\frac{1}{n},\frac{8}{3},1+\frac{1}{n}]})\max\{\log(\mathcal{N}_{K}(\mathfrak{q})),1\}}{H_{K}(f_{d})^{\frac{1}{n}\frac{1}{d^{1+\frac{1}{n}}}}\mathcal{N}_{K}(\mathfrak{q})}.
\label{affine count3.50}
\end{equation}

On the other hand, since the reduction modulo $\mathfrak{p}$ of $F_{L}$ for all primes $\mathfrak{p}$ not dividing $L$ is absolutely irreducible whenever $f$ is so, and the number of primes in $\mathcal{O}_{K}$ dividing $L$ is at most $d_{K}$, it follows from Definition \ref{definition of b(f)} that $b(F_{L})$ coincides with $b(f)$ up to a factor $O_{K}(1)$. 
 
This and \eqref{F_L}  give that the first summand on the right hand side \eqref{affine count2} is bounded by

\begin{equation}
B^{\frac{1}{d^{\frac{1}{n}}}} d^{[4,\frac{14}{3},\frac{14}{3},0]-\frac{1}{n}}\frac{b(f)\max\{\log(\mathcal{N}_{K}(\mathfrak{q})),1\}}{H_{K}(f_{d})^{\frac{1}{n}\frac{1}{d^{1+\frac{1}{n}}}}\mathcal{N}_{K}(\mathfrak{q})}.
\label{affine count3.5}
\end{equation}

By inequalities \eqref{affine count3.50} and \eqref{affine count3.5}, it follows that the polynomial $G_{L}$ has degree at most

\begin{align}
M\lesssim_{K,n}& B^{\frac{1}{d^{\frac{1}{n}}}}d^{[2,6,2,2]-\frac{1}{n}}\frac{\min\{\log(H_{K}(f_{d}))+d\log(B)+d^{[2,1+\frac{1}{n},\frac{8}{3},1+\frac{1}{n}]},d^{[2,-\frac{4}{3},\frac{8}{3},-2]}b(f)\}\max\{\log(\mathcal{N}_{K}(\mathfrak{q})),1\}}{H_{K}(f_{d})^{\frac{1}{n}\frac{1}{d^{1+\frac{1}{n}}}}\mathcal{N}_{K}(\mathfrak{q})} \label{bound for the degree}\\ & +d^{1-\frac{1}{n}}\log(B\mathcal{N}_{K}(\mathfrak{q}))+d^{[4-\frac{1}{n},7,\frac{14}{3}-\frac{1}{n},3]}\left(\frac{\max\{\log(\mathcal{N}_{K}(\mathfrak{q})),1\}}{\mathcal{N}_{K}(\mathfrak{q})}+1\right). \nonumber
\end{align}

Thus the polynomial $g(X_{1},\ldots ,X_{n+1}):=G_{L}(X_{1},\ldots ,X_{n+1},L)$ has degree  bounded by the right hand side of \eqref{bound for the degree}, it is not divisible by $f$ and it vanishes on the subset 
$$\mathcal{A}:=\left\{\boldsymbol x \in \mathcal{Z}(f)\cap \mathcal{O}_{K}^{n+1}:\boldsymbol x\text{ specialises to }P_{i}\text{ for all }i\text{ and }H_{K}(\boldsymbol x:L)\leq B\right\}.$$
Now we will see that $\mathcal{Z}(f)_{\text{aff}}(\mathcal{O}_{K},B;P_{1},\ldots ,P_{u})$ is contained in $\mathcal{A}$. For that, it is enough to see that for $\boldsymbol x\in [B]_{\mathcal{O}_{K}}^{n+1}$, it holds $H_{K}(\boldsymbol x:L)\leq B$. In the case when $K$ is a number field, for any $v\in M_{K,\infty}$, $|L|_{v}=|L|_{\infty}^{\frac{n_{v}}{d_{K}}}\leq B^{\frac{n_{v}}{d_{K}^{2}}}$. Then, for all $(x_{1},\ldots ,x_{n+1})\in [B]_{\mathcal{O}_{K}}^{n+1}$, we have 
\begin{align}
H_{K}(x_{1}:\ldots :x_{n+1}:L) =\left(\prod_{v\in M_{K}}\max_{i}\{|x_{i}|_{v},|L|_{v}\}\right)^{d_{K}}  & \leq \left( \prod_{v\in M_{K,\infty}}\max_{i}\left\{|\sigma_{v}(x_{i})|^{\frac{n_{v}}{d_{K}}},B^{\frac{n_{v}}{d_{K}^{2}}}\right\} \right)^{d_{K}} \leq \left( \prod_{v\in M_{K,\infty}}B^{\frac{n_{v}}{d_{K}^{2}}}\right)^{d_{K}}=B.\label{argument for the bound}
\end{align}

In the case when $K$ is a function field, since $v_{\infty}$ was chosen above the place $|\cdot |_{\infty}$, it holds $|L|_{v_{\infty}}=|L|_{\infty}^{\frac{n_{v_{\infty}}}{d_{K}}}\leq B^{\frac{n_{v_{\infty}}}{d_{K}^{2}}}\leq B^{\frac{1}{d_{K}}}$, where $n_{v_{\infty}}$ is the degree of the finite extension $K_{v_{\infty}}/\mathbbm{k}_{\infty}$, with $K_{v_{\infty}}$ and $\mathbbm{k}_{\infty}$ the completions of $K$ and $\mathbbm{k}$ with respect to $v_{\infty}$ and $|\cdot |\infty$, respectively. Then, for all $(x_{1},\ldots ,x_{n+1})\in [B]_{\mathcal{O}_{K}}^{n+1}$, we have
\begin{equation}
H_{K}(x_{1}:\ldots :x_{n+1}:L)=\left(\prod_{v\in M_{K}}\max_{i}\{|x_{i}|_{v},|L|_{v}\}\right)^{d_{K}}  \leq \left( \max\left\{|x_{i}|_{v_{\infty}},B^{\frac{1}{d_{K}}}\right\} \right)^{d_{K}}\leq B.
\label{argument for the bound2}
\end{equation}
This ends the proof in the case when there exists a prime element $L\in \mathcal{P}'$ verifying   that $L\not |f_{0}$.

Now we must deal with the other case, namely, when $\displaystyle \prod_{L\in \mathcal{P}'}L\;|f_{0}.$
If $f_{0}\neq 0$, using \eqref{norm} we deduce
\begin{equation*}
\sum_{L\in \mathcal{P}'}\log(\mathcal{N}_{K}(L))=d_{K}\sum_{L\in \mathcal{P}'}\log(L)\leq \log(\mathcal{N}_{K}(f_{0}))\leq \log(H_{K}(f_{0})).
\end{equation*}
 Since
$$H_{K}(f_{0})\leq H_{K,\text{aff}}(f_{0})=\left(\prod_{v}\max\{1,|f_{0}|_{v}\}\right)^{d_{K}}\leq \left( \prod_{v}\max_{i}\{1,|f_{i}|_{v}\} \right)^{d_{K}}=H_{K,\text{aff}}(f),$$
we deduce
\begin{equation}
d_{K}\sum_{L\in \mathcal{P}'}\log(L)\leq \log(H_{K,\text{aff}}(f)).
\label{affine count4}
\end{equation}
Now, if $H_{K,\text{aff}}(f)\gtrsim_{K} {d+n+1\choose n+1}^{d_{K}{d+n+1\choose n+1}}B^{d{d+n+1\choose n+1}}$, we are done by Lemma \ref{interpolation argument}\eqref{interpolation argument a}. Otherwise, by \eqref{affine count4} we get
\begin{align*}
d_{K}\sum_{L\in \mathcal{P}'}\log(L) \leq d_{K}{d+n+1\choose n+1}\log\left( {d+n+1\choose n+1} \right)+d{d+n+1\choose n+1}\log(B)+O_{K}(1).
\end{align*}
By  \eqref{bound for P'}  there exists a constant $c$ depending only on $\mathbbm{k}$, such that 
\begin{equation}
d_{K}c B^{\frac{1}{d_{K}}}\leq d{d+n+1\choose n+1}\log(B)+d_{K}{d+n+1\choose n+1}\log\left( {d+n+1\choose n+1} \right)+O_{K}(1).
\label{affine count6}
\end{equation}
From \eqref{affine count6} it follows that $B$ is bounded by a polynomial in $d$, and hence $B^{\frac{1}{nd^{\frac{1}{n}}}}\lesssim_{n,K}1$. 

Let us take $L=1$ and consider the polynomial $F_{1}$. Then Theorem \ref{polynomial method2} for $F_{1}$ gives a homogeneous polynomial $G_{1}\in \mathcal{O}_{K}[X_{1},\ldots ,X_{n+2}]$ of degree

\begin{align}
M\lesssim_{K,n} & B^{\frac{1}{d^{\frac{1}{n}}}}d^{[4,\frac{14}{3},\frac{14}{3},0]-\frac{1}{n}}\frac{b(F_{1})\max\{\log(\mathcal{N}_{K}(\mathfrak{q})),1\}}{H_{K}(F_{1})^{\frac{1}{n}\frac{1}{d^{1+\frac{1}{n}}}}\mathcal{N}_{K}(\mathfrak{q})}+d^{1-\frac{1}{n}}\log(B\mathcal{N}_{K}(\mathfrak{q})) \label{affine count7}\\ & +d^{[4-\frac{1}{n},7,\frac{14}{3}-\frac{1}{n},3]}\left(\frac{\max\{\log(\mathcal{N}_{K}(\mathfrak{q})),1\}}{\mathcal{N}_{K}(\mathfrak{q})}+1\right) \nonumber,
\end{align}
not divisible by $F_{1}$, and vanishing on all $X_{1}(K,B;\tilde{P}_{1},\ldots ,\tilde{P}_{u})$. Now we use Remark \ref{remark about primitive polynomials} to find $\lambda \in\mathcal{O}_{K}^{\times}$ such that $H_{K,\text{aff}}(\lambda F_{1})\leq c_{1}^{d_{K}}H_{K}(F_{1})$. The same argument used to prove inequality \eqref{F_L} gives $H_{K,\text{aff}}(\lambda F_{1})\geq H_{K}(f_{d})$. This, together with the fact that $b(F_{1})=b(f)$, Lemma \ref{b(f)} and Lemma \ref{bound on b(f)} allows us to conclude the bound

\begin{equation}
d^{[4,\frac{14}{3},\frac{14}{3},0]-\frac{1}{n}}\frac{b(\lambda F_{1})}{H_{K,\text{aff}}(\lambda F_{1})^{\frac{1}{n}\frac{1}{d^{1+\frac{1}{n}}}}}\lesssim_{n,K}d^{[2,6,2,2]-\frac{1}{n}}\frac{\min\{\log H_{K}(f_{d})+d^{[2,1+\frac{1}{n},\frac{8}{3},1+\frac{1}{n}]},d^{[2,-\frac{4}{3},\frac{8}{3},-2]}b(f)\}}{H_{K}(f_{d})^{\frac{1}{n}\frac{1}{d^{1+\frac{1}{n}}}}}.
\label{affine count8}
\end{equation}

From \eqref{argument for the bound} and \eqref{argument for the bound2} with $L=1$, and from \eqref{affine count7} and \eqref{affine count8}, it follows that $g(X_{1},\ldots ,X_{n+1}):=G_{1}(X_{1},\ldots ,X_{n+1},1)$ verifies the conclusion of Theorem \ref{ell-venk}.

In order to finish the proof of Theorem \ref{ell-venk}, it remains to cover the case when $f_{0}=0$, namely $f(0)=0$. By the Combinatorial nullstellensatz (see \cite[Theorem 1.2]{Alon}), we may find $\boldsymbol A=(a_{1},\ldots ,a_{n+1})\in \mathcal{O}_{\mathbbm{k}}^{n+1}$ with $f(a_{1},\ldots ,a_{n+1})\neq 0$ and $H_{\mathbbm{k}}(a_{i})\leq d$ for all $1\leq i\leq n+1$. Let us consider the polynomial $\tilde{f}(\boldsymbol X):=f(\boldsymbol X+\boldsymbol A)$; we see that $\tilde{f}(0)\neq 0$, $H_{K}(f_{d})=H_{K}(\tilde{f_{d}})$, and $b(\tilde{f})=b(f)$. Thus, reasoning as in the two previous cases with $\tilde{f}$ in place of $f$ and $\tilde{B}=B+d$ in place of $B$ we obtain a polynomial $\tilde{g}\in \mathcal{O}_{K}[X_{1},\ldots ,X_{n+1}]$ of degree

\begin{align*}
M\sim_{K,n} & \tilde{B}^{\frac{1}{d^{\frac{1}{n}}}}d^{[2,6,2,2]-\frac{1}{n}}\frac{\min\{\log(H_{K}(f_{d}))+d\log(\tilde{B})+d^{[2,1+\frac{1}{n},\frac{8}{3},1+\frac{1}{n}]},d^{[2,-\frac{4}{3},\frac{8}{3},-2]}b(f)\}\max\{\log(\mathcal{N}_{K}(\mathfrak{q})),1\}}{H_{K}(f_{d})^{\frac{1}{n}\frac{1}{d^{1+\frac{1}{n}}}}\mathcal{N}_{K}(\mathfrak{q})}\\ &\;+\;d^{1-\frac{1}{n}}\log(\tilde{B}\mathcal{N}_{K}(\mathfrak{q}))+d^{[4-\frac{1}{n},7,\frac{14}{3}-\frac{1}{n},3]}\left(\frac{\max\{\log(\mathcal{N}_{K}(\mathfrak{q})),1\}}{\mathcal{N}_{K}(\mathfrak{q})}+1\right),
\end{align*}
\noindent not divisible by $\tilde{f}$, and vanishing on $\mathcal{Z}(\tilde{f})_{\text{aff}}(\mathcal{O}_{K},\tilde{B};P_{1},\ldots, P_{u})$. Then $g(\boldsymbol X)=\tilde{g}(\boldsymbol X-\boldsymbol A)$ verifies the conclusion of Theorem \ref{ell-venk}. 
\end{proof}

As a consequence of Theorem \ref{ell-venk}, B\'ezout's theorem, and the estimate $\frac{\log(H_{K}(f_{d}))}{H_{K}(f_{d})^{\frac{1}{d^{2}}}}\lesssim d^{2}$, we get the next improvement of the bounds of  Bombieri and Pila.

\begin{corollary}
Let $K$ be a global field of degree $d_{K}$. For any irreducible polynomial $f\in \mathcal{O}_{K}[X_{1},X_{2}]$ of degree $d$ and any $B\geq 1$, if $\mathcal{Z}(f)\subseteq \mathbb{A}_{K}^{2}$ denotes the corresponding curve, it holds
\begin{align*}
N_{\emph{\text{aff}}}(\mathcal{Z}(f),\mathcal{O}_{K},B) & \lesssim_{K}B^{\frac{1}{d}}\frac{\min\{d^{[2,6,2,2]}\log(H_{K}(f_{d}))+d^{[3,7,3,3]}\log(B)+d^{[4,8,\frac{14}{3},4]},d^{[4,\frac{14}{3},\frac{14}{3},0]} b(f)\}}{H_{K}(f_{d})^{\frac{1}{d^{2}}}}\\ &+d\log(B)+d^{[4,8,\frac{14}{3},4]} \lesssim_{K}d^{[3,7,3,3]}B^{\frac{1}{d}}(\log(B)+d^{[1,1,\frac{5}{3},1]}).
\end{align*}
\label{Walsh affine corollary}
\end{corollary}

Now we are in position to prove Theorem \ref{teorema1.5} of the introduction by extending Corollary \ref{Walsh affine corollary}  to general curves using a projection argument. This is accomplished by means of an effective change of variables, as in \cite{P2}.

We also have:

\begin{theorem}
Let $K$ be a global field of degree $d_{K}$. Given $n>1$, for any integral curve $C\subseteq \mathbb{A}_{K}^{n}$ of degree $d$, it holds
$$N_{\emph{\text{aff}}}(C,\mathcal{O}_{K},B)\lesssim_{K,n}d^{[3,7,3,3]}B^{\frac{1}{d}}(\log(B)+d^{[1,1,\frac{5}{3},1]}).$$
\label{Walsh general affine curves}
\end{theorem}

\begin{proof}
We argue as in the proof of Theorem \ref{Walsh general curves}, first reducing Theorem \ref{Walsh general affine curves} to the case of a planar curve and then we apply Corollary \ref{Walsh affine corollary}. We remark that the change of variables in \cite[Theorem 4.3]{P2} works in the affine setting, namely one takes the projective closure of $\overline{C}$ of $C$ in $\mathbb{P}_{K}^{n}$ and then one constructs an adequate finite morphism $\varphi:\overline{C}\to \mathbb{P}^{2}$, which maps $C$ to the affine part $\mathbb{A}^{2}=\mathbb{P}^{2}\cap \varphi(\mathbb{A}^{n})$.  
\end{proof}

\begin{remark}
As in Remark \ref{curvas no absolutamente irreducibles}, if $C\subseteq \mathbb{A}_{K}^{n}$ is an integral curve which is not geometrically irreducible, then $N_{\text{aff}}(C,\mathcal{O}_{K},B)\leq d^{2}$.
\end{remark}

\section{Partitioning a hypersurface in $\mathbb{P}^{3}$ by curves of small degree}\label{Section 6}

The purpose of this section is to prove Theorem \ref{bound for the non-singular locus outside a subset}, which is an affine variant of \cite[Theorem 3.16]{Salberger} for varieties defined over global fields. This will be used in Section \ref{Section 7} to prove the dimension growth conjecture in the cases when the variety has degree $4\leq d\leq 15$ if $K$ is a number field, and in the cases when the variety has degree $4\leq d\leq 65$ if $K$ is a function field. 

In \cite[$\S 3$]{Salberger} Salberger developed a technique that allowed him to partition the $\mathbb{Q}$-rational points of bounded height of a projective variety $X\subseteq \mathbb{P}_{\mathbb{Q}}^{n}$ in a small number of subvarieties of codimension at most $2$. We follow the same strategy of \cite{Salberger}, that is, we construct a family of hypersurfaces which vanish on the rational points of $X$ of bounded height and prescribed reduction modulo $\mathfrak{p}$ for many primes $\mathfrak{p}$. The degree of these hypersurfaces is bounded also in terms of the density of this subset of primes, and this is reflected in the parameter $\mathfrak{q}$ in the statements of the theorems of the previous sections. This will allow us to control the number and the degree of the  subvarieties of the partition. The main technical tool is Proposition \ref{construction of a family of cycles} which is a variant of \cite[Main Lemma 3.2]{Salberger}.  For its proof, instead of using \cite[Theorem 2.2]{Salberger} and \cite[Lemma 2.8]{Salberger} we rely on the results of Section \ref{subsection 5.2} and we give a streamlined construction of a suitable large subset of primes with small norm. This allows us to simplify the presentation of the proof given in \cite{Salberger}. 

Before ending this brief introduction we mention that all bounds in this section can be made effective on the dependence on the degree $d$. Unlike the previous sections, we choose not to make this dependence explicit because this would further complicate the proof  and the final bound obtained by the methods of this section is double exponential on $d$.

For any $X\subseteq \mathbb{P}_{K}^{n}$ projective variety, given $B\in \mathbb{R}_{>0}$ we will denote
$$X_{\text{aff}}(\mathcal{O}_{K},B):=\{(x_{1},\ldots, x_{n})\in [B]_{\mathcal{O}_{K}}^{n}:(1:x_{1}:\ldots :x_{n})\in X(K)\},$$
$$N_{\text{aff}}(X,\mathcal{O}_{K},B):=|X_{\text{aff}}(\mathcal{O}_{K},B)|.$$
Moreover, if $\{\mathfrak{p}_{1},\ldots ,\mathfrak{p}_{u}\}$ is a subset of primes such that for all $1\leq i\leq u$ we let $P_{i}$ be a  non-singular $\mathbb{F}_{\mathfrak{p}_{i}}$-point on $X_{\mathfrak{p}_{i}}$, we denote 
$$X_{\text{aff}}(\mathcal{O}_{K},B;P_{1},\ldots ,P_{u}):=\{\boldsymbol x\in X_{\text{aff}}(\mathcal{O}_{K},B):\boldsymbol x\text{ specialises to }P_{i} \text{ in }X_{\mathfrak{p}_{i}}\text{ for all }i\},$$

We will use the following notation. Given a hypersurface $X\subseteq \mathbb{P}_{K}^{n+1}$ and $\boldsymbol x\in X$, we will denote
\begin{equation*}
\pi_{\boldsymbol x}:=\smashoperator[r]{\prod_{\substack{\mathfrak{p}\in M_{K,\text{fin}}\\ \boldsymbol x\text{ speacilises to a}\\ \text{ singular }\mathbb{F}_{\mathfrak{p}}-\text{point in }X_{\mathfrak{p}} }}}\; \mathfrak{p} \qquad \text{ and } \qquad \pi_{X}:=\smashoperator[r]{\prod_{\substack{\mathfrak{p}\in M_{K,\text{fin}}\\ X_{\mathfrak{p}} \text{ is not geometrically integral}}}}\; \mathfrak{p}.
\end{equation*}
The following lemmas, which generalize Lemma \cite[Lemma 3.1]{Salberger} and \cite[Lemma 3.2]{Salberger} give estimates for $\pi_{\boldsymbol x}$ and $\pi_{X}$.

\begin{lemma}
Let $X\subseteq \mathbb{P}_{K}^{n+1}$ be a geometrically integral hypersurface defined over $K$ by a homogeneous polynomial $f\in \mathcal{O}_{K}[X_{0},\ldots ,X_{n+1}]$ of degree $d\geq 2$.  
Let us suppose that $X$ is the only hypersurface of degree $d$ containing $X_{\emph{\text{aff}}}(\mathcal{O}_{K},B)$. Then for any non-singular point $\boldsymbol x$ in $X_{\emph{\text{aff}}}(\mathcal{O}_{K},B)$, it holds $\log (\mathcal{N}_{K}(\pi_{\boldsymbol x}))\lesssim_{n,K,d}\log(B)$. 
\label{reduction to singular points}
\end{lemma}
\begin{proof}
By Lemma \ref{interpolation argument}\eqref{interpolation argument b} and the assumption on $X$, we have that $H_{K,\text{aff}}(f)\lesssim_{K}{d+n+1\choose n+1}^{d_{K}{d+n+1\choose n+1}}B^{d{d+n+1\choose n+1}}$. Now, given a non-singular point $\boldsymbol x$, there exists $1\leq j\leq n+1$ such that $\frac{\partial f}{\partial X_{j}}(\boldsymbol x)\neq 0$. Then for all prime $\mathfrak{p}$ appearing in the product $\pi_{\boldsymbol x}$, we have $\mathfrak{p}|\frac{\partial f}{\partial X_{j}}(\boldsymbol x)$. By \eqref{polynomialheight2} and \eqref{norm}, this implies:
\begin{align*}
\mathcal{N}_{K}(\pi_{\boldsymbol x})\leq H_{K}\left(\frac{\partial f}{\partial X_{j}}(\boldsymbol x)\right) & \leq {d+n\choose n+1}^{d_{K}}H_{K,\text{aff}}\left(\frac{\partial f}{\partial X_{j}}\right)H_{K}(1:\boldsymbol x)^{d-1} \lesssim_{n,K}d^{(n+2)d_{K}}{d+n+1\choose n+1}^{d_{K}{d+n+1\choose n+1}}B^{d{d+n+1\choose n+1}+d-1}
\end{align*}

By taking logarithms, we conclude $\log (\mathcal{N}_{K}(\pi_{\boldsymbol x}))\lesssim_{n,K,d}\log(B)$.
\end{proof}

\begin{lemma}
Let $X\subseteq \mathbb{P}_{K}^{n+1}$ be a geometrically integral hypersurface defined by a homogeneous  polynomial $f\in \mathcal{O}_{K}[X_{0},\ldots ,X_{n+1}]$ of degree $d\geq 2$. Then, either there is a polynomial $g\in \mathcal{O}_{K}[X_{1},\ldots, X_{n+1}]$ of degree $d$ not divisible by $f$ which contains $X_{\emph{\text{aff}}}(\mathcal{O}_{K},B)$, or $\log(\mathcal{N}_{K}(\pi_{X}))\lesssim_{n,K,d}\log(B)$. 
\label{no geo integro3}
\end{lemma}

\begin{proof}
If there is no polynomial $g\in \mathcal{O}_{K}[X_{0},\ldots ,X_{n+1}]$ of degree $d$, not divisible by $f$, which contains $X_{\text{aff}}(\mathcal{O}_{K},B)$, by Lemma \ref{interpolation argument}\eqref{interpolation argument b} it holds $H_{K,\text{aff}}(f)\lesssim_{K}  {d+n+1\choose n+1}^{d_{K}{d+n+1\choose n+1}} B^{d{d+n+1\choose n+1}}$. Now, note that any prime $\mathfrak{p}| \pi_{X}$ is also a prime for which $f$ modulo $\mathfrak{p}$ is not absolutely irreducible. Then, following the proof of inequalities \eqref{no geo integro1} and \eqref{no geo integro2} in Lemma \ref{b(f)}, we have
$$\prod_{\mathfrak{p}\in \pi_{X}}\mathcal{N}_{K}(\mathfrak{p})\leq\begin{cases}d^{3(d^{2}-1)d_{K}^{2}}\left[ {d+n\choose n}3^{d} \right]^{(d^{2}-1)d_{K}^{2}} H_{K,\text{aff}}(f)^{(d^{2}-1)}& \text{ if char}(K)=0, \\ H_{K,\text{aff}}(f)^{12d^{6}} & \text{ if }0<\text{char}(K)\leq d(d-1),\\ H_{K,\text{aff}}(f)^{d^{2}-1} & \text{ if char}(K)>d(d-1). \end{cases}$$  
The proof  finishes by taking logarithms.
\end{proof}

Now we are in condition to prove the main technical lemma of this section, which is a variant of \cite[Main Lemma 3.2]{Salberger}. In what follows, by a prime divisor on $X$ we shall mean a closed integral subscheme of codimension one.

\begin{proposition}
Let $n\geq 2$ and let $X\subseteq \mathbb{P}_{K}^{n+1}$ be a geometrically integral hypersurface defined by a polynomial $f\in \mathcal{O}_{K}[X_{0},\ldots ,X_{n+1}]$ of degree $d\geq 2$. Then there exists a family $\{D_{\gamma}\}_{\gamma\in \Gamma}$ of prime divisors on $X$, and a (possibly empty) subset $Z(\mathfrak{q})\}_{\mathfrak{q}\in \mathcal{Q}}$ of effective cycles of codimension $2$ on X with $\mathcal{Q}$ a subset of ideals, such that the following conditions hold:
\begin{enumerate}
\item The index subset $\Gamma$ has size $|\Gamma|\lesssim_{n,K,d}B^{\frac{1}{d^{\frac{1}{n}}}}\log(B)$ and for each $\gamma\in \Gamma$, it holds $\deg(D_{\gamma})\lesssim_{n,K,d}\log(B)^{2}$. \label{construction divisors}
\item There is some positive constant $c\lesssim_{n,K,d}1$ such that
$$\sum_{\mathfrak{q}\in \mathcal{Q}}\deg(Z(\mathfrak{q}))\lesssim_{n,K,d}B^{\frac{n}{d^{\frac{1}{n}}}}\exp\left(c\frac{\log(B)}{\log(\log(B))}\right).$$ \label{construction cycles}
\item For each non-singular point $\boldsymbol x\in X_{\emph{\text{aff}}}(\mathcal{O}_{K},B)$ which does not lie in $\bigcup_{\gamma\in \Gamma}D_{\gamma}$ there exists $\mathfrak{q}\in \mathcal{Q}$ such that $\boldsymbol x\in \emph{\text{Supp}}(Z(\mathfrak{q}))$ and such $\boldsymbol x$ specialises to a non-singular $\mathbb{F}_{\mathfrak{p}_{i}}$-point for each prime $\mathfrak{p}_{i}$ dividing $\mathfrak{q}$. \label{cycles cover missing points}
\end{enumerate}
\label{construction of a family of cycles}
\end{proposition}

When $n=2$, Proposition \ref{construction of a family of cycles} says that we may partition a hypersurface $X$ in such a way that those points $\boldsymbol x\in X_{\text{aff}}(\mathcal{O}_{K},B)$ either lie in a small family of curves of low degree, or lie in an exceptional subset of controlled size.  

Roughly speaking, the proof of Proposition \ref{construction of a family of cycles} is as follows. First, we localize a big subset of primes $\mathcal{P}$ such that for all $\mathfrak{p}\in \mathcal{P}$, the reduction $X_{\mathfrak{p}}$ is geometrically irreducible, and all non-singular $\boldsymbol x\in X_{\text{aff}}(\mathcal{O}_{K},B)$ specialises to a non-singular point in $X_{\mathfrak{p}}$ for many primes $\mathfrak{p}\in \mathcal{P}$. Then we construct a hypersurface $Y$ that covers all points of $X_{\text{aff}}(\mathcal{O}_{K},B)$. By the second assumption on the subset $\mathcal{P}$, a large family $\{D_{\gamma}\}_{\gamma\in \Gamma}$ of irreducible components of $ X\cap Y$ will have points which specialise to non-singular points for many primes in $\mathcal{P}$. Then by the first assumption on $\mathcal{P}$ and the Lang-Weil estimate, it will turn out that those non-singular points missed by $\{D_{\gamma}\}_{\gamma}$ lie in a subvariety of codimension $2$ with controlled degree.

\begin{proof}
First let us suppose that $B\lesssim_{K}1$. Then cover $X_{\text{aff}}(\mathcal{O}_{K},B)$ by $\lesssim_{n,K}1$ hyperplanes and let $(D_{\gamma})_{\gamma\in \Gamma}$ be this family of  hyperplane sections and let $\mathcal{Q}=\emptyset$. 

Now, let us suppose that $X_{\text{aff}}(\mathcal{O}_{K},B)$ is contained in another hypersurface $Y$ of degree $d$. In this case let $(D_{\gamma})_{\gamma\in \Gamma}$ be the components of $X\cap Y$.  By Bezout's Theorem \cite[Example 8.4.6]{Fulton}, $|\Gamma|\leq d^{2}$ and $\deg(D_{\gamma})\leq d^{2}$ for each $\gamma\in \Gamma$. Let $\mathcal{Q}=\emptyset$. 

In the previous two cases, all the assertions of Proposition \ref{construction of a family of cycles} are verified. Thus, we may suppose that $X\subseteq \mathbb{P}_{K}^{n+1}$ is the only hypersurface of degree $d$ containing $X_{\text{aff}}(\mathcal{O}_{K},B)$ and that $B\gtrsim_{K}1$.  The next step will be to construct a family of auxiliar hypersurfaces. 

By Lemma \ref{reduction to singular points} and Lemma \ref{no geo integro3}, there are positive constants $k_{1}\lesssim_{n,K,d}1$, $\delta\lesssim_{n,K,d}1$ such that
\begin{equation}
\text{for all non-singular } \boldsymbol x\in X_{\text{aff}}(\mathcal{O}_{K},B),\; \mathcal{N}_{K}(\pi_{\boldsymbol x})\leq B^{k_{1}}\text{ and }\sum_{\mathfrak{p}|\pi_{X}}\log(\mathcal{N}_{K}(\mathfrak{p}))\leq \delta\log(B).
\label{k_{1}}
\end{equation}
We begin by localizing an adequate large subset of primes.
\begin{claim}
Given $c'\geq 1$, for all $c>2(\delta+c')$, if $I:=[c\log(B),2c\log(B)]$ we have $\sum_{\substack{\mathfrak{p}\not |\pi_{X}\\ \mathcal{N}_{K}(\mathfrak{p})\in I}}\log(\mathcal{N}_{K}(\mathfrak{p}))\geq c'\log(B)$ for $B\gtrsim_{K}1$.
\label{enough primes}
\end{claim}
\begin{proof}[Proof of Claim \ref{enough primes}]
Let $c\geq 1$.  Let us suppose that $\sum_{\substack{\mathfrak{p}\not |\pi_{X}\\ \mathcal{N}_{K}(\mathfrak{p})\in I}}\log(\mathcal{N}_{K}(\mathfrak{p}))\leq c'\log(B)$. Then by \eqref{Bertrand} for $B\gtrsim_{K}1$ it holds:
$$\frac{1}{2}c\log(B) \leq \sum_{\substack{\mathfrak{p}\\ \mathcal{N}_{K}(\mathfrak{p})\in I}}\log(\mathcal{N}_{K}(\mathfrak{p}))=\sum_{\substack{\mathfrak{p}|\pi_{X}\\\mathcal{N}_{K}(\mathfrak{p})\in I}}\log(\mathcal{N}_{K}(\mathfrak{p}))+\sum_{\substack{\mathfrak{p}\not |\pi_{X}\\ \mathcal{N}_{K}(\mathfrak{p})\in I}}\log(\mathcal{N}_{K}(\mathfrak{p}))\leq (\delta+c')\log(B).$$
Thus we arrive at a contradiction taking $c>2(\delta+c')$. 
\end{proof}

Let $c':=k_{1}+1$,  and let us take $c:=2(\delta+k_{1}+1)+1$, $I:=[c\log(B),2c\log(B)]$.  Let $\mathfrak{q}^{\ast}:=\prod_{\substack{\mathfrak{p}\not |\pi_{X}\\ \mathcal{N}_{K}(\mathfrak{p})\in I}}\mathfrak{p}$. By \eqref{landau2} and Claim \ref{enough primes}, 
\begin{equation}
\left(k_{1}+1\right)\log(B)\leq \log(\mathcal{N}_{K}(\mathfrak{q}^{\ast}))\leq 2c_{2,K}c\log(B).
\label{enough primes2}
\end{equation} 

Let $\mathfrak{q}=\prod_{i=1}^{u}\mathfrak{p}_{i}\;\;|\mathfrak{q}^{\ast}$. For each $i$ denote $(X_{\mathfrak{p}_{i}})_{\text{ns}}$ for the non-singular locus of $X_{\mathfrak{p}_{i}}$, and let $(P_{1},\ldots ,P_{u})\in \prod_{i=1}^{u}(X_{\mathfrak{p}_{i}})_{\text{ns}}(\mathbb{F}_{\mathfrak{p}_{i}})$. We choose a hypersurface $Y(P_{1},\ldots ,P_{u})$ as the homogenization of the one in Theorem \ref{ell-venk}. By \eqref{enough primes2} it holds 
\begin{align}
\deg(Y(P_{1},\ldots, P_{u})) & \lesssim_{K,n,d}B^{\frac{1}{d^{\frac{1}{n}}}}\log(B)\mathcal{N}_{K}(\mathfrak{q})^{-1}\log(\mathcal{N}_{K}(\mathfrak{q}))+\log(B\mathcal{N}_{K}(\mathfrak{q}))+\log(\mathcal{N}_{K}(\mathfrak{q}))+1 \label{degree for q_{u}}\\ & \lesssim_{K,n,d} B^{\frac{1}{d^{\frac{1}{n}}}}\log(B)\mathcal{N}_{K}(\mathfrak{q})^{-1}\log(\mathcal{N}_{K}(\mathfrak{q}))+\log(B)+1. \nonumber
\end{align}

When $\mathfrak{q}=(1)$ (in which case we use the convention $u=0$), there is a projective hypersurface $Y$ vanishing on $X_{\text{aff}}(\mathcal{O}_{K},B)$ and not identically zero on $X$, with
\begin{equation}
\deg(Y)\lesssim_{K,n,d}B^{\frac{1}{d^{\frac{1}{n}}}}\log(B)+\log(B)+1\lesssim_{K,n,d}B^{\frac{1}{d^{\frac{1}{n}}}}\log(B).
\label{degree on Y}
\end{equation}  

Now we will define the subset $\Gamma$ and the subset of prime divisors $D_{\gamma}\subseteq X$. This divisors will be the irreducible components of $X\cap Y$ which are contained in $Y(P_{1},\ldots ,P_{u})$ for some sequence $(P_{1},\ldots ,P_{u})\in \prod_{i=1}^{u}(X_{\mathfrak{p}_{i}})_{\text{ns}}(\mathbb{F}_{\mathfrak{p}_{i}})$ with $\mathfrak{q}=\prod_{i=1}^{u}\mathfrak{p}_{i}$ verifying $\mathcal{N}_{K}(\mathfrak{q})\geq B^{\frac{1}{d^{\frac{1}{n}}}}$. Then by \eqref{degree for q_{u}}, \eqref{degree on Y} and B\'ezout's theorem \cite[Example 8.4.6]{Fulton} it holds
\begin{equation*}
\deg(D_{\gamma})\leq \deg(X)\deg(Y(P_{1},\ldots, P_{u}))\lesssim_{n,K,d}(\log(B))^{2}, 
\end{equation*}
\begin{equation*}
|\Gamma|\leq \deg(X\cap Y)\leq \deg(X)\deg(Y)\lesssim_{n,K,d}B^{\frac{1}{d^{\frac{1}{n}}}}\log(B), 
\end{equation*} 
thus $\{D_{\gamma}\}_{\gamma}$ verifies the conditions in Proposition \ref{construction of a family of cycles}\eqref{construction divisors}.

\begin{claim}
For all $\boldsymbol x\in X_{\emph{\text{aff}}}(\mathcal{O}_{K},B)$ non-singular not lying in $\bigcup_{\gamma\in \Gamma}D_{\gamma}$ there is some $\mathfrak{q}|\mathfrak{q}^{\ast}$, relatively prime to $\pi_{\boldsymbol x}\pi_{X}$ and such that $B^{\frac{1}{d^{\frac{1}{n}}}}\leq \mathcal{N}_{K}(\mathfrak{q})\leq 2cB^{\frac{1}{d^{\frac{1}{n}}}}\log(B)$.
\label{existence of q}
\end{claim}

\begin{proof}
Let $\boldsymbol x\in X_{\text{aff}}(\mathcal{O}_{K},B)$ be a non-singular point which does not lie in $\bigcup_{\gamma\in \Gamma}D_{\gamma}$. From \eqref{k_{1}} and \eqref{enough primes2} it follows that there is some ideal factor $\mathfrak{q}$ of $\mathfrak{q}^{\ast}$, relatively prime to $\pi_{\boldsymbol x}\pi_{X}$, that verifies $\mathcal{N}_{K}(\mathfrak{q})\geq B^{\frac{1}{d^{\frac{1}{n}}}}$. Let $\mathfrak{q}|\mathfrak{q}^{\ast}$ of minimal norm such that it is relatively prime with $\pi_{\boldsymbol x}\pi_{X}$ and $\mathcal{N}_{K}(\mathfrak{q})\geq B^{\frac{1}{d^{\frac{1}{n}}}}$. Let us suppose that $\mathcal{N}_{K}(\mathfrak{q})>2cB^{\frac{1}{d^{\frac{1}{n}}}}\log(B)$. Since any prime $\mathfrak{p}|\mathfrak{q}^{\ast}$ has norm in $I$, there is some prime $\mathfrak{p}|\mathfrak{q}$ with $\mathcal{N}_{K}(\mathfrak{p})\leq 2c\log(B)$. Thus $\mathcal{N}_{K}(\mathfrak{q}\mathfrak{p}^{-1})\geq B^{\frac{1}{d^{\frac{1}{n}}}}$, which contradicts the minimality of $\mathfrak{q}$. 
\end{proof}

If $\boldsymbol x$ and $\mathfrak{q}=\prod_{i=1}^{u}\mathfrak{p}_{i}$ are as in Claim \ref{existence of q}, then $\boldsymbol x$ specialises to a non-singular $\mathbb{F}_{\mathfrak{p}_{i}}$-rational point $P_{i}$ on $X_{\mathfrak{p}_{i}}$ for all $i$, thus $\boldsymbol x\in Y(P_{1},\ldots ,P_{u})$. Let $D_{\boldsymbol x}$ be an irreducible component of $X\cap Y$ containing $\boldsymbol x$. It can not hold that $D_{\boldsymbol x}\subseteq Y(P_{1},\ldots ,P_{u})$,  since otherwise $\boldsymbol x\in \bigcup_{\gamma\in \Gamma}D_{\gamma}$. Hence, there is some $t\in \{0,\ldots ,u-1\}$ with $D_{\boldsymbol x}\subseteq Y(P_{1},\ldots, P_{t})$ but $D_{\boldsymbol x}\not \subseteq Y(P_{1},\ldots, P_{t+1})$. If $D_{\boldsymbol x}'$ is an irreducible component of $X\cap Y(P_{1},\ldots, P_{t+1})$ containing $\boldsymbol x$, then $\boldsymbol x\in D_{\boldsymbol x}\cap D_{\boldsymbol x}'$. This will motivate the construction of the cycles.

Given $\mathfrak{q}=\mathfrak{q}_{t+1}=\prod_{i=1}^{t+1}\mathfrak{p}_{i}$ with $t\geq 0$, let $Z(P_{1},\ldots ,P_{t+1})$ be the formal sum of all the irreducible components of $D\cap D'$ where $D$ is an irreducible component of $X\cap Y(P_{1},\ldots ,P_{t})$ and $D'\neq D$ is an irreducible component of $X\cap Y(P_{1},\ldots ,P_{t+1})$. By B\'ezout's theorem \cite[Example 8.4.6]{Fulton}, inequalities \eqref{degree for q_{u}} and \eqref{degree on Y}, if $\mathfrak{q}_{t}:=\mathfrak{q}_{t+1}\mathfrak{p}_{t+1}^{-1}$, then it holds
\begin{align}
& \deg(Z(P_{1},\ldots ,P_{t+1})) \leq \deg(X)\deg(Y(P_{1},\ldots ,P_{t}))\deg(Y(P_{1},\ldots, P_{t+1})) \label{bound for the degree of Z(P)} \\ & \lesssim_{K,n,d} \left( B^{\frac{1}{d^{\frac{1}{n}}}}\log(B)\mathcal{N}_{K}(\mathfrak{q}_{t+1})^{-1}\mathcal{N}_{K}(\mathfrak{p}_{t+1})\log(\mathcal{N}_{K}(\mathfrak{q}_{t}))+\log(B)+1\right)\left(B^{\frac{1}{d^{\frac{1}{n}}}}\log(B)\mathcal{N}_{K}(\mathfrak{q}_{t+1})^{-1}\log(\mathcal{N}_{K}(\mathfrak{q}_{t+1}))+\log(B)+1.\right). \nonumber \\ &  \lesssim_{K,n,d} \left( B^{\frac{1}{d^{\frac{1}{n}}}}\log(B)\mathcal{N}_{K}(\mathfrak{q}_{t+1})^{-1}(\log(B))^{2}+\log(B)+1\right)\left(B^{\frac{1}{d^{\frac{1}{n}}}}\log(B)\mathcal{N}_{K}(\mathfrak{q}_{t+1})^{-1}\log(B)+\log(B)+1.\right). \nonumber
\end{align}

We define
$$Z(\mathfrak{q})=Z(\mathfrak{q}_{t+1}):=\sum_{(P_{1},\ldots ,P_{t+1})\in \prod_{i=1}^{t+1}(X_{\mathfrak{p}_{i}})_{\text{ns}}(\mathbb{F}_{\mathfrak{p}_{i}})}Z(P_{1},\ldots ,P_{t+1}).$$
In order to bound the degree of $Z(\mathfrak{q}_{t+1})$ we need to estimate the cardinal of $\prod_{i=1}^{t+1}(X_{\mathfrak{p}_{i}})_{\text{ns}}(\mathbb{F}_{\mathfrak{p}_{i}})$.
\begin{claim}
It holds $\left| \prod_{i=1}^{t+1}(X_{\mathfrak{p}_{i}})_{\text{ns}}(\mathbb{F}_{\mathfrak{p}_{i}}) \right|\leq \mathcal{N}_{K}(\mathfrak{q}_{t+1})^{n}\exp\left(C\frac{\log(B)}{\log(\log(B))}\right)$ for some $C\lesssim_{K,n,d}1$.
\label{enough primes3}
\end{claim}
\begin{proof}[Proof of Claim \ref{enough primes3}]
By definition, for all $i$, $X_{\mathfrak{p}_{i}}$ is geometrically irreducible. Then the Lang-Weil estimate gives $|X_{\mathfrak{p}_{i}}(\mathbb{F}_{\mathfrak{p}_{i}})|\leq \mathcal{N}_{K}(\mathfrak{p}_{i})^{n}+A\mathcal{N}_{K}(\mathfrak{p}_{i})^{n-\frac{1}{2}}$ for some positive constant $A\lesssim_{d,n}1$. This, together with the facts that $t\lesssim_{K,n,d} \log(B)$ and that for all $i$, $\mathfrak{p}_{i}\in I$, thus $\mathcal{N}_{K}(\mathfrak{p}_{i})\geq \frac{1}{2}c\log(B)$, yield
$$\left| \prod_{i=1}^{t+1}(X_{\mathfrak{p}_{i}})_{\text{ns}}(\mathbb{F}_{\mathfrak{p}_{i}}) \right|\leq \mathcal{N}_{K}(\mathfrak{q}_{t+1})^{n}\prod_{i=1}^{t+1}\left(1+A\mathcal{N}_{K}(\mathfrak{p}_{i})^{-\frac{1}{2}}\right)\leq \mathcal{N}_{K}(\mathfrak{q}_{t+1})^{n}\left(1+\frac{2^{\frac{1}{2}}A}{c^{\frac{1}{2}}(\log(B))^{\frac{1}{2}}}\right)^{t+1}\leq \mathcal{N}_{K}(\mathfrak{q}_{t+1})^{n}\exp\left( C\frac{\log(B)}{\log(\log(B))}\right),$$
for some $C\lesssim_{K,n,d}1$. 
\end{proof}

By \eqref{bound for the degree of Z(P)} and Claim \ref{enough primes3},  it follows that
\begin{align*}
\deg(Z(\mathfrak{q}_{t+1}))& \lesssim_{K,n,d} \left( B^{\frac{1}{d^{\frac{1}{n}}}}\log(B)\mathcal{N}_{K}(\mathfrak{q}_{t+1})^{-1}(\log(B))^{2}+\log(B)+1\right)\left(B^{\frac{1}{d^{\frac{1}{n}}}}\log(B)\mathcal{N}_{K}(\mathfrak{q}_{t+1})^{-1}\log(B)+\log(B)+1\right)\\ & \qquad \quad \mathcal{N}_{K}(\mathfrak{q}_{t+1})^{n}\exp\left(C\frac{\log(B)}{\log(\log(B))}\right) \nonumber \\ & \lesssim_{K,n,d} \Big(B^{\frac{2}{d^{\frac{1}{n}}}}\mathcal{N}_{K}(\mathfrak{q}_{t+1})^{n-2}(\log(B))^{5}+B^{\frac{1}{d^{\frac{1}{n}}}}\mathcal{N}_{K}(\mathfrak{q}_{t+1})^{n-1}(\log(B))^{4}+B^{\frac{1}{d^{\frac{1}{n}}}}\mathcal{N}_{K}(\mathfrak{q}_{t+1})^{n-1}(\log(B))^{3} \nonumber \\ & \qquad \quad +B^{\frac{1}{d^{\frac{1}{n}}}}\mathcal{N}_{K}(\mathfrak{q}_{t+1})^{n-1}(\log(B))^{3}+B^{\frac{1}{d^{\frac{1}{n}}}}\mathcal{N}_{K}(\mathfrak{q}_{t+1})^{n-1}(\log(B))^{2}+(\log(B))^{2}\mathcal{N}_{K}(\mathfrak{q}_{t+1})^{n}\Big) \exp\left(C\frac{\log(B)}{\log(\log(B))}\right).\nonumber
\end{align*}
If $\mathcal{N}_{K}(\mathfrak{q})\lesssim_{K,n,d} B^{\frac{1}{d^{\frac{1}{n}}}}\log(B)$,  it follows that
\begin{equation}
\deg(Z(\mathfrak{q}_{t+1}))\lesssim_{K,n,d}B^{\frac{n}{d^{\frac{1}{n}}}}(\log(B))^{n+3}\exp\left(C\frac{\log(B)}{\log(\log(B))}\right)\lesssim_{K,n,d}B^{\frac{n}{d^{\frac{1}{n}}}}\exp\left(C'\frac{\log(B)}{\log(\log(B))}\right)
\label{bound for the degree Z}
\end{equation}
for some positive constant $C'\lesssim_{K,n,d,\varepsilon}1$. 

Then we define
$$\mathcal{Q}:=\left\{ \mathfrak{q}| \mathfrak{q}^{\ast}:\mathfrak{q}\neq (1)\text{ and } \mathcal{N}_{K}(\mathfrak{q})\leq 2cB^{\frac{1}{d^{\frac{1}{n}}}}\log(B)\right\}.$$

\begin{claim}
There is some $C''\lesssim_{K,n,d}1$such that it holds $|\mathcal{Q}|\leq \exp\left(C''\frac{\log(B)}{\log(\log(B))}\right)$.
\label{size of Q}
\end{claim}
\begin{proof}[Proof of Claim \ref{size of Q}]
It will suffice to bound the number of divisors of $\mathfrak{q}^{\ast}$. By Landau Prime Ideal Theorem or the Riemann Hypothesis over function fields, there is a constant $c_{K}\lesssim_{K,n,d}1$ such that $|\{\mathfrak{p}:\mathcal{N}_{K}(\mathfrak{p})\in I\}|\leq c_{K}\frac{c\log(B)}{\log(c\log(B))}$. Then
$$|\mathcal{Q}|\leq \left|\left\{\mathfrak{q}:\mathfrak{q}|\mathfrak{q}^{\ast}\right\}\right|\leq 2^{|\{\mathfrak{p}:\mathcal{N}_{K}(\mathfrak{p})\in I\}|}\leq \exp\left(c_{K}\frac{c\log(B)}{\log(c\log(B))}\right)\leq \exp\left(C''\frac{\log(B)}{\log(\log(B))}\right),$$
for some $C''\lesssim_{K,n,d}1$.
\end{proof}

Thus, \eqref{bound for the degree Z} and Claim \eqref{size of Q} imply that 
$$\sum_{\mathfrak{q}\in \mathcal{Q}}\deg(Z(\mathfrak{q}))\lesssim_{K,n,d}B^{\frac{n}{d^{\frac{1}{n}}}}\exp\left(C''\frac{\log(B)}{\log(\log(B))}\right)\exp\left(C'\frac{\log(B)}{\log(\log(B))}\right)\lesssim_{K,n,d}B^{\frac{n}{d^{\frac{1}{n}}}}\exp\left(C'''\frac{\log(B)}{\log(\log(B))}\right),$$
for some $C'''\lesssim_{K,n,d}1$. This proves Proposition \ref{construction of a family of cycles}\eqref{construction cycles}. Moreover, by Claim \ref{existence of q} and the construction of $\{Z(\mathfrak{q})\}_{\mathfrak{q}\in \mathcal{Q}}$, it follows that $\{D_{\gamma}\}_{\gamma\in \Gamma}$ and $\{Z(\mathfrak{q})\}_{\mathfrak{q}\in \mathcal{Q}}$ verify the statement of Proposition \ref{construction of a family of cycles}\eqref{cycles cover missing points}.
\end{proof}

\begin{remark}
In Proposition \ref{construction of a family of cycles}\eqref{construction cycles} one could get the better bound $\lesssim_{K,n,d}B^{\frac{n}{d^{\frac{1}{n}}}}\exp(c\frac{(\log(B))^{\frac{1}{2}}}{\log(\log(B))})$ if in the proof of Claim \ref{enough primes3} one bounds the quantity $\prod_{i=1}^{t+1}(1+\mathcal{N}_{K}(\mathfrak{p}_{i})^{-\frac{1}{2}})$ as in \cite[$\S 1.5.5$]{Tenenbaum}. We choose not to do so because when we use it in the proof of an affine variant of the dimension growth conjecture (Proposition \ref{induction step n=2}) this saving will be absorbed in a factor $B^{\varepsilon}$. This remark also applies to the bound in the first summand of the next theorem.
\end{remark}

\begin{theorem}
Let $X\subseteq \mathbb{P}_{K}^{3}$ be a geometrically integral hypersurface of degree $d$, defined over $K$, and let $X_{\text{ns}}$ be the non-singular locus of $X$. Then for all $\nu>0$ there exists a subset of $\lesssim_{K,d}B^{\frac{1}{\sqrt{d}}}\log(B)$ geometrically integral curves $D_{\lambda}\subseteq X$, $\lambda\in \Lambda=\Lambda_{\nu}$, of degree at most $\frac{1}{\nu}$, such that it holds
$$N_{\emph{\text{aff}}}\left(X_{\text{ns}}-\bigcup_{\lambda \in \Lambda}D_{\lambda},\mathcal{O}_{K},B\right)\lesssim_{K,d}\begin{cases}B^{\frac{2}{\sqrt{d}}}\exp\left(c\frac{\log(B)}{\log(\log(B))}\right)+\frac{1}{\nu^{4}}B^{\frac{1}{\sqrt{d}}+\nu}(\log(B))^{2}& \text{ if }K \text{ is a number field},\\ B^{\frac{2}{\sqrt{d}}}\exp\left(c\frac{\log(B)}{\log(\log(B))}\right)+\frac{1}{\nu^{8}}B^{\frac{1}{\sqrt{d}}+\nu}(\log(B))^{2} & \text{ if }K \text{ is a function field},\end{cases}$$ 
for some positive constant $c\lesssim_{K,d}1$.
\label{bound for the non-singular locus outside a subset}
\end{theorem}
\begin{proof}
Let $\{D_{\gamma}\}_{\gamma\in \Gamma}$, and $\{Z(\mathfrak{q})\}_{\mathfrak{q}\in \mathcal{Q}}$ as in Proposition \ref{construction of a family of cycles} for $n=2$. By Proposition \ref{construction of a family of cycles}\eqref{cycles cover missing points}, it holds
$$N_{\text{aff}}\left(X_{\text{ns}}-\bigcup_{\gamma\in \Gamma}D_{\gamma},\mathcal{O}_{K},B\right)\leq \sum_{\mathfrak{q}\in \mathcal{Q}}\deg(Z(\mathfrak{q})).$$

 Hence, by Proposotion \ref{construction of a family of cycles}\eqref{construction cycles}, there is a positive constant $c\lesssim_{K,d}1$ such that it holds
 \begin{equation}
 N_{\text{aff}}\left(X_{\text{ns}}-\bigcup_{\gamma\in \Gamma}D_{\gamma},\mathcal{O}_{K},B\right)\lesssim_{K,d}B^{\frac{2}{\sqrt{d}}}\exp\left(c\frac{\log(B)}{\log(\log(B))}\right).
 \label{bound for the non-singular locus1}
 \end{equation}
 Now, let $\Lambda\subseteq \Gamma$ be the subset of all indexes $\gamma\in \Gamma$ for which $\deg(D_{\gamma})\leq \frac{1}{\nu}$. Then
\begin{equation}
N_{\text{aff}}\left(X_{\text{ns}}-\bigcup_{\lambda\in \Lambda}D_{\lambda},\mathcal{O}_{K},B\right)\leq N_{\text{aff}}\left(X_{\text{ns}}-\bigcup_{\gamma\in \Gamma}D_{\gamma},\mathcal{O}_{K},B\right)+|\Gamma|\max_{\gamma\in \Gamma\backslash \Lambda}N_{\text{aff}}(D_{\gamma},\mathcal{O}_{K},B).
\label{bound for the non-singular locus2}
\end{equation}  
It remains to bound $N_{\text{aff}}(D_{\gamma},\mathcal{O}_{K},B)$ for any $\gamma\in \Gamma\backslash \Lambda$. Now, we apply Theorem \ref{Walsh general affine curves} to each $D_{\gamma}$ with $\gamma\in \Gamma\backslash \Lambda$ to obtain
\begin{equation*}
N_{\text{aff}}(D_{\gamma},\mathcal{O}_{K},B)\lesssim_{K}\begin{cases}\frac{1}{\nu^{4}}B^{\nu}\log(B) & \text{ if }K\text{ is a number field}\\ \frac{1}{\nu^{8}}B^{\nu}\log(B) & \text{ if }K \text{ is a function field}.\end{cases}
\end{equation*}
By Proposition \ref{construction of a family of cycles}\eqref{construction divisors}, $|\Gamma|\lesssim_{K,d}B^{\frac{1}{\sqrt{d}}}\log(B)$, from which we deduce
\begin{equation}
|\Gamma|\max_{\gamma\in \Gamma\backslash \Lambda}N_{\text{aff}}(D_{\gamma},\mathcal{O}_{K},B)\lesssim_{K,d}\begin{cases} \frac{1}{\nu^{4}}B^{\frac{1}{\sqrt{d}}+\nu}(\log(B))^{2} & \text{ if }K\text{ is a number field},\\ \frac{1}{\nu^{8}}B^{\frac{1}{\sqrt{d}}+\nu}{(\log(B))}^{2} & \text{ if }K\text{ is a function field}.\end{cases}
\label{paso del lema2}
\end{equation}

The theorem follows from \eqref{bound for the non-singular locus1}, \eqref{bound for the non-singular locus2}, and \eqref{paso del lema2}.
\end{proof}

\section{Uniform dimension growth conjecture}
\label{Section 7}

In this section we prove Theorem \ref{teorema3} and Theorem \ref{teorema4} of the introduction, which generalize \cite[Theorem 0.3 and Theorem 0.4]{Salberger}, \cite[Theorem 1 and Theorem 4]{Cluckers}, and \cite[Theorem 1.3 and Theorem 4.2]{Vermeulen} to varieties over global fields. The proof follows the strategy given in \cite{Salberger2}. For this, we use an inductive argument: we establish Theorem \ref{teorema3} in the case $X\subseteq \mathbb{A}_{K}^{n}$ with $n=3$ (this is the content of Proposition \ref{induction step n=2}), and afterwards we intersect $X$ with an adequate hyperplane $H$ so that $X\cap H$ is of smaller dimension and  suitable to apply the inductive hypothesis. As in \cite{Salberger2} we deduce Theorem \ref{teorema4} by using a projection argument to reduce it to the case of a hypersurface, and then by taking affine cones and applying Theorem \ref{teorema3}.   

In order to prove Proposition \ref{induction step n=2} we proceed as in \cite[Proposition 4.3.4]{Cluckers}, namely, first we establish that for large degree, the counting function $N_{\text{aff}}(\mathcal{Z}(f),\mathcal{O}_{K},B)$ is at most $\lesssim_{K}d^{e}B$. This is done by using Theorem \ref{ell-venk} to cover $X_{\text{aff}}(\mathcal{O}_{K},B)$ with a hypersurface of small degree and then by bounding the contribution of the rational points coming from the irreducible components of $X\cap Y$ using Theorem \ref{Walsh general affine curves} for the ones of degree at least $2$ and Proposition \ref{contribution1} for the ones of degree $1$. Having proved Proposition \ref{induction step n=2} for large degree, the remaining cases are dealt with Theorem \ref{bound for the non-singular locus outside a subset}.

Now we begin to carry out this program. First we establish the following generalization of \cite[Proposition 4.3.3]{Cluckers}, where the authors give an effective estimate of \cite[Proposition 1, case $D=1$]{Salberger2}, with explicit dependence on the degree of the hypersurface. 

\begin{proposition}
Let $K$ be a global field . There exists a constant $c=c(K,n)$ such that for all $f\in \mathcal{O}_{K}[X_{1},X_{2},X_{3}]$ of degree $d\geq 3$ satisfying that the homogeneous part of higher degree $f_{d}$ is irreducible, and for all finite sets $I$ of curves $C\subseteq \mathbb{A}_{K}^{3}$ of degree $1$ lying on the hypersurface $\mathcal{Z}(f)$ defined by $f$, and all $B\geq 1$, it holds
$$N_{\emph{\text{aff}}}\left(\mathcal{Z}(f)\cap \left(\bigcup_{C\in I}C\right),\mathcal{O}_{K},B\right)\leq cd^{6}B+|I|.$$
\label{contribution1}
\end{proposition}

\begin{proof}
Let us write $I=I_{1}\cup I_{2}$ where 
$$I_{1}:=\{C\in I:N_{\text{aff}}(C,\mathcal{O}_{K},B)\leq 1\},\; I_{2}:=\{C\in I:N_{\text{aff}}(C,\mathcal{O}_{K},B)>1\}.$$
It is clear that 
$$N_{\text{aff}}\left(\mathcal{Z}(f)\cap \left(\bigcup_{C\in I_{1}}C\right),\mathcal{O}_{K},B\right)\leq |I_{1}|.$$
 Now, for any $C\in I_{2}$ there exist $\boldsymbol a=(a_{1},a_{2},a_{3}),\boldsymbol w=(w_{1},w_{2},w_{3})\in \mathcal{O}_{K}^{3}$ with $\boldsymbol a\in [B]_{\mathcal{O}_{K}}^{3}$, and $C(K)=\{\boldsymbol a+\lambda \boldsymbol w:\lambda\in K\}$.
 
\begin{claim}
Let $C$ be the line defined by $\boldsymbol a+\lambda \boldsymbol w$ with $\boldsymbol a\in [B]_{\mathcal{O}_{K}}^{3}$ and $\boldsymbol w\in \mathcal{O}_{K}^{3}$. Then
$$|C(K)\cap [B]_{\mathcal{O}_{K}}^{3}|\lesssim _{K}\frac{B}{H_{K}(\boldsymbol w)}.$$
\label{key bound2}
\end{claim} 
\begin{proof}[Proof of Claim \ref{key bound2}]
While in the case $K=\mathbb{Q}$ or $K=\mathbb{F}_{q}(T)$ the claim is trivial, in general the proof is more involved due to the fact $\mathcal{O}_{K}$ may not be a principal ideal domain. This is why we use the geometry of numbers and the theory of divisors. 

By Proposition \ref{Serre}, we may suppose that $\boldsymbol w$ is $\mathfrak{p}$-primitive for all $\mathfrak{p}$ of $\mathcal{O}_{K}$ with $\mathcal{N}_{K}(\mathfrak{p})>c_{2}$, and $\prod_{v\in M_{K,\text{fin}}}\max_{i}|w_{i}|_{v}\geq c_{3}$. In particular, for all $v\in M_{K,\text{fin}}$, $\max_{i} \text{ord}_{\mathfrak{p}_{v}}(w_{i})\lesssim_{K}1$. Let $S=M_{K,\infty}\cup \{v:\mathcal{N}_{K}(\mathfrak{p}_{v})\leq c_{2}\}$. We let 
$$\mathcal{O}_{K,S}:=\{x\in \mathcal{O}_{K}:|x|_{v}\leq 1\text{ for all }v\notin S\}.$$
If $K$ is a number field, then
\begin{align*}
C(K)\cap [B]_{\mathcal{O}_{K}}^{3} & =\{\boldsymbol a+\lambda\boldsymbol w\in \mathcal{O}_{K}^{3}:\house{a_{i}+\lambda w_{i}}\leq B^{\frac{1}{d_{K}}}\text{ for all }i\}\\ & \subseteq \left\{\boldsymbol a+\lambda \boldsymbol w\in \mathcal{O}_{K}^{3}:\lambda\in \mathcal{O}_{K,S},|\sigma(\lambda)|\leq \frac{2B^{\frac{1}{d_{K}}}}{\max_{i}|\sigma(w_{i})|}\forall \sigma:K\hookrightarrow \mathbb{C},|\lambda|_{v}\leq c_{4}\forall v\in S\backslash M_{K,\infty}\right\},
\end{align*}
for some $c_{4}:=c_{4}(K)$. Now, if $K_{v}$ denotes the completion of $K$ with respect to $v$, we have that $\mathcal{O}_{K,S}$ is a lattice under the usual embedding $\mathcal{O}_{K,S}\hookrightarrow \prod_{v\in S}K_{v}$. Denoting $B_{v}(0,r_{v})$ for the usual disk in $K_{v}$ of center $0$ and radius $r_{v}$, we have
\begin{align*}
|C(K)\cap [B]_{\mathcal{O}_{K}}^{3}|& \leq \left| \mathcal{O}_{K,S}\cap \prod_{v\in M_{K,\infty}}B_{v}\left(0,\frac{2B^{\frac{1}{d_{K}}}}{\max_{i}|\sigma(w_{i})|}\right)\times \prod_{v\in S\backslash M_{K,\infty}}B_{v}(0,c_{4}) \right|\\ & \lesssim_{K} \prod_{\sigma\text{ real} }\frac{2B^{\frac{1}{d_{K}}}}{\max_{i}|\sigma(w_{i})|} \prod_{\sigma\text { complex}}\frac{4B^{\frac{2}{d_{K}}}}{\max_{i}|\sigma(w_{i})|^{2}}
\lesssim_{K}\frac{B}{H_{K}(\boldsymbol w)}.\nonumber
\end{align*} 
If $K$ is a function field over $\mathbb{F}_{q}(T)$, then 
\begin{align*}
C(K)\cap [B]_{\mathcal{O}_{K}}^{3} & =\{\boldsymbol a+\lambda\boldsymbol w\in \mathcal{O}_{K}^{3}:|a_{i}+\lambda w_{i}|_{v_{\infty}}\leq B^{\frac{1}{d_{K}}}\text{ for all }i\}\\ & \subseteq \left\{\boldsymbol a+\lambda\boldsymbol w\in \mathcal{O}_{K}^{3}:\lambda\in\mathcal{O}_{K,S},|\lambda|_{v_{\infty}}\leq \frac{2B^{\frac{1}{d_{K}}}}{\max_{i}|w_{i}|_{v_{\infty}}},|\lambda|_{v}\leq c_{4}\forall v\in S\backslash \{v_{\infty}\} \right\}:=A
\end{align*}
for some constant $c_{4}:=c_{4}(K)$. Now we argue as in \cite[Proposition 2.2]{P2}. Any $\lambda \in A$ satisfies $H_{K}(\lambda)\lesssim_{K}\frac{B}{H_{K}(\boldsymbol w)}$. Since $H_{K}(\lambda)=H_{K}(\lambda^{-1})$, the positive divisor $\sum_{v:\text{ord}_{v}(\lambda)\geq 0}\text{ord}_{v}(\lambda)\cdot v$ has degree $\lesssim_{K}\log_{q}(\frac{B}{H_{K}(\boldsymbol w)})$; by \cite[Proposition 2.2]{P2}, there are at most $\lesssim_{K}\frac{B}{H_{K}(\boldsymbol w)}$ such divisors. Since $\text{ord}_{v}(\lambda)\gtrsim_{K}1$ for all $v\in S\backslash \{v_{\infty}\}$ and $\sum_{v\in M_{K}}\text{ord}_{v}(\lambda)\deg(v)=0$, we conclude that the subset $A$ has at most $\lesssim_{K}\frac{B}{H_{K}(\boldsymbol w)}$ elements. 
\end{proof}
 
Now, for any $\boldsymbol w\in K^{3}$ there are at most $d(d-1)$ lines $C\in I_{2}$ in the direction of $\boldsymbol w$ (this is because each such line intersects a generic hyperplane in $\mathbb{A}^{3}$ in a common point of the hypersurfaces defined by $f$ and its directional derivate in the direction $\boldsymbol w$). 

Since $C\in I_{2}$, it verifies the hypothesis of Claim \ref{key bound2}, then we have $H_{K}(\boldsymbol w)\leq cB$. Moreover, since $C\subseteq \mathcal{Z}(f)$, $f_{d}(\boldsymbol w)=0$. Consider the set $A_{i}:=\{\boldsymbol w\in \mathbb{P}^{2}(K):f_{d}(\boldsymbol w)=0,H_{K}(\boldsymbol w)=i\}$. Then  by Claim \ref{key bound2},
\begin{equation}
N_{\text{aff}}\left(\mathcal{Z}(f)\cap \bigcup_{C\in I}C,\mathcal{O}_{K},B\right)\leq |I_{1}|+(d-1)d\sum_{i=1}^{\left\lfloor cB \right\rfloor}|A_{i}|\frac{cB}{i}.
\label{contribution of curves degree 1}
\end{equation}
Now Corollary \ref{Walsh corollary} implies that $\sum_{i=1}^{k}|A_{i}|\lesssim_{K}d^{4}k^{\frac{2}{d}}$. This, and summation by parts give
\begin{align}
\sum_{i=1}^{\left\lfloor cB\right\rfloor}|A_{i}|\frac{cB}{i} & =\sum_{i=1}^{\left\lfloor cB\right\rfloor}|A_{i}|\frac{cB}{\left\lfloor cB\right\rfloor}+\sum_{i=1}^{\left\lfloor cB\right\rfloor}\left(\sum_{i=1}^{k}|A_{i}|\right)\left(\frac{cB}{k}-\frac{cB}{k+1}\right) \lesssim_{K} d^{4}B^{\frac{2}{d}}+\sum_{i=1}^{\left\lfloor cB\right\rfloor}d^{4}k^{\frac{2}{d}}\frac{cB}{k(k+1)}\lesssim_{K}d^{4}B, \label{summation by parts}
\end{align}
where the last bound is because $d\geq 3$. Replacing \eqref{summation by parts} in \eqref{contribution of curves degree 1} finishes the proof.  
\end{proof}

\begin{proposition}
Let $K$ be a global field of degree $d_{K}$, and let $e=18$ if $K$ is a number field and $e=64$ if $K$ is a function field. Then for all polynomial $f\in \mathcal{O}_{K}[X_{1},X_{2},X_{3}]$ of degree $d$ whose homogeneous part of highest degree is absolutely irreducible, it holds
$$N_{\emph{\text{aff}}}(\mathcal{Z}(f),\mathcal{O}_{K},B)\lesssim_{K}d^{e}B\text{ whenever }d\geq 5.$$
In the case $d=3,4$, for all $\varepsilon>0$ it holds
$$N_{\emph{\text{aff}}}(\mathcal{Z}(f),\mathcal{O}_{K},B)\lesssim_{K,\varepsilon} \begin{cases} B^{1+\varepsilon} & \text{ if }d=4,\\ B^{\frac{2}{\sqrt{3}}+\varepsilon} & \text{ if }d=3.\end{cases}$$
\label{induction step n=2}
\end{proposition}

\begin{proof}
The strategy of the proof follows \cite[Proposition 4.3.4]{Cluckers}, namely, first we establish that for large degree, the counting function $N_{\text{aff}}(\mathcal{Z}(f),\mathcal{O}_{K},B)$ is at most $\lesssim_{K}d^{e}B$. This is done by using Theorem \ref{ell-venk} to cover $X_{\text{aff}}(\mathcal{O}_{K},B)$ with a hypersurface of small degree. Then one bounds the contribution of the rational points coming from the irreducible components of $X\cap Y$. Now, the contribution of those irreducible components of degree $1$ is bounded with Proposition \ref{contribution1}. By B\'ezout's theorem, the degree of these curves can be large as $\lesssim_{K,d} B^{\frac{1}{\sqrt{d}}}$, hence one can not apply directly Theorem \ref{Walsh general affine curves} to deal with the components of $X\cap Y$ of degree $\gtrsim_{K}\log(B)$. This technical obstruction will be dealt by assuming that the degree of $X$ is large enough.

Having proved Proposition \ref{induction step n=2} for large degree, the remaining cases are dealt with Theorem \ref{bound for the non-singular locus outside a subset}. Then, after enlarging the implicit constant, one concludes $N_{\text{aff}}(\mathcal{Z}(f),\mathcal{O}_{K},B)\lesssim_{K}d^{e}B$. We remark that, while Theorem \ref{bound for the non-singular locus outside a subset} gives the bound $N_{\text{aff}}(\mathcal{Z}(f),\mathcal{O}_{K},B)\lesssim_{K,d}B$ for $d\geq 5$, it gives a double exponential dependence on $d$; it is for this reason that we  do not use it to prove Proposition \ref{induction step n=2} for all $d\geq 5$. 

For any prime for which $f_{d} \text{ mod }\mathfrak{p}$ is absolutely irreducible the reduction of $f$ modulo $\mathfrak{p}$ is absolutely irreducible. Hence $b(f)\leq b(f_{d})$. Using Lemma \ref{b(f)} and Theorem \ref{ell-venk}, for any $B\geq 1$ we find a polynomial $g\in \mathcal{O}_{K}[X_{1},X_{2},X_{3}]$ of degree 
\begin{equation}
\lesssim_{K}\begin{cases}d^{\frac{7}{2}}B^{\frac{1}{\sqrt{d}}} & \text{ if }K \text{ is a number field},\\ d^{7}B^{\frac{1}{\sqrt{d}}} & \text{ if }K \text{ is a function field},\end{cases}
\label{auxiliar polynomial bound}
\end{equation}
not divisible by $f$ and vanishing on all $\mathcal{Z}(f)_{\text{aff}}(\mathcal{O}_{K},B)$. Let $\mathcal{Z}(f,g)$ be the intersection of $f=0$ and $g=0$ (here the intersection is being considered with its reduced structure). Let $\mathcal{C}$ be the subset of irreducible components of $\mathcal{Z}(f,g)$. By B\'ezout's theorem \cite[Example 8.4.6]{Fulton},
\begin{equation}
|\mathcal{C}|\leq \sum_{C\in \mathcal{C}}\deg(C)\lesssim_{K}\begin{cases}d^{\frac{9}{2}}B^{\frac{1}{\sqrt{d}}} & \text{ if }K \text{ is a number field},\\ d^{8}B^{\frac{1}{\sqrt{d}}} & \text{ if }K \text{ is a function field}.\end{cases}
\label{bezout}
\end{equation}
By Proposition \ref{contribution1} and \eqref{bezout}, the contribution of the irreducible components of $\mathcal{Z}(f,g)$ of degree $1$ is
\begin{equation}
\lesssim_{K}d^{6}B+|\mathcal{C}|\lesssim_{K}\begin{cases} d^{6}B & \text{ if }K \text{ is a number field},\\ d^{8}B & \text{ if }K \text{ is a function field}.\end{cases}
\label{contribution of the irreducible linear components}
\end{equation} 

Now, let us suppose that $C_{1},\ldots, C_{k}$ are the irreducible components of $\mathcal{Z}(f,g)$ of degree greater than $1$, arranged in such a way that $\deg(C_{i})\leq \log(B)$ for all $1\leq i\leq m$, and $\deg(C_{i})>\log(B)$ for all $i>m$. Let us denote $\delta_{i}:=\deg(C_{i})$. By Theorem \ref{Walsh general affine curves}, for all $i$ it holds

\begin{equation}
N_{\text{aff}}(C_{i},\mathcal{O}_{K},B)\lesssim_{K}\begin{cases}\delta_{i}^{3}B^{\frac{1}{\delta_{i}}}(\log(B)+\delta_{i}) & \text{ if }K\text{ is a number field}, \\ \delta_{i}^{7}B^{\frac{1}{\delta_{i}}}(\log(B)+\delta_{i}) & \text{ if } K \text{ is a function field}. \end{cases} 
\label{contribution of non linear degree}
\end{equation} 

\begin{claim}
For all $1\leq i\leq m$ it holds $N_{\emph{\text{aff}}}(C_{i},\mathcal{O}_{K},B)\lesssim_{K}B^{\frac{1}{2}}(\log(B)+1)$.
\label{contribution of small non linear degree}
\label{claim 7.4}
\end{claim}
\begin{proof}[Proof of Claim \ref{claim 7.4}]Let us suppose that $K$ is a number field.
Let us set $\psi(\delta):=\delta^{4}B^{\frac{1}{\delta}}$. Then $\log(\psi(\delta))=4\log(\delta)+\frac{\log(B)}{\delta}=\log(B)(4\log_{B}(\delta)+\frac{1}{\delta})$. Since the function $\delta_{i}\mapsto 4\log_{B}(\delta_{i})+\frac{1}{\delta_{i}}$ is decreasing in $(0,\frac{\log(B)}{4})$ and increasing in $(\frac{\log(B)}{4},+\infty)$, the maximum value of $\psi(\delta)$ in $[2,\log(B)]$ is 
\begin{equation}
\max\{\psi(2),\psi(\log(B))\}=\max\{2^{4}B^{\frac{1}{2}},(\log(B))^{4}B^{\frac{1}{\log(B)}}\}\lesssim B^{\frac{1}{2}}.
\label{maximum inequality}
\end{equation}
The inequalities \eqref{contribution of non linear degree}, \eqref{maximum inequality}, and the trivial bound $\delta\geq 1$ give Claim \ref{contribution of small non linear degree} for number fields. The case of function fields is analogous.
\end{proof}

By Claim \ref{contribution of small non linear degree} and \eqref{bezout} we have
\begin{equation}
\sum_{i=1}^{m}N_{\text{aff}}(C_{i},\mathcal{O}_{K},B)\lesssim_{K}B^{\frac{1}{2}}(\log(B)+1)m\lesssim_{K}\begin{cases} d^{\frac{9}{2}}B & \text{ if }K\text{ is a number field}, \\ d^{8}B & \text{ if }K \text{ is a function field}.\end{cases}
\end{equation}

On the other hand, if $\delta>\log(B)$ then $B^{\frac{1}{\delta}}$ is bounded, thus \eqref{bezout} and \eqref{contribution of non linear degree} imply
\begin{equation}
\sum_{i=m+1}^{k}N_{\text{aff}}(C_{i},\mathcal{O}_{K},B)\lesssim_{K}\begin{cases}\sum_{i=m+1}^{k}\delta_{i}^{4}\lesssim_{K}\left(\sum_{i=m+1}^{k}\delta_{i}\right)^{4}\lesssim_{K} d^{18}B^{\frac{4}{\sqrt{d}}} &  \text{ if }K \text{ is a number field},\\ \sum_{i=m+1}^{k}\delta_{i}^{8}\lesssim_{K}\left( \sum_{i=m+1}^{k}\delta_{i} \right)^{8}\lesssim_{K}d^{64}B^{\frac{8}{\sqrt{d}}} & \text{ if }K \text{ is a function field}.\end{cases}
\label{contribution of higher linear degree}
\end{equation}

Combining \eqref{auxiliar polynomial bound}, Claim \ref{contribution of the irreducible linear components}, \eqref{contribution of small non linear degree}, and \eqref{contribution of higher linear degree}, Proposition \ref{induction step n=2} follows for $d\geq 16$ if $K$ is a number field and for $d\geq 64$ if $K$ is a function field.

For the remaining values of $d$, by Theorem \ref{bound for the non-singular locus outside a subset} with $\nu=\frac{1}{2\sqrt{d}}$, there is a constant $c\lesssim_{K,d}1$ such that for all $B\geq 1$ there exists a subset of $\lesssim_{K,d}B^{\frac{1}{\sqrt{d}}}\log(B)$ geometrically integral curves $D_{\lambda}\subseteq X$, $\lambda \in \Lambda$, of degree at most $\lesssim_{K}\sqrt{d}$ such that for all $\varepsilon>0$ it  verifies 
\begin{equation*}
N_{\text{aff}}\left( X_{\text{ns}}-\bigcup_{\lambda \in \Lambda}D_{\lambda},\mathcal{O}_{K},B \right)\lesssim_{K,d,\varepsilon}B^{\frac{2}{\sqrt{d}}+\varepsilon}+B^{\frac{2}{\sqrt{d}}},
\end{equation*}
where we used that $\exp\left(\frac{\log(B)}{\log(\log(B))}\right)\lesssim_{d,K,\varepsilon}B^{\varepsilon}$. Hence, for all $\varepsilon>0$ we have
$$N_{\text{aff}}\left( X_{\text{ns}}-\bigcup_{\lambda \in \Lambda}D_{\lambda},\mathcal{O}_{K},B \right)\lesssim_{K,d,\varepsilon}\begin{cases} B & \text{ if }d\geq 5, \\ B^{1+\varepsilon} & \text{ if }d=4,\\ B^{\frac{2}{\sqrt{3}}+\varepsilon} & \text{ if }d=3.\end{cases}$$

Thus, it remains to bound the contribution of the points lying in any of the curves $D_{\lambda}$, and the points lying in the complement of $X\backslash X_{\text{ns}}$. By Proposition \ref{contribution1}, those curves $D_{\lambda} \subseteq \Lambda$ of degree $1$ contribute at most $\lesssim_{K,d}B$ points, while by Theorem \ref{Walsh general affine curves} those curves $D_{\lambda} \subseteq \Lambda$ of degree at least $2$ contribute at most $\lesssim_{K,d}B^{\frac{1}{2}+\varepsilon}$ points. On the other hand, by B\'ezout's theorem, $X\backslash X_{\text{ns}}$ is a union of irreducible curves the sum of whose degrees is bounded by a constant. Applying Theorem \ref{Walsh corollary} to those curves of degree at least $2$, and Proposition \ref{contribution1} to those of degree $1$ yields that the rational points coming from $X\backslash X_{\text{ns}}$ is $\lesssim_{K,d}B$. 
\end{proof}

\begin{lemma}
Let $K$ be a global field. Let $n\geq 3$ and let $X\subseteq \mathbb{P}_{K}^{n}$ be a geometrically integral hypersurface of degree $d$.  Let $\kappa=(n+1)(d^{2}-1)$ if $K$ is a number field, and $\kappa=12(n+1)d^{7}$ if $K$ is a function field. Then there exists a non-zero form $F\in \mathcal{O}_{K}[Y_{0},\ldots ,Y_{n}]$ of degree at most $\kappa$ such that $F(A)=0$ whenever the hyperplane section $H_{A}\cap X$ is not geometrically integral, where $A\in (\mathbb{P}^{n})^{\ast}$ and $H_{A}\subseteq \mathbb{P}^{n}$ denotes the hyperplane corresponding to the linear form $A$.
\label{Adequate form}
\end{lemma}

\begin{proof}
The proof is similar to the one given in \cite[Proposition 4.3.7]{Cluckers} for $K=\mathbb{Q}$ and in \cite[Lemma 4.6]{Vermeulen} for $K=\mathbb{F}_{q}(T)$.
\end{proof}

Now we are in condition to prove the dimension growth conjecture for varieties over global fields, namely Theorem \ref{teorema3} and Theorem \ref{teorema4} for the introduction.

\begin{proof}[Proof of Theorem \ref{teorema3}]
Let $n\geq 3$ and $X\subseteq \mathbb{A}_{K}^{n}$ be a geometrically integral hypersurface of degree $d\geq 3$ defined by a polynomial $f\in \mathcal{O}_{K}[X_{1},\ldots ,X_{n}]$ with absolutely irreducible highest degree part. We proceed by induction on $n$, where the base case $n=3$ is Proposition \ref{induction step n=2}.

Now, let us assume that $n>3$ and that the theorem holds for all lower $n$. Let $f_{d}$ be the homogeneous part of highest degree of $f$; since it is absolutely irreducible, it defines a geometrically integral hypersurface in $\mathbb{P}_{K}^{n-1}$. Applying the Combinatorial Nullstellensatz (see \cite[Theorem 1.2]{Alon}) to $F$ of Lemma \ref{Adequate form} there exists $A=(a_{1},\ldots ,a_{n})$ such that the hyperplane section $\{f_{d}=0\}\cap \{\sum_{i}a_{i}X_{i}=0\}$ is geometrically integral of degree $d$, with all $a_{i}\in \mathcal{O}_{\mathbbm{k}}$ having $H_{\mathbbm{k}}(a_{i})\leq n(d^{2}-1)$ or $H_{\mathbbm{k}}(a_{i})\leq 12nd^{7}$ if $K$ is a number field or a function field, respectively. Let $\gamma=2$ or $\gamma=7$ if $K$ is a number field or a function field, respectively. 
Since there exists a constant $c(n)\lesssim_{n}1$ such that for any $(x_{1},\ldots ,x_{n})\in [B]_{\mathcal{O}_{K}}^{n}$ we have $a_{1}x_{1}+\cdots +a_{n}x_{n}\in [c(n)d^{\gamma}B]_{\mathcal{O}_{K}}$, it follows that
$$N_{\text{aff}}(\mathcal{Z}(f),\mathcal{O}_{K},B)\leq \sum_{k\in [c(n)d^{\gamma}B]_{\mathcal{O}_{K}}}N_{\text{aff}}\left(\{f=0\}\cap \left\{ \sum_{i}a_{i}X_{i}=k \right\},\mathcal{O}_{K},B\right).$$

For each $k$, the variety $\{f=0\}\cap \{\sum_{i}a_{i}X_{i}=k\}$ is a hypersurface in the affine plane $\{\sum_{i}a_{i}X_{i}=k\}$, thus after a change of variables it is described by a polynomial $g\in \mathcal{O}_{K}[X_{1},\ldots ,X_{n-1}]$ whose homogeneous part of highest degree is absolutely irreducible by the construction of $A$. The proof follows from the induction hypothesis and the fact that $|[B]_{\mathcal{O}_{K}}|\sim_{K}B$.
\end{proof}

\begin{proof}[Proof of Theorem \ref{teorema4}]
We make a change of variables as in the proof of Theorem \ref{Walsh general curves} to reduce Theorem \ref{teorema4}  to the case of a hypersurface. Hence let $n\geq 3$ and consider an  irreducible polynomial $f\in \mathcal{O}_{K}[X_{0},\ldots, X_{n}]$ of degree $d\geq 3$. Then we take the affine cone $C(f)$ defined by $\mathcal{Z}(f)$; it is an affine hypersurface in $\mathbb{A}_{K}^{n+1}$. By Proposition \ref{Serre} or by Remark \ref{explicacion sobre la caja}, for any point in $\mathcal{Z}(f)(K,B)$, there is a lift in $\mathbb{A}_{K}^{n+1}$ which lies in $\mathcal{Z}(C(f))_{\text{aff}}(\mathcal{O}_{K},[c_{1}B]_{\mathcal{O}_{K}}^{n+1})$. Thus 
$$N(\mathcal{Z}(f),K,B)\leq N_{\text{aff}}(C(f),\mathcal{O}_{K},c_{1}B).$$
If $f$ is absolutely irreducible, applying Theorem \ref{teorema3} finishes the proof.  Otherwise, by Remark \ref{abs irred assumption}, there exists a homogeneous polynomial $g\in \mathcal{O}_{K}[X_{0},\ldots ,X_{n+1}]$ of degree  $d$ not divisible by $f$ and vanishing on all $K$-rational point of $C(f)$. Then
\begin{equation}
N_{\text{aff}}(C(f),\mathcal{O}_{K},c_{1}B)\leq N_{\text{aff}}(C(f)\cap \mathcal{Z}(g),\mathcal{O}_{K},c_{1}B).
\label{final bound}
\end{equation}
Now, it holds that $C(f)\cap \mathcal{Z}(g)$ is a variety of dimension $n-2$. By a hyperplane section argument as in \cite[Page 91]{BrowningHeathBrown}, it holds
\begin{equation}
N_{\text{aff}}(C(f)\cap \mathcal{Z}(g),\mathcal{O}_{K},B)\lesssim_{K,n} dB^{n-2}.
\label{final bound2}
\end{equation}
Then inequalities \eqref{final bound} and \eqref{final bound2} yield the conclusion of Theorem \ref{teorema4} for integral projective varieties which are not geometrically irreducible.
\end{proof}


\bibliography{paper}

\begin{thebibliography}{10}

\bibitem{Alon}
N.~Alon.
\newblock Combinatorial {N}ullstellensatz.
\newblock volume~8, pages 7--29. 1999.
\newblock Recent trends in combinatorics (M\'{a}trah\'{a}za, 1995).

\bibitem{Bombieri}
E.~Bombieri and W.~Gubler.
\newblock {\em Heights in {D}iophantine geometry}, volume~4 of {\em New
  Mathematical Monographs}.
\newblock Cambridge University Press, Cambridge, 2006.

\bibitem{Bombieri0}
E.~Bombieri and J.~Pila.
\newblock The number of integral points on arcs and ovals.
\newblock {\em Duke Math. J.}, 59(2):337--357, 1989.

\bibitem{Bombieri2}
E.~Bombieri and J.~Vaaler.
\newblock On {S}iegel's lemma.
\newblock {\em Invent. Math.}, 73(1):11--32, 1983.

\bibitem{Broberg}
N.~Broberg.
\newblock A note on a paper by {R}. {H}eath-{B}rown: ``{T}he density of
  rational points on curves and surfaces'' [{A}nn. of {M}ath. (2) {\bf 155}
  (2002), no. 2, 553--595; mr1906595].
\newblock {\em J. Reine Angew. Math.}, 571:159--178, 2004.

\bibitem{BrobSal}
N.~Broberg and P.~Salberger.
\newblock Counting rational points on threefolds.
\newblock In {\em Arithmetic of higher-dimensional algebraic varieties ({P}alo
  {A}lto, {CA}, 2002)}, volume 226 of {\em Progr. Math.}, pages 105--120.
  Birkh\"{a}user Boston, Boston, MA, 2004.

\bibitem{Browningbook}
T.~D. Browning.
\newblock {\em Quantitative arithmetic of projective varieties}, volume 277 of
  {\em Progress in Mathematics}.
\newblock Birkh\"{a}user Verlag, Basel, 2009.

\bibitem{BrowningHeathBrown}
T.~D. Browning and D.~R. Heath-Brown.
\newblock Counting rational points on hypersurfaces.
\newblock {\em J. Reine Angew. Math.}, 584:83--115, 2005.

\bibitem{Salberger2}
T.~D. Browning, D.~R. Heath-Brown, and P.~Salberger.
\newblock Counting rational points on algebraic varieties.
\newblock {\em Duke Math. J.}, 132(3):545--578, 2006.

\bibitem{Cafure}
A.~Cafure and G.~Matera.
\newblock Improved explicit estimates on the number of solutions of equations
  over a finite field.
\newblock {\em Finite Fields Appl.}, 12(2):155--185, 2006.

\bibitem{Cluckers}
W.~Castryck, R.~Cluckers, P.~Dittmann, and K.~H. Nguyen.
\newblock The dimension growth conjecture, polynomial in the degree and without
  logarithmic factors.
\newblock {\em Algebra Number Theory}, 14(8):2261--2294, 2020.

\bibitem{Chen}
H.~Chen.
\newblock Explicit uniform estimation of rational points {I}. {E}stimation of
  heights.
\newblock {\em J. Reine Angew. Math.}, 668:59--88, 2012.

\bibitem{Chen2}
H.~Chen.
\newblock Explicit uniform estimation of rational points {II}. {H}ypersurface
  coverings.
\newblock {\em J. Reine Angew. Math.}, 668:89--108, 2012.

\bibitem{Cluckers0}
R.~Cluckers, A.~Forey, and F.~Loeser.
\newblock Uniform {Y}omdin-{G}romov parametrizations and points of bounded
  height in valued fields.
\newblock {\em Algebra Number Theory}, 14(6):1423--1456, 2020.

\bibitem{Cohen}
S.~D. Cohen.
\newblock The distribution of {G}alois groups and {H}ilbert's irreducibility
  theorem.
\newblock {\em Proc. London Math. Soc. (3)}, 43(2):227--250, 1981.

\bibitem{Cox}
D.~A. Cox, J.~Little, and D.~O'Shea.
\newblock {\em Ideals, varieties, and algorithms}.
\newblock Undergraduate Texts in Mathematics. Springer, Cham, fourth edition,
  2015.
\newblock An introduction to computational algebraic geometry and commutative
  algebra.

\bibitem{Ellenberg}
J.~Ellenberg and A.~Venkatesh.
\newblock On uniform bounds for rational points on nonrational curves.
\newblock {\em Int. Math. Res. Not.}, (35):2163--2181, 2005.

\bibitem{Fulton}
W.~Fulton.
\newblock {\em Intersection theory}, volume~2 of {\em Ergebnisse der Mathematik
  und ihrer Grenzgebiete (3) [Results in Mathematics and Related Areas (3)]}.
\newblock Springer-Verlag, Berlin, 1984.

\bibitem{Gao}
S.~Gao.
\newblock Factoring multivariate polynomials via partial differential
  equations.
\newblock {\em Math. Comp.}, 72(242):801--822, 2003.

\bibitem{Gortz}
U.~G\"{o}rtz and T.~Wedhorn.
\newblock {\em Algebraic geometry {I}}.
\newblock Advanced Lectures in Mathematics. Vieweg + Teubner, Wiesbaden, 2010.
\newblock Schemes with examples and exercises.

\bibitem{Grenie}
L.~Greni\'{e} and G.~Molteni.
\newblock Explicit versions of the prime ideal theorem for {D}edekind zeta
  functions under {GRH}.
\newblock {\em Math. Comp.}, 85(298):889--906, 2016.

\bibitem{Hartshorne}
R.~Hartshorne.
\newblock {\em Algebraic geometry}.
\newblock Springer-Verlag, New York-Heidelberg, 1977.
\newblock Graduate Texts in Mathematics, No. 52.

\bibitem{Heath-Brown83}
D.~R. Heath-Brown.
\newblock Cubic forms in ten variables.
\newblock {\em Proc. London Math. Soc. (3)}, 47(2):225--257, 1983.

\bibitem{Heath-Brown}
D.~R. Heath-Brown.
\newblock The density of rational points on curves and surfaces.
\newblock {\em Ann. of Math. (2)}, 155(2):553--595, 2002.

\bibitem{Helfgott}
H.~A. Helfgott and A.~Venkatesh.
\newblock How small must ill-distributed sets be?
\newblock In {\em Analytic number theory}, pages 224--234. Cambridge Univ.
  Press, Cambridge, 2009.

\bibitem{Hindry}
M.~Hindry and J.~H. Silverman.
\newblock {\em Diophantine geometry}, volume 201 of {\em Graduate Texts in
  Mathematics}.
\newblock Springer-Verlag, New York, 2000.
\newblock An introduction.

\bibitem{Iwaniec}
H.~Iwaniec and E.~Kowalski.
\newblock {\em Analytic number theory}, volume~53 of {\em American Mathematical
  Society Colloquium Publications}.
\newblock American Mathematical Society, Providence, RI, 2004.

\bibitem{Kaltofen}
E.~Kaltofen.
\newblock Effective {N}oether irreducibility forms and applications.
\newblock volume~50, pages 274--295. 1995.
\newblock 23rd Symposium on the Theory of Computing (New Orleans, LA, 1991).

\bibitem{Liu}
C.~{Liu}.
\newblock {Determinant method and the pseudo-effective threshold}.
\newblock {\em arXiv e-prints}, page arXiv:1910.00306, Oct. 2019.

\bibitem{Liu2}
C.~{Liu}.
\newblock {On the global determinant method}.
\newblock {\em arXiv e-prints}, page arXiv:2101.07453, Jan. 2021.

\bibitem{P2}
J.~M. {Menconi}, M.~{Paredes}, and R.~{Sasyk}.
\newblock {The inverse sieve problem for algebraic varieties over global
  fields}.
\newblock {\em To appear in Rev. Mat. Iberoam.}, page DOI: 10.4171/rmi/1261,
  Dec. 2020.

\bibitem{MR0282975}
D.~Mumford.
\newblock Varieties defined by quadratic equations.
\newblock In {\em Questions on {A}lgebraic {V}arieties ({C}.{I}.{M}.{E}., {III}
  {C}iclo, {V}arenna, 1969)}, pages 29--100. Edizioni Cremonese, Rome, 1970.

\bibitem{Red}
D.~Mumford.
\newblock {\em The red book of varieties and schemes}, volume 1358 of {\em
  Lecture Notes in Mathematics}.
\newblock Springer-Verlag, Berlin, 1988.

\bibitem{Pilaafin}
J.~Pila.
\newblock Density of integral and rational points on varieties.
\newblock Number 228, pages 4, 183--187. 1995.
\newblock Columbia University Number Theory Seminar (New York, 1992).

\bibitem{Rosen}
M.~Rosen.
\newblock {\em Number theory in function fields}, volume 210 of {\em Graduate
  Texts in Mathematics}.
\newblock Springer-Verlag, New York, 2002.

\bibitem{Ruppert}
W.~Ruppert.
\newblock Reduzibilit\"{a}t ebener {K}urven.
\newblock {\em J. Reine Angew. Math.}, 369:167--191, 1986.

\bibitem{Salberger1}
P.~Salberger.
\newblock On the density of rational and integral points on algebraic
  varieties.
\newblock {\em J. Reine Angew. Math.}, 606:123--147, 2007.

\bibitem{Salberger}
P.~Salberger.
\newblock Counting rational point on projective varieties.
\newblock {\em Preprint}, 2009.

\bibitem{Sedunova}
A.~Sedunova.
\newblock On the {B}ombieri-{P}ila method over function fields.
\newblock {\em Acta Arith.}, 181(4):321--331, 2017.

\bibitem{Serre1983}
J.~Serre.
\newblock {\em Autour du th{\'e}or{\`e}me de Mordell-Weil}.
\newblock Number v. 1 in Autour du th{\'e}or{\`e}me de Mordell-Weil.
  Universit{\'e} Pierre et Marie Curie, 1983.

\bibitem{Serre}
J.~P. Serre.
\newblock {\em Lectures on the {M}ordell-{W}eil theorem}.
\newblock Aspects of Mathematics, E15. Friedr. Vieweg \& Sohn, Braunschweig,
  1989.
\newblock Translated from the French and edited by Martin Brown from notes by
  Michel Waldschmidt.

\bibitem{Serre2}
J.~P. Serre.
\newblock {\em Topics in {G}alois theory}, volume~1 of {\em Research Notes in
  Mathematics}.
\newblock Jones and Bartlett Publishers, Boston, MA, 1992.
\newblock Lecture notes prepared by Henri Damon [Henri Darmon], With a foreword
  by Darmon and the author.

\bibitem{Tenenbaum}
G.~Tenenbaum.
\newblock {\em Introduction to analytic and probabilistic number theory},
  volume 163 of {\em Graduate Studies in Mathematics}.
\newblock American Mathematical Society, Providence, RI, third edition, 2015.
\newblock Translated from the 2008 French edition by Patrick D. F. Ion.

\bibitem{Thunder}
J.~L. Thunder.
\newblock Siegel's lemma for function fields.
\newblock {\em Michigan Math. J.}, 42(1):147--162, 1995.

\bibitem{Vermeulen}
F.~{Vermeulen}.
\newblock {Points of bounded height on curves and the dimension growth
  conjecture over $\mathbb{F}_q[t]$}.
\newblock {\em arXiv e-prints}, page arXiv:2003.10988, Mar. 2020.

\bibitem{Walsh3}
M.~N. Walsh.
\newblock Bounded rational points on curves.
\newblock {\em Int. Math. Res. Not. IMRN}, (14):5644--5658, 2015.

\end{thebibliography}
\bibliographystyle{abbrv}

\end{document}